\documentclass[a4paper]{article}

\usepackage{amsmath,amsthm,amssymb}
\usepackage{todonotes}
\usepackage[all]{xy}

\newtheorem{theorem}{Theorem}[subsection]
\newtheorem{proposition}[theorem]{Proposition}
\newtheorem{lemma}[theorem]{Lemma}
\newtheorem{corollary}[theorem]{Corollary}
\newtheorem{remark}[theorem]{Remark}
\theoremstyle{definition}
\newtheorem{example}[theorem]{Example}
\newtheorem{definition}[theorem]{Definition}
\newtheorem{notation}[theorem]{Notation}

\numberwithin{equation}{section}

\newcommand{\pushoutcorner}[1][dr]{\save*!/#1+1.2pc/#1:(1,-1)@^{|-}\restore}
\newcommand{\pullbackcorner}[1][dr]{\save*!/#1-1.2pc/#1:(-1,1)@^{|-}\restore}

\newcommand{\vrt}[1]{V(#1)}
\newcommand{\fhom}{\mathrm{\underline{Hom}}}

\newcommand{\cat}[1]{\mathbb{#1}}
\newcommand{\catc}{\cat{C}}
\newcommand{\set}{\mathbf{Set}}
\newcommand{\asm}{\mathbf{Asm}}

\newcommand{\btwo}{\mathbf{2}}
\newcommand{\bthree}{\mathbf{3}}
\newcommand{\vtcl}{\mathrm{vert}}
\newcommand{\smcatc}{\mathcal{C}}
\newcommand{\opcat}[1]{{#1}^\mathrm{op}}
\newcommand{\fmly}{\operatorname{Fam}}
\newcommand{\fmlyc}{\fmly(\catc)}
\newcommand{\cats}{\mathbf{Cat}}
\newcommand{\smcat}[1]{\mathcal{#1}}
\newcommand{\cod}{\operatorname{cod}}
\newcommand{\dom}{\operatorname{dom}}
\newcommand{\intv}{\mathbb{I}}

\newcommand{\xalg}[1]{{#1}\operatorname{-Alg}}
\newcommand{\bicart}{\operatorname{Bicart}}

\newcommand{\names}{\mathbb{A}}
\newcommand{\perma}{\operatorname{Perm}(\names)}

\newcommand{\ffcat}{\mathbf{FF}}
\newcommand{\awfscat}{\mathbf{AWFS}}
\newcommand{\lawfscat}{\mathbf{LAWFS}}
\newcommand{\rawfscat}{\mathbf{RAWFS}}

\newcommand{\powset}{\mathcal{P}}
\newcommand{\powfin}{\powset_\mathrm{fin}}
\newcommand{\nat}{\mathbb{N}}

\newcommand{\sub}{\mathsf{01Sub}}

\newcommand{\dmmod}{\mathsf{dM}}

\newcommand{\yoneda}{\mathbf{y}}

\title{Lifting Problems in Grothendieck Fibrations}
\author{Andrew W Swan}

\begin{document}

\maketitle

\begin{abstract}
  Many interesting classes of maps from homotopical algebra can be
  characterised as those maps with the right lifting property against
  certain sets of maps (such classes are sometimes referred to as
  \emph{cofibrantly generated}). In a more sophisticated notion due to
  Garner (referred to as \emph{algebraically cofibrantly generated})
  the set of maps is replaced with a diagram over a small category.

  We give a yet more general definition where the set or diagram of
  maps is replaced with a vertical map in a Grothendieck fibration. In
  addition to an interesting new view of the existing examples above,
  we get new notions, such as computable lifting problems in presheaf
  assemblies, and internal lifting problems in a topos.

  We show that under reasonable conditions one can define a notion of
  universal lifting problem and carry out step-one of Garner's small
  object argument. We give explicit descriptions of what the general
  construction looks like in some examples.
\end{abstract}

\tableofcontents

\section{Introduction}
\label{sec:introduction}

We first recall some standard notions in homotopical algebra. See
e.g. \cite[Chapters 11 and 12]{riehlcht} for more details.

Given two maps $m \colon U \rightarrow V$ and
$f \colon X \rightarrow Y$ in a category $\mathbb{C}$, we say that $m$
has the \emph{left lifting property} with respect to $f$ and $f$ has
the \emph{right lifting property} with respect to $m$ if for every
commutative square, as in the solid lines below (which we refer to as
a \emph{lifting problem}), there is a \emph{diagonal filler}, which is
the dotted line below, making two commutative triangles.
\begin{equation*}
  \begin{gathered}
    \xymatrix{ U \ar[d]_m \ar[r] & X \ar[d]^f \\
      V \ar@{.>}[ur] \ar[r] & Y}
  \end{gathered}
\end{equation*}

\begin{definition}
  A \emph{weak factorisation system} (wfs) is two classes of maps
  $\mathcal{L}$ and $\mathcal{R}$, which are closed under retracts
  such that every element of $\mathcal{L}$ has the left lifting
  property against every element of $\mathcal{R}$, and any map factors
  as an element of $\mathcal{L}$ followed by an element of
  $\mathcal{R}$.
\end{definition}

\begin{definition}
  We say a weak factorisation system $(\mathcal{L}, \mathcal{R})$ is
  \emph{cofibrantly generated} if there is some set $I$ of morphisms in
  $\catc$ such that $\mathcal{R} = I^\pitchfork$.
\end{definition}

A well known result due to Quillen, the \emph{small object argument}
shows that in categories satisfying certain conditions any set $I$
cofibrantly generates a wfs.

We note however, that the definition of cofibrantly generated still
makes sense in the absence of any weak factorisation
system. Specifically, we say that a class $\mathcal{R}$ is
\emph{cofibrantly generated} by a set $I$ if $R = I^\pitchfork$. This
can still be a useful thing to do as it can give us an easy way to
give concise definitions of classes of maps. In this paper we will
focus on classes of maps considered in the semantics of homotopy type
theory, although the techniques developed may be more widely
applicable.

A generalisation of cofibrantly generated was developed by
Garner in \cite{garnersmallobject}. The definition was again
originally stated for wfs's, or more precisely a more structured
notion called \emph{algebraic weak factorisation system} (awfs)
developed by Grandis and Tholen in \cite{grandistholennwfs}
(originally referred to as \emph{natural weak factorisation
  system}). We again note however, that the definition still makes
sense for arbitrary classes of maps.

\begin{definition}[Garner]
  \label{def:alglift}
  Let $\catc$ be a category. Let $\smcat{A}$ and $\smcat{B}$ be small
  categories and let $F \colon \smcat{A} \rightarrow \catc^\btwo$ and
  $G \colon \smcat{B} \rightarrow \catc^\btwo$ be functors. We say $F$
  has the \emph{left lifting property} against $G$ and $G$ has the
  \emph{right lifting property} against $F$ if the following
  holds. For all objects $a$ of $\smcat{A}$ and $b$ of $\smcat{B}$,
  and for all lifting problems of $F(a)$ against $G(b)$, we have a
  choice of filler. Furthermore these fillers satisfy a
  \emph{uniformity condition} which states that for all morphisms
  $\sigma \colon a \rightarrow a'$ in $\smcat{A}$, all $\tau \colon
  b \rightarrow b'$ in $\smcat{B}$ and all commutative cubes with
  the square $F(\sigma)$ on the left and $G(\tau)$ on the right, the
  resulting ``diagonal square'' formed by the fillers commutes.
\end{definition}

\begin{definition}[Garner]
  Let $\catc$ be a category. We say a class of maps $\mathcal{R}$ is
  \emph{algebraically cofibrantly generated} if there is a small
  category $\smcat{I}$ and a functor
  $J \colon \smcat{I} \rightarrow \catc^\btwo$ such that $f$ belongs
  to the class if and only if it has the right lifting property
  against $J$.
\end{definition}

It is natural to also consider the notion of cofibrantly generated
category as defined below.
\begin{definition}
  Let $\catc$ be a category. Let $\smcat{I}$ be a small category and
  $J \colon \smcat{I} \rightarrow \catc^\btwo$. Write $J^\pitchfork$
  for the category defined as follows. An object of $J^\pitchfork$ is
  a morphism $f$ of $\catc$ together with a uniform choice of
  diagonal fillers of $J$ against $f$. A morphism is a commutative
  square in $\catc$ which is compatible with the fillers.

  We say a category $\cat{D}$ and functor
  $U \colon \cat{D} \rightarrow \catc^\btwo$ is \emph{cofibrantly
    generated} if they are isomorphic to the forgetful functor
  $J^\pitchfork \rightarrow \catc^\btwo$ for some $J \colon \smcat{I}
  \rightarrow \catc^\btwo$.
\end{definition}

Garner developed an improved version of Quillen's small object
argument, referred to as the \emph{algebraic small object
  argument}. As a consequence of this result, under certain
assumptions, cofibrantly generated categories are isomorphic to
categories of algebras over a monad on $\catc^\btwo$. As part of the
proof, known as \emph{step-one} of the algebraic small object argument,
Garner considered a weaker notion of \emph{left half of an algebraic
  weak factorisation system}. This is already sufficient to show that
cofibrantly generated categories are isomorphic to categories of
algebras over a pointed endofunctor on $\catc^\btwo$.

We will develop a further generalisation of lifting problem, in which
the set $I \subseteq \catc^\btwo$, or diagram
$J \colon \smcat{I} \rightarrow \catc^\btwo$ is replaced with a
vertical map in a Grothendieck fibration.
We will also describe some
interesting examples that aren't included under existing definitions
of cofibrantly generated. Our main two examples will be category
indexed presheaf assemblies and codomain fibrations. The first of
these is a variant of Garner's definition of algebraically cofibrantly
generated applied to presheaf categories, but where the choice of
fillers for lifting problems must satisfy the additional requirement
of being uniformly computable. The second allows us to formalise a
notion of lifting problem internal to a topos (or more generally any
locally cartesian closed category). This is will allow us to better
understand certain ideas considered by Van den Berg and Frumin in
\cite{vdbergfrumin} and by Pitts and Orton in
\cite{pittsortoncubtopos}. In both of these cases we note that the
underlying category is not necessarily cocomplete, which is a necessary
condition required for Garner's small object argument, even for the
relatively simple step-one part.

We will then develop some constructions that can be carried out for
the general definition under reasonable assumptions (which do not
require the underlying category to be complete or cocomplete). The
first is universal lifting problem, in which the collection of all
lifting problems between two families of maps corresponds to one
single lifting problem in the total category of the fibration, which is uniquely
determined by a universal property. We then define an abstract version
of step-one of the small object argument that exhibits cofibrantly
generated categories as categories of algebras over pointed
endofunctors.

We will also show how to generalise the definition of cofibrantly
generated algebraic weak factorisation system to our setting. Although
we won't see any new examples of such awfs's in this paper, we will
see some smaller results in this direction. We will show how the
existing results by Garner relate to this definition. We will also
develop a sufficient criterion for the existence of cofibrantly
generated awfs's  in terms of the existence of choices of
initial objects for certain categories of algebras. This will then be
used in a future paper alongside a new generalisation of dependent
$W$-types to construct new examples of awfs's.

We will also show how to generalise the Leibniz construction to a
fibration, and a generalised notion of lifting problem due to
Sattler. Both may have useful application when applying these ideas to
the semantics of homotopy type theory (work in progress by the author
suggests that the latter can be used to better understand and
generalise the implementation of higher inductive types in cubical
sets).

We draw attention the following ideas that motivated various aspects
of this work that are good to bear in mind when reading.

\paragraph{The similarities and differences between two
  generalisations of cubical sets}
Cubical sets (from \cite{coquandcubicaltt}), are inspired by ideas in
homotopical algebra, and in particular the simplicial set model
\cite{voevodskykapulkinlumsdainess}. However, over time two distinct
approaches have arisen.

In \cite{gambinosattlerpi}, Gambino and
Sattler view cubical sets as a special case of general constructions
in homotopical algebra, including Garner's notion of lifting problem
and small object argument
and the Leibniz construction. We will refer to this as the ``algebraic
approach.''
On the other hand Orton and Pitts view
cubical sets as a special case of a construction in the internal logic
of a topos. We will refer to this as the ``internal logic approach.''

The two approaches are closely connected, and as Gambino and Sattler
show, there is a large amount of overlap. However, there are also a
number of curious differences. In using Garner's definition of
cofibrantly generated, the algebraic approach makes essential use of
the notion of ``small category'' and thereby to an external notion of
set that does not correspond to anything internal to the category we
are working with. We see this again with the requirement that the
category we work with is cocomplete. In this sense the internal logic
approach, which does not require cocompleteness would appear to be
more flexible. However, there are also several ways in which the
algebraic is more general. For example, much of the general theory in
the algebraic approach can be carried out without the category of
study being locally cartesian closed: if we need to talk about maps
from an object $X$ to an object $Y$, we just use the set of
morphisms. This would be impossible in the internal logic approach,
where we need local cartesian closedness in order to talk about maps
from $X$ to $Y$ internally.

We will see that what the two approaches have in common is that both
can be seen as studying part of a locally small bifibration. The
differences between the two approaches can be explained in terms of
differences between the bifibrations that we are working over. The
external notion of set and small category in the algebraic approach
arises from the base category of the fibration of ``category indexed
families.'' For the internal logic approach, we are (implicitly)
working over a codomain fibration, where the base category is the just
the category we are studying. The essential use of infinite colimits
in the algebraic approach arise from the opcartesian maps in the
fibration we use, which are left Kan extensions. On the other hand the
opcartesian maps over a codomain fibration are just given by
composition, so we don't see any kind of infinite colimit over a small
category (even internally).

\paragraph{Applications to presheaf assemblies}
One of the main applications that the author hopes to use in future work is to
understanding lifting problems within presheaf assemblies (categories
of presheaves constructed internally in categories of assemblies). In
particular working with presheaf assemblies should not involve minor
but tedious modifications of existing arguments. Instead the same
proofs and definitions should apply to both the existing definitions
of cubical sets and to new realizability variants. We will in fact see
two different approaches that could be used in the future in presheaf
assemblies. The first is simply to build on the existing approach of Pitts
and Orton\footnote{Technically their paper does not apply to presheaf
  assemblies because of the use of subobject classifiers, but the
  author expects this can be worked around without
  difficulty. Certainly the main results in this paper do not require
  subobject classifier.}. The
other is a more ``hands on'' approach based on Garner's definition of
algebraic lifting problem. In this approach a functor
$J \colon \mathcal{I} \rightarrow \catc^\btwo$ (like in definition
\ref{def:alglift}) is annotated with extra computational
information. A filler is then required to be both uniform in the
algebraic sense, while also uniformly computable. This has quite a
different character to the internal approach, and might also be
useful for some applications.

A key point to make about presheaf assemblies is that they are not
cocomplete, even for just countable colimits.

\paragraph{A generalisation of Garner's algebraic lifting problem that
  does not mention small categories}
On seeing the essential role that small categories play in Garner's
definition, one might expect that something analogous, such as
internal categories appear in any generalisation. Indeed in early
drafts of this paper, that was the approach taken. However, somewhat
surprisingly, our general definition will not involve any notion of
small category or internal category, and yet still includes Garner's
definition as a special case by choosing the Grothendieck fibration
appropriately (using so called \emph{category indexed family
  fibrations}).

\paragraph{The distinction between fibred and strongly fibred awfs's}
When working over a fibration it's natural to require that functors
preserve cartesian maps. We will see that in this case there are in
fact two different senses in which an awfs may preserve cartesian maps,
that we refer to as \emph{fibred} and \emph{strongly fibred}.

For a general fibration the difference is clear to see. The definition
of strongly fibred always involves pullbacks, whereas fibred only
mentions cartesian maps from the fibration we are working over.

Over a codomain fibration, it is possible to confuse the two notions
since both can roughly be described as ``stability under pullback.''
However, the distinction is still important to make. We will see that
over codomain fibrations, cofibrantly generated lawfs's are always
fibred (and so are cofibrantly generated awfs's when they exist). In
general, cofibrantly generated lawfs's need not be strongly fibred
(and if an awfs fails to be strongly fibred this is already the case
for the lawfs at step-one), but we will see a useful special case
where they are.

Throughout we pay careful attention to when we know that functors are
fibred and when we don't.

\subsection*{Acknowledgements}
\label{sec:acknowlegdements}

I'm grateful to Benno van den Berg for many helpful comments and
corrections.

\section{Lifting Problems in a Fibration}
\label{sec:definition-examples}

\subsection{Review of Grothendieck Fibrations}
\label{sec:revi-groth-fibr}

We recall the definition of Grothendieck fibration. See
e.g. \cite[Chapter 1]{jacobs} for a more detailed introduction.

\begin{definition}
  \label{def:vertical}
  Let $p \colon \cat{E} \rightarrow \cat{B}$ be a functor. A morphism
  $f \colon A \rightarrow B$ in $\cat{E}$ is \emph{vertical} if
  $p(A) = p(B)$ and $p(f)$ is the identity morphism. We will often
  refer to a vertical map over $I \in \cat{B}$ as a \emph{family of
    maps over $I$}.
\end{definition}

\begin{definition}
  Let $p \colon \cat{E} \rightarrow \cat{B}$ be a functor and let $I$
  be an object of $\cat{B}$. We define the \emph{fibre category over
    $I$}, $\cat{E}_I$ to consist of objects $X$ of $\cat{E}$ such that
  $p(X) = I$, and morphisms (vertical) maps $f$ such that $p(f) =
  1_I$.
\end{definition}

\begin{definition}
  \label{def:cartesian}
  Let $p \colon \cat{E} \rightarrow \cat{B}$ be a functor and
  $u \colon I \rightarrow J$ a morphism in $\cat{B}$. A morphism
  $f \colon X \rightarrow Y$ is \emph{cartesian over u} if $p(f) = u$
  and for every $g \colon Z \rightarrow Y$ in $\cat{E}$ with
  $p(g) = u \circ w$ for some $w \colon p(Z) \rightarrow I$ there is a
  unique $h \colon Z \rightarrow X$ in $\cat{E}$ above $w$ with
  $f \circ h = g$.
\end{definition}

We think of cartesian maps as ``substitutions.''

\begin{definition}
  \label{def:clovfib}
  A \emph{cloven Grothendieck fibration} is a functor
  $p \colon \cat{E} \rightarrow \cat{B}$ together with a choice of
  cartesian lifts. That is, for each $X \in \cat{E}$ and
  $\sigma \colon I \rightarrow p(Y)$ in $\cat{B}$ we are given an
  object $\sigma^\ast Y$ in $\cat{E}$ and a cartesian morphism
  $\overline{\sigma}(X) \colon u^\ast Y \rightarrow Y$ over
  $\sigma$. We call the choice of lifts a \emph{cleavage} for $p$. We
  will write $\overline{\sigma}(X)$ just as $\overline{\sigma}$ if $X$
  is clear from context. We will refer to $\sigma^\ast(Y)$ as the
  \emph{reindexing} of $Y$.

  We refer to $\cat{B}$ as the \emph{base} and to $\cat{E}$ as the
  \emph{total category} of the fibration.
\end{definition}

\begin{remark}
  The term \emph{cloven} refers to the fact that we require a choice
  of cartesian lifts, not just that there exists at least one
  cartesian lift in each case. In this paper we will only consider
  cloven fibrations. In fact the same applies for all categorical
  structure in this paper. For instance, when we say that a category
  has certain limits, we really mean that we are given a choice of
  limit for each diagram of that shape.
\end{remark}

\begin{proposition}
  For each $\sigma \colon I \rightarrow J$, reindexing along $\sigma$
  defines a functor $\cat{E}_J \rightarrow \cat{E}_I$.
\end{proposition}

An important concept that we will use throughout the paper is that of
fibrations of vertical maps, which we define below.
\begin{definition}
  Let $p \colon \cat{E} \rightarrow \cat{B}$ be a functor. We define
  the \emph{category of vertical maps}, $\vrt{\cat{E}}$, to be the
  full subcategory of $\cat{E}^\btwo$ consisting of vertical maps.
\end{definition}

\begin{proposition}
  The functors $p \circ \dom$ and $p \circ \cod$ from $\vrt{\cat{E}}$
  to $\cat{B}$ are equal.
\end{proposition}

\begin{notation}
  We will sometimes write the functor $p \circ \cod = p \circ \dom$ as
  $p$, when it is clear to do so from context.
\end{notation}

\begin{proposition}
  Let $p \colon \cat{E} \rightarrow \cat{B}$ be a fibration. Then a
  morphism in $\vrt{\cat{E}}$ is cartesian over $\cat{B}$ if and only
  if it is levelwise cartesian over $p$. The functor $p \circ \cod
  \colon \vrt{\cat{E}} \rightarrow \cat{B}$ is also a fibration.
\end{proposition}

\begin{proof}
  It is easy to check that if a morphism is levelwise cartesian then
  it is cartesian in $\vrt{\cat{E}}$. We can hence give $p \circ \cod$
  the structure of a cloven fibration, by applying the cleavage of $p$
  levelwise. But then any cartesian map in $\vrt{\cat{E}}$ must be
  isomorphic to such a morphism, and so must be levelwise cartesian in
  $\cat{E}$.
\end{proof}

\begin{remark}
  While working with Grothendieck fibrations, we will sometimes refer
  to the reference by Johnstone \cite{theelephant}. In these cases the
  results are actually for indexed categories. However, indexed
  categories and Grothendieck fibrations are equivalent and one may
  readily switch between them. See \cite[Section B1.3]{theelephant}
  for details.
\end{remark}

The main examples of Grothendieck fibrations that we will consider are
the following.

\begin{example}
  \label{ex:setindexfmly}
  Let $\catc$ be any category.
  We define the fibration of \emph{set
    indexed families of $\catc$}, $p \colon \fmlyc \rightarrow \set$ as
  follows. Objects of $\fmlyc$ are set indexed families of objects of
  $\catc$, $(C_i)_{i \in I}$ (where $I$ is an indexing set and $C_i$ are
  objects of $\catc$).
  
  Vertical maps $F$ over a set $I$ consist of set indexed
  families of morphisms in $\catc$.
  
  Cartesian maps are maps in $\fmlyc$ that are levelwise isomorphisms,
  and reindexing is just reindexing in the usual sense.
\end{example}

\begin{example}
  \label{ex:catindexfmly}
  The set indexed family fibration $\fmlyc \rightarrow \set$ extends to
  a fibration $\fmlyc \rightarrow \cats$. In this case we define
  $\fmly$ to consist of pairs $\langle \smcat{A}, X \rangle$ where
  $\smcat{A}$ is a small category and $X$ is a functor $\smcat{A}
  \rightarrow \catc$.

  A vertical map from $X \colon \smcat{A} \rightarrow \catc$ to $Y
  \colon \smcat{A} \rightarrow \catc$ is a natural transformation
  between the functors, and reindexing is given by composition.
\end{example}

\begin{example}
  Again let $\catc$ be any category. A map is cartesian over the
  codomain functor $\cod \colon \catc^\btwo \rightarrow \catc$ if and
  only if it is a pullback in $\catc$. Hence $\cod$ is a fibration if
  and only if $\catc$ has pullbacks.
\end{example}

\subsection{Definition of Lifting Problem}
\label{sec:defin-lift-probl}

We now give our general definition of lifting problem (which we will
actually refer to as \emph{family of lifting problems}) and fillers, as
well as a few equivalent definitions.

In the below, let $p \colon \cat{E} \rightarrow \cat{B}$ be a
fibration.

\begin{definition}
  Let $m \colon U \rightarrow V$ and $f \colon X \rightarrow Y$ be
  vertical maps over objects $I$ and $J$ of $\cat{B}$ respectively. A
  \emph{family of lifting problems from $m$ to $f$} consists of an
  object $K$ of $\cat{B}$, morphisms $\sigma \colon K \rightarrow I$,
  $\tau \colon K \rightarrow J$, together with morphisms making the
  top and bottom of a commutative square in $\cat{E}_K$ as below.
  \begin{equation}
    \label{eq:3}
    \begin{gathered}
    \xymatrix{ \sigma^\ast(U) \ar[r] \ar[d]_{\sigma^\ast(m)} & \tau^\ast(X)
      \ar[d]^{\tau^\ast(f)} \\
        \sigma^\ast(V) \ar[r] & \tau^\ast(Y) }      
    \end{gathered}
  \end{equation}
  We will refer to such commutative squares as families of lifting
  problems from $m$ to $f$ over $K$.
\end{definition}

\begin{definition}
  A \emph{solution}, or \emph{choice of diagonal fillers} for a family
  of lifting problems is a map
  $j \colon \sigma^\ast(V) \rightarrow \tau^\ast(X)$ in $\cat{E}_K$, making
  \eqref{eq:3} into two commutative triangles:
  \begin{equation*}
    \begin{gathered}
    \xymatrix{ \sigma^\ast(U) \ar[r] \ar[d]_{\sigma^\ast(m)} & \tau^\ast(X)
      \ar[d]^{\tau^\ast(f)} \\
        \sigma^\ast(V) \ar[r] \ar[ur]|j & \tau^\ast(Y) }      
    \end{gathered}
  \end{equation*}
  (Note that by considering the image of the diagram in $\cat{B}$ we
  see $j$ is necessarily a vertical map.)
\end{definition}

\begin{example}
  Over a trivial fibration $\cat{C} \rightarrow 1$, a family of
  lifting problems is just a lifting problem in $\cat{C}$ and a
  solution is a diagonal filler.
\end{example}

We will often use one of the reformulations of lifting property
below. Each is easy to check using the cartesianness of
$\overline{\tau}(f)$ and the characterisation of cartesian maps in
$\vrt{\cat{E}}$ as levelwise cartesian maps in $\cat{E}$.

\begin{proposition}
  Families of lifting problems from $m$ to $f$ correspond precisely to
  an object $K$ of $\cat{B}$, morphisms
  $\sigma \colon K \rightarrow I$, $\tau \colon K \rightarrow J$,
  together with a vertical map
  $\sigma^\ast(m) \rightarrow \tau^\ast(f)$ in $\vrt{\cat{E}}_K$.
\end{proposition}

\begin{proposition}
  Families of lifting problems from $m$ to $f$ correspond precisely to
  an object $K$ of $\cat{B}$, morphisms
  $\sigma \colon K \rightarrow I$, $\tau \colon K \rightarrow J$,
  together with morphisms over $\tau$ making the top and bottom of a
  commutative square in $\cat{E}$ as below.
  \begin{equation}
    \label{eq:21}
    \begin{gathered}
    \xymatrix{ \sigma^\ast(U) \ar[r] \ar[d]_{\sigma^\ast(m)} & X
      \ar[d]^{f} \\
        \sigma^\ast(V) \ar[r] & Y }
    \end{gathered}
  \end{equation}

  Moreover, solutions of \eqref{eq:3} correspond precisely to
  maps $j \colon \sigma^\ast(V) \rightarrow X$ making \eqref{eq:21}
  into two commutative triangles.
  
  (Note that by considering the image of the diagram in $\cat{B}$ we
  see $j$ necessarily lies over $\tau$.)
\end{proposition}

\begin{proposition}
  Families of lifting problems from $m$ to $f$ correspond precisely to
  an object $K$ of $\cat{B}$, morphisms
  $\sigma \colon K \rightarrow I$, $\tau \colon K \rightarrow J$,
  together with a morphism $\sigma^\ast(m) \rightarrow f$ in
  $\vrt{\cat{E}}$ lying over $\tau$.
\end{proposition}

\begin{proposition}
  Families of lifting problems from $m$ to $f$ correspond precisely to
  isomorphism classes of spans $m \leftarrow \cdot \rightarrow f$ in
  $\vrt{\cat{E}}$ where the map $\cdot \rightarrow m$ is cartesian.
\end{proposition}

\section{Universal Lifting Problems}
\label{sec:univ-lift-probl}

\subsection{Review of Hom Objects and Locally Small Fibrations}
\label{sec:review-hom-objects}

We recall the definitions of hom object and locally small
fibration. We essentially follow Jacobs' presentation in \cite[Section
9.5]{jacobs}, although we will change the terminology a little for
convenience. We will also make some basic observations that will be
used later.

\begin{definition}
  \label{def:homrealdef}
  For any objects $X$ and $Y$ in $\cat{E}$ with $p(X) = p(Y) = I$, we
  say \emph{the hom object from $X$ to $Y$ exists} if the functor from
  $\opcat{(\cat{B} / I)}$ to $\set$ sending $u \colon J \rightarrow I$
  to $\cat{E}(u^\ast(X), u^\ast(Y))$ is representable. We write the
  representing object as $h_0 \colon \fhom_I(X, Y) \rightarrow I$, and
  $h_1$ for the canonical map $h_0^\ast (X) \rightarrow h_0^\ast(Y)$
  corresponding to the identity on $h_0$. We will refer to
  $(\fhom_I(X, Y), h_0, h_1)$ as \emph{the hom object from $X$ to $Y$}.
\end{definition}

Note that for each $X$ and $Y$, the hom object from $X$ to $Y$ is
unique up to isomorphism when it exists. This justifies referring to
it as ``the'' hom object.

\begin{proposition}
  \label{prop:homiasunivspan}
  To specify a hom object from $X$ to $Y$ over $I$ is to specify a
  span from $X$ to $Y$, $X \leftarrow Z \rightarrow Y$ where the left
  map $Z \rightarrow X$ is cartesian, which is universal, in the sense
  that for every span $X \leftarrow Z' \rightarrow Y$ where
  $Z' \rightarrow X$ is cartesian, there is a unique map
  $t \colon Z' \rightarrow Z$ in the commutative diagram below.
  \begin{equation*}
    \xymatrix{ X & & Y \\
      & Z \ar[ul] \ar[ur] & \\
      & Z' \ar@/^/[uul] \ar@/_/[uur] \ar@{.>}[u]^t &}
  \end{equation*}
\end{proposition}

\begin{notation}
  We will sometimes write the object $Z$ appearing in proposition
  \ref{prop:homiasunivspan} as $\fhom_I(X, Y)$ when it is clear to do
  so from context.
\end{notation}

\begin{definition}
  A fibration $p \colon \cat{E} \rightarrow \cat{B}$ is \emph{locally
    small} if we have a choice of hom object
  $(\fhom_I(X, Y), h_0, h_1)$ for all $I \in \cat{B}$ and all
  $X, Y \in \cat{E}_I$.
\end{definition}

\begin{proposition}
  \label{prop:homasunivspan}
  Let $p \colon \cat{E} \rightarrow \cat{B}$ be a fibration, and
  suppose that $\cat{B}$ has all finite limits. The $p$
  is locally small if and only if the following holds. For all
  $I, J \in \cat{B}$, all $X \in \cat{E}_I$ and for all
  $Y \in \cat{E}_J$ there is a choice of object
  $\fhom(X, Y) \in \cat{B}$ together with maps
  $\sigma \colon \fhom(X, Y) \rightarrow I$,
  $\tau \colon \fhom(X, Y) \rightarrow J$ and
  $h \colon \sigma^\ast(X) \rightarrow Y$ over $\tau$, which are
  universal in the following sense. For any $K$, together with maps
  $\sigma' \colon K \rightarrow X$, $\tau' \colon K \rightarrow Y$ and
  $h' \colon \sigma'^\ast(X) \rightarrow Y$ over $\tau'$, there is a
  unique map $t \colon K \rightarrow \fhom(X, Y)$ such that $t$ and
  the canonical map $\sigma'^\ast(X) \rightarrow \sigma^\ast(X)$ over
  $t$ make the following diagrams commute.
  
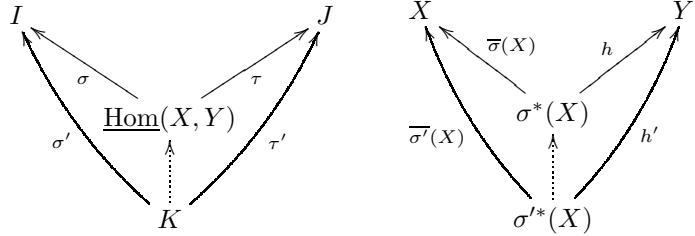
\begin{equation*}
    \begin{gathered}
      \xymatrix{ I & & J \\
        & \fhom(X, Y) \ar[ur]_\tau \ar[ul]^\sigma & \\
        & K \ar@{.>}[u] \ar@/^/[uul]^{\sigma'} \ar@/_/[uur]_{\tau'} &}
    \end{gathered}
    \qquad
    \begin{gathered}
      \xymatrix{ X & & Y \\
        & \sigma^\ast(X) \ar[ur]^h \ar[ul]_{\overline{\sigma}(X)} & \\
        & \sigma'^\ast(X) \ar@{.>}[u]
        \ar@/^/[uul]^{\overline{\sigma'}(X)}
        \ar@/_/[uur]_{h'} &}
    \end{gathered}
  \end{equation*}
\end{proposition}

\begin{proof}
  Given $X$ over $I$ and $Y$ over $J$, we define $\fhom(X, Y)$ as
  follows. Write $\pi_0$ and $\pi_1$ for the projections from $I
  \times J$. We take $\fhom(X, Y)$ to be $\fhom_{I \times J}(X, Y)$.

  Conversely, given $X$ and $Y$ both in $\cat{E}_I$, note that we have
  a map $\fhom(X, Y) \rightarrow I \times I$ given by $\sigma$ and
  $\tau$ above. We define $\fhom_I(X, Y)$ to be the pullback of
  $\fhom(X, Y)$ along the diagonal map $I \rightarrow I \times I$.

  The remaining details are left as an exercise for the reader (see
  \cite[Exercise 9.5.2]{jacobs}, and also see \cite[Lemma
  A.2]{johnstonett}, although the formulation of hom object used there
  is a little different to one we use).
\end{proof}

\begin{notation}
  We will follow the convention that whenever we write in the
  subscript object from $\cat{B}$, as in $\fhom_I(X, Y)$, we are using
  definition \ref{def:homrealdef}. Whenever we drop the subscript as
  in $\fhom(X, Y)$ we are following the alternative definition from
  proposition \ref{prop:homasunivspan}
\end{notation}

\begin{notation}
  We will sometimes write the $\sigma^\ast(X)$ from proposition
  \ref{prop:homasunivspan} as $\fhom(X, Y)$ when it is clear to do so
  from context.
\end{notation}

\begin{remark}
  Just as for lifting problems, we can give an alternative formulation
  of $\fhom(X, Y)$, where instead of specifying a map
  $h \colon \sigma^\ast(X) \rightarrow Y$ over $\tau$, we can specify
  a vertical map $h' \colon \sigma^\ast(X) \rightarrow \tau^\ast(Y)$.
\end{remark}

We will use the following lemmas later.

\begin{lemma}
  \label{lem:vertarrhom}
  Suppose that $p \colon \cat{E} \rightarrow \cat{B}$ is locally small
  and $\cat{B}$ has pullbacks. Then
  $\vrt{\cat{E}} \rightarrow \cat{B}$ is also locally small.
\end{lemma}

\begin{proof}
  Let $m, f$ be objects of $\vrt{\cat{E}}$. Then we can view
  them as vertical maps $m \colon U \rightarrow V$ and $f \colon X
  \rightarrow Y$. We take $\fhom_I(m, f)$ to be the pullback below.
  \begin{equation*}
    \begin{gathered}
      \xymatrix{ \fhom_I(m, f) \pullbackcorner \ar[r] \ar[d] &
        \fhom_I(U, X) \ar[d] \\
        \fhom_I(V, Y) \ar[r] & \fhom_I(U, Y)}
    \end{gathered}
  \end{equation*}
  It is straightforward to verify that this does give a hom object.
\end{proof}

\begin{lemma}
  \label{lem:homprojpb}
  Let $X$, $Y$ and $Y'$ be objects of $\cat{E}$ over $I$, $J$ and $J'$
  respectively. Let $f \colon Y' \rightarrow Y$ be cartesian. Then the
  following square is a pullback, where the map $\fhom(X, Y')
  \rightarrow \fhom(X, Y)$ is the canonical map corresponding to
  composition with $f$.
  \begin{equation*}
    \xymatrix{\fhom(X, Y') \ar[r]^{\tau'} \ar[d] & J' \ar[d]^{p(f)} \\
      \fhom(X, Y) \ar[r]_\tau & J }
  \end{equation*}
\end{lemma}

\begin{proof}
  Suppose we are given a diagram in $\cat{B}$ of the following form.
  \begin{equation}
    \label{eq:22}
    \begin{gathered}
    \xymatrix{ K \ar@/^/[drr]^\alpha \ar@/_/[ddr]_\beta & & \\
      & \fhom(X, Y') \ar[r]^{\tau'} \ar[d] & J' \ar[d]^{p(f)} \\
      & \fhom(X, Y) \ar[r]_\tau & J }      
    \end{gathered}
  \end{equation}
  This gives us the following diagram in $\cat{E}$.
  \begin{equation*}
    \xymatrix{ X & & Y \\
      & \sigma^\ast(X) \ar[ur] \ar[ul] & \\
      & \beta^\ast \sigma^\ast(X) \ar[u] \ar@/^/[uul] \ar@/_/[uur] &}
  \end{equation*}
  Since $f$ is cartesian, there is a unique map $\beta^\ast
  \sigma^\ast(X) \rightarrow Y'$ over $\beta$ allowing us to extend
  the diagram to the one below.
  \begin{equation*}
    \xymatrix{ X & & Y \\
      & \sigma^\ast(X) \ar[ur] \ar[ul] & Y' \ar[u]_f \\
      & \beta^\ast \sigma^\ast(X) \ar[u] \ar@/^/[uul] \ar[ur] &}
  \end{equation*}
  But now this gives us a map $K \rightarrow \fhom(X, Y')$ extending
  \eqref{eq:22}, which is unique by a similar argument, and so the
  square is a pullback, as required.
\end{proof}

\subsection{Definition and Existence of Universal Lifting Problems}
\label{sec:defin-exist-univ}

Universal lifting problem is an important concept that allows us to
view the class of all lifting problems as a single uniquely determined
ordinary lifting problem in the total category of the fibration.

\begin{definition}
  Let $m$ and $f$ be vertical maps over $I$ and $J$, as before.
  A \emph{coherent choice of solutions} consists of a choice of solution
  for each family of lifting problems, satisfying the following
  condition. Suppose that we are given the diagram below
  \begin{equation}
    \label{eq:4}
    \begin{gathered}
      \xymatrix{ & & I \\
        K' \ar[r]_k \ar@/^/[urr]^{f'} \ar@/_/[drr]_{g'} & K \ar[ur]_f
        \ar[dr]^g & \\
        & & J }
    \end{gathered}
  \end{equation}
  and suppose that we are given a family of lifting problems over
  $K$.

  Then we also have a family of lifting problems over $K'$ given by
  applying $k^\ast$ to the family of lifting problems over $K$. The
  choice of solutions must satisfy that the choice of solution to the
  lifting problems over $K'$ is given by applying $k^\ast$ to the
  choice of solution to the lifting problems over $K$.
\end{definition}

\begin{definition}
  Let $\cat{B}$ have finite limits, let $p$ be a locally small
  fibration and let $m$ and $f$ be vertical morphisms over $I$ and $J$
  respectively. We define a family of lifting problems denoted the
  \emph{universal family of lifting problems}, or just \emph{universal
    lifting problem}. Take $K$ to be
  $\fhom(m, f)$ (the hom object in $\vrt{\cat{E}}$, as defined in
  proposition \ref{prop:homasunivspan}) and take commutative square to
  be the morphism $h_1 \colon \sigma^\ast (m) \rightarrow \tau^\ast (f)$
  in $\vrt{\cat{E}}$ given together with $\fhom(m, f)$.
\end{definition}

\begin{lemma}
  \label{lem:univlplem}
  For any family of lifting problems over $K$, there is a unique map
  $t \colon K \rightarrow \fhom(m, f)$ such that the
  family of lifting problems is given by applying $t^\ast$ to the
  universal family.
\end{lemma}

\begin{proof}
  This follows directly from the characterisation of locally small in
  proposition \ref{prop:homasunivspan}.
\end{proof}

\begin{proposition}
  Suppose that a fibration $p \colon \cat{E} \rightarrow \cat{B}$ is
  locally small and $\cat{B}$ has all finite limits. Then the
  universal family of lifting problems from $m$ to $f$ exists for any
  vertical maps $m$ and $f$.
\end{proposition}

\begin{proof}
  By lemma \ref{lem:vertarrhom}.
\end{proof}

Finally, we give a simple, yet powerful result that allows us to
characterise the coherent existence of a filler for every lifting
problem as a single filler for the universal lifting problem.
\begin{proposition}
  \label{prop:univlp}
  Let $\cat{B}$ have finite limits, let $p$ be a locally small
  fibration and let $m$ and $f$ be vertical morphisms. Then the
  following are equivalent.
  \begin{enumerate}
  \item Every family of lifting problems from $m$ to $f$ has a
    solution.
  \item The universal family of lifting problems from $m$ to $f$ has a
    solution.
  \item There is a coherent choice of solutions to all families of
    lifting problems from $m$ to $f$.
  \end{enumerate}
\end{proposition}

\begin{proof}
  Follows easily from the definition of universal family of lifting
  problems.
\end{proof}

\begin{definition}
  If $m$ and $f$ are families of maps over $I$ and $J$ respectively
  and one of the equivalent conditions in proposition
  \ref{prop:univlp} holds, then we say \emph{$m$ has the fibred left
    lifting property against $f$} and \emph{$f$ has the fibred right
    lifting property against $m$}.
\end{definition}

\section{Step-one of the Small Object Argument}

The overall aim of this section is theorem \ref{thm:step1} where we
will produce a more general version of step-one of Garner's small
object argument as appears in \cite{garnersmallobject} (see also
\cite[Chapter 12]{riehlcht}). This will produce for each vertical map
$m$ a left half of an awfs $(L_1, R_1)$, whose right maps are families
of maps with the fibred right lifting property against $m$ (and
moreover $R_1$-algebra structures correspond precisely to solutions to
the universal lifting problem from $m$ to $f$). For this
we will need the Grothendieck fibration to have the additional
structure of a bifibration. As in the previous section, we also assume
that the fibration is locally small.

A motivation for doing this is that the main role that awfs's play in
the cubical set model of type theory is not in constructing Kan
fibrations, but in giving an elegant definition of maps with the
structure of a fibration in terms of algebras over a monad. However,
this can already be done with an lawfs, except that fibrations are
algebras over a pointed endofunctor rather than a monad. Using this
idea it is possible to give a simple way to combine a split
comprehension category with an lawfs to obtain a new split
comprehension category, and the cubical set model of type theory can
be viewed as an instance of this construction\footnote{This will
  appear as part of a future paper by the author.}.

\subsection{A Review of Bifibrations}
\label{sec:bifibrations}

We review some basic definitions and observations on bifibrations. See
e.g. \cite[Section 9.1]{jacobs} for more details.

\begin{definition}
  Let $p \colon \cat{E} \rightarrow \cat{B}$ be a functor. A morphism
  $X \rightarrow Y$ in $\cat{E}$ is \emph{opcartesian} if the
  corresponding map $Y \rightarrow X$ is cartesian in $\opcat{p} \colon
  \opcat{\cat{E}} \rightarrow \opcat{\cat{B}}$. We say $p$ is an
  \emph{opfibration} if $\opcat{p} \colon
  \opcat{\cat{E}} \rightarrow \opcat{\cat{B}}$ is a fibration. We say
  $p$ is a \emph{bifibration} if it is both a fibration and an
  opfibration.  
\end{definition}

\begin{proposition}
  Let $p \colon \cat{E} \rightarrow \cat{B}$ be a cloven
  fibration. Then opcartesian maps over $u \colon I \rightarrow J$
  correspond precisely to left adjoints to $u^\ast \colon \cat{E}_J
  \rightarrow \cat{E}_I$ (which we will write as $\coprod_u \dashv
  u^\ast$).
\end{proposition}

\begin{proof}
  See \cite[Lemma 9.1.2]{jacobs} for full details. Here we just note
  that given $\coprod_u \dashv u^\ast$ with unit $\eta$ and
  $X \in p^{-1}(I)$, the map
  $\bar{u}(\coprod_u X) \circ \eta_X \colon X \rightarrow \coprod_u X$ is
  opcartesian.
\end{proof}

\begin{proposition}
  \label{prop:opcartfunct}
  Let $p \colon \cat{E} \rightarrow \cat{B}$ be a
  bifibration. Suppose that we are given the following diagram in
  $\cat{E}$.
  \begin{equation}
    \label{eq:24}
    \begin{gathered}
      \xymatrix{ X \ar[d] \ar[r] & U \\
        Y \ar[r] & V }
    \end{gathered}
  \end{equation}
  If the top map is opcartesian and we are given a map $t
  \colon p(U) \rightarrow p(V)$ making a commutative square in
  $\cat{B}$ then there is a unique map $U
  \rightarrow V$ over $t$ making a commutative square.  
\end{proposition}

\begin{proof}
  This is easy to show from the definition of opcartesian.
\end{proof}

In particular, let $p$ be cloven, and suppose we are given the
following square in $\cat{B}$.
\begin{equation}
  \label{eq:19}
  \begin{gathered}
    \xymatrix{ A \ar[r]^r \ar[d]_s & B \ar[d]^t \\
      C \ar[r]^u & D }
  \end{gathered}
\end{equation}
Then if we are given $f \colon X \rightarrow Y$ over $s$, this gives
us a canonical map $\coprod_r X \rightarrow \coprod_u Y$. It is easy
to check that in fact this is functorial using the uniqueness in
proposition \ref{prop:opcartfunct}.

\begin{definition}
  We say that $p$ satisfies the \emph{Beck-Chevalley condition} if the
  following holds.

  Suppose we are given a square in $\cat{E}$ of the following form.
  \begin{equation*}
    \xymatrix{ X \ar[d]_f \ar[r]^h & V \ar[d]^g \\
      Y \ar[r]_k & W }
  \end{equation*}
  If the image of the square in $\cat{B}$ is a pullback, $g$ is
  cartesian and $k$ is opcartesian then $h$ is opcartesian if and only
  if $f$ is cartesian.
\end{definition}

\begin{proposition}
  \label{prop:opfibcart}
  Suppose that $p$ satisfies the Beck-Chevalley condition, that
  \eqref{eq:19} is a pullback square and that $f$ is cartesian. Then
  the canonical map $\coprod_r X \rightarrow \coprod_u Y$ is
  cartesian.
\end{proposition}

\begin{proof}
  We consider the following commutative square in $\cat{E}$, where the
  bottom map is the canonical opcartesian map and the top is the
  unique map over $r$ making the square commute.
  \begin{equation*}
    \begin{gathered}
      \xymatrix{ X \ar[d]_f \ar[r] & t^\ast (\coprod_u Y)
        \ar[d]^{\bar{t}(\coprod_u Y)} \\
        Y \ar[r] & \coprod_u Y}
    \end{gathered}
  \end{equation*}
  Then by assumption $f$ is cartesian, so applying Beck-Chevalley, the
  top map is opcartesian. We deduce that there is a vertical
  isomorphism $\coprod_r X \cong t^\ast(\coprod_u Y)$. But this
  implies that the canonical map $\coprod_r X \rightarrow \coprod_u Y$
  is cartesian.  
\end{proof}

\subsection{Compositions of Fibrations and Bifibrations}
\label{sec:comp-fibr-bifibr}

In this section we will prove some useful lemmas about cartesian maps
over compositions of functors. We phrase the statements of the lemmas
so that they can be easily dualised to give corresponding results for
opcartesian maps. We will then apply these results to the composition
of fibrations
$\vrt{\cat{E}} \stackrel{\cod}{\rightarrow} \cat{E}
\stackrel{p}{\rightarrow} \cat{B}$ and to the composition of
opfibrations
$\vrt{\cat{E}} \stackrel{\dom}{\rightarrow} \cat{E}
\stackrel{p}{\rightarrow} \cat{B}$. We will in particular see that
opcartesian maps over $\dom$ are characterised as a square with
both horizontal maps opcartesian followed by a pushout. This will be
key to relating our general definition of step-one to Garner's
definition.

Throughout this section, we assume that we have functors $p$, $q$ and
$r$ as below.
\begin{equation*}
  \begin{gathered}
    \xymatrix{ \cat{F} \ar[dr]_r \ar[rr]^q & & \cat{E} \ar[dl]^p \\
      & \cat{B} &}
  \end{gathered}
\end{equation*}

\begin{lemma}
  \label{lem:cartvertabsolute}
  Let $r$ and $p$ be bifibrations, let $q$ preserve opcartesian
  maps, let $I \in \cat{B}$ and let $f \colon X \rightarrow Y$ in
  $\cat{F}_I$ be cartesian over
  $q_I \colon \cat{F}_I \rightarrow \cat{E}_I$. Then $f$ is also
  cartesian as a map in $\cat{F}$ over $q$.
\end{lemma}

\begin{proof}
  A straightforward diagram chase.
\end{proof}

\begin{lemma}
  \label{lem:cartabsolute}
  Suppose that $r$ and $p$ are bifibrations, and $q$ preserves
  cartesian maps. Then every map $f$ in $\cat{F}$ cartesian over $r$
  is also cartesian over $q$.
\end{lemma}

\begin{proof}
  We need to show $f$ is cartesian over $q$. Suppose we have the
  following three diagrams in $\cat{F}$, $\cat{E}$ and $\cat{B}$
  respectively.
  \begin{equation*}
    \begin{gathered}
      \xymatrix{ Z \ar@/^/[drr]^g & & \\
        & X \ar[r]_f & Y}
    \end{gathered}
    \begin{gathered}
      \xymatrix{ q(Z) \ar@/^/[drr]^{q(g)} \ar[dr]_{h} & & \\
        & q(X) \ar[r]_{q(f)} & q(Y)}
    \end{gathered}
    \begin{gathered}
      \xymatrix{ r(Z) \ar@/^/[drr]^{r(g)} \ar[dr]_{p(h)} & & \\
        & r(X) \ar[r]_{r(f)} & r(Y)}
    \end{gathered}    
  \end{equation*}
  Then there is a unique map $t$ over $p(h)$ making the left hand
  diagram commute. However, using the fact that $q(f)$ is cartesian in
  $\cat{E}$ over $r(f)$, we have that $q(t) = h$. Hence $t$ lies over
  $h$, and is clearly the unique such map. Hence $f$ is cartesian over
  $q(f)$.
\end{proof}

\begin{lemma}
  \label{lem:cartincomp}
  Suppose that $r$ and $p$ are bifibrations, and $q$ preserves both
  cartesian and opcartesian maps, and is a (cloven)
  isofibration. Suppose further that for all $I \in \cat{B}$, $q_I$ is
  a (cloven) fibration (and the cleavage is uniform in $I$). Then $q$
  is a fibration, and moreover every cartesian map over $q$ can be
  factored as a map vertical over $r$ (and necessarily cartesian over
  $q$) followed by a map cartesian over both $r$ and $q$.
\end{lemma}

\begin{proof}
  Let $f \colon X \rightarrow q(Y)$ in $\cat{E}$. Since $p$ is a
  fibration, we may factor $f$ as a map $f_1$ vertical over $p$,
  followed by $f_2$ cartesian over $p$. Then, since $r$ is a fibration
  we have a cartesian map $g_2$ over $p(f)$. Since $q(g_2)$ is
  cartesian and $q$ an isofibration, we may assume without loss of
  generality that $q(g_2) = f_2$. We also have that $g_2$ is cartesian
  over $q$ by lemma \ref{lem:cartabsolute}.

  Next, since $q_{p(X)}$ is a fibration, we have a vertical map $g_1$
  which is cartesian over $f_1$ in $q_{p(X)}$, and hence also
  cartesian as a map over $q$ by lemma \ref{lem:cartvertabsolute},
  with $\cod(f_1) = \dom(f_2)$.

  Now note that $f_2 \circ f_1$ is a composition of cartesian
  morphisms and so cartesian over $g$.

  This provides a cartesian lift witnessing that $q$ is a
  fibration. However, note also that every cartesian map is isomorphic
  to one of this form, and the property we are considering is preserved
  by isomorphism, so every cartesian map factors as described.
\end{proof}

\begin{lemma}
  \label{lem:opcartincomp}
  Suppose that $r$ and $p$ are bifibrations, and $q$ preserves both
  cartesian and opcartesian maps, and is a (cloven)
  isofibration. Suppose further that for all $I \in \cat{B}$, $q_I$
  is a (cloven) opfibration. Then $q$ is an opfibration, and moreover
  every opcartesian map over $q$ can be factored as a map opcartesian
  over both $r$ and $q$ followed by a map vertical over $r$ (and
  necessarily opcartesian over $q$).
\end{lemma}

\begin{proof}
  Dual to lemma \ref{lem:cartincomp}.
\end{proof}

\begin{lemma}
  \label{lem:fibredfunctorcomp}
  Suppose we are given the diagram of functors below.
  \begin{equation*}
    \begin{gathered}
      \xymatrix{ & \cat{F}_2 \ar[dr]^{q_2}  \ar'[d][dd]^{r_2} & \\
        \cat{F}_1 \ar[ur]^{\chi} \ar[rr]^{q_1} \ar[dr]_{r_1}
        & & \cat{E} \ar[dl]^p \\
        & \cat{B} &}
    \end{gathered}
  \end{equation*}
  Suppose further that $p$, $r_1$ and $r_2$ are (cloven) bifibrations,
  that $q_1$ and $q_2$ preserve cartesian and opcartesian maps over
  $\cat{B}$ and are (cloven) isofibrations. Suppose further that for
  all $I$, $q_{1, I}$ and $q_{2, I}$ are (cloven)
  fibrations. Then $\chi$ preserves all cartesian maps over $\cat{E}$
  if and only if it preserves cartesian maps over $\cat{B}$ and for
  each $I \in \cat{B}$, $\chi_I$ preserves cartesian maps over
  $\cat{E}_I$.
\end{lemma}

\begin{proof}
  Follows easily from the characterisation of cartesian maps in
  lemma \ref{lem:cartincomp}.
\end{proof}

\begin{lemma}
  \label{lem:codcartchar}
  Suppose that $p \colon \cat{E} \rightarrow \cat{B}$ is a bifibration
  and that for each $I \in \cat{B}$, $\cat{E}$ has pullbacks. Then
  $\cod \colon \vrt{\cat{E}} \rightarrow \cat{E}$ is a fibration
  and for any map $s$ in $\vrt{\cat{E}}$, the following are
  equivalent.
  \begin{enumerate}
  \item $s$ is a pullback, regarded as a square in $\cat{E}$.
  \item $s$ is cartesian over $\cod$.
  \item $s$ factors as a map vertical over $p \circ \cod$ and
    cartesian over $\cod$ followed by a map cartesian over both $\cod$
    and $p \circ \cod$.
  \item $s$ (regarded as a square in $\cat{E}$) factors as pullback
    with all maps vertical over $\cat{B}$, followed by a (pullback)
    square where both horizontal maps are cartesian in $\cat{E}$ over
    $p$.
  \end{enumerate}
\end{lemma}

\begin{proof}
  The usual proof that cartesian maps over $\cod$ are exactly pullback
  squares easily generalises to give us $1 \Leftrightarrow
  2$. Together with the observation that cartesian maps in
  $\vrt{\cat{E}}$ are exactly the levelwise cartesian maps, this
  also gives us $3 \Leftrightarrow 4$. (And it is easy to check that
  levelwise cartesian maps in $\vrt{\cat{E}}$ are pullbacks in
  $\cat{E}$.)

  Finally we get $2 \Leftrightarrow 3$ by lemma
  \ref{lem:cartincomp}. (It is easy to check that $\cod$ satisfies the
  necessary conditions to apply the lemma.)
\end{proof}

\begin{lemma}
  \label{lem:domopcartchar}
  Suppose that $p \colon \cat{E} \rightarrow \cat{B}$ is a bifibration
  and that for each $I \in \cat{B}$, $\cat{E}$ has pushouts. Then
  $\dom \colon \vrt{\cat{E}} \rightarrow \cat{E}$ is an
  opfibration and for any map $s$ in $\vrt{\cat{E}}$, the
  following are equivalent.
  \begin{enumerate}
  \item $s$ is a pushout, regarded as a square in $\cat{E}$.
  \item $s$ is opcartesian over $\dom$.
  \item $s$ factors as a map opcartesian over both $\dom$ and
    $p \circ \dom$ followed by a map vertical over $p \circ \dom$ and
    opcartesian over $\dom$.
  \item $s$ (regarded as a square in $\cat{E}$) factors as a (pushout)
    square where both horizontal maps are opcartesian in $\cat{E}$,
    followed by a pushout with all maps vertical over $\cat{B}$.
  \end{enumerate}
\end{lemma}

\begin{proof}
  Dual to lemma \ref{lem:codcartchar}.
\end{proof}

\subsection{Adjunctions, (Co)Pointed Endofunctors and (Co)monads over a
  Category}

We give definitions of adjunctions, (co)pointed endofunctors and
(co)monads over a fibration. These are fairly standard, although we
drop the requirement that they preserve cartesian maps, which is often
assumed in other places (see e.g. \cite[Chapter 1]{jacobs}).

\begin{definition}
  Let $p \colon \cat{E} \rightarrow \cat{B}$ and $q \colon \cat{F}
  \rightarrow \cat{B}$ be fibrations. An \emph{adjunction over
    $\cat{B}$} consists of an adjunction $F \dashv G$ from $\cat{E}$ to
  $\cat{F}$ such that $q \circ F = p$, $p \circ G = q$ and the unit
  and counit maps are vertical.
\end{definition}

\begin{remark}
  \label{rmk:adjnotfibred}
  We do not require that $F$ and $G$ in an adjunction over $\cat{B}$
  preserve cartesian maps. However, one can deduce from the definition
  that $G$ always does (see \cite[Exercise 1.8.5]{jacobs}). If
  $F$ also preserves cartesian maps, we say the adjunction is a
  \emph{fibred adjunction}.
\end{remark}

We will use later the following internal version of the usual
isomorphism between maps $F X \rightarrow Y$ and maps $X \rightarrow G
Y$.
\begin{lemma}
  \label{lem:fibredadjlemhom}
  Suppose that $p \colon \cat{E} \rightarrow \cat{B}$ and $q \colon
  \cat{F} \rightarrow \cat{B}$ are fibrations and $F \dashv G$ is a
  fibred adjunction from $\cat{E}$ to $\cat{F}$ over
  $\cat{B}$. Suppose that $X$ is an object of $\cat{E}$ over $I$ and
  $Y$ is an object of $\cat{F}$ over $J$. Suppose that the hom object
  from $F X$ to $Y$ exists, and is of the form below.
  \begin{equation*}
    \begin{gathered}
      \xymatrix{ I & & J \\
        & \fhom(F X, Y) \ar[ul]^\sigma \ar[ur]_\tau &}
    \end{gathered}
    \qquad
    \begin{gathered}
      \xymatrix{ F X & & & Y \\
        & \sigma^\ast(F X) \ar[r]_{\cong} \ar[ul]^{\overline
          \sigma(X)} &
         F(\sigma^\ast(X)) \ar[ur]_h &}
    \end{gathered}
  \end{equation*}

  Then the hom object from $X$ to $G Y$ (exists and) is of the form
  below.
  \begin{equation*}
    \begin{gathered}
      \xymatrix{ I & & J \\
        & \fhom(F X, Y) \ar[ul]^\sigma \ar[ur]_\tau &}
    \end{gathered}
    \qquad
    \begin{gathered}
      \xymatrix{ X & & G Y \\
        & \sigma^\ast(X) \ar[ul]^{\overline \sigma(X)} \ar[ur]_{\bar h} &}
    \end{gathered}
  \end{equation*}
  In particular $\fhom(F X, Y) \cong \fhom(X, G Y)$.
\end{lemma}

\begin{proof}
  Straightforward.
\end{proof}

\begin{lemma}
  \label{lem:ulpadjunction}
  Suppose that $p \colon \cat{E} \rightarrow \cat{B}$ and
  $q \colon \cat{F} \rightarrow \cat{B}$ are locally small fibrations,
  $\cat{B}$ has finite limits and $F \dashv G$ is a fibred adjunction
  from $\cat{E}$ to $\cat{F}$ over $\cat{B}$.

  Suppose that $m$ is a vertical map in $\cat{E}$ over $I$ and $f$ is
  a vertical map in $\cat{F}$ over $J$. Then solutions to the
  universal lifting property from $F m$ to $f$ correspond precisely
  to solutions of the universal lifting property from $m$ to $G f$.
\end{lemma}

\begin{proof}
  By applying lemma \ref{lem:fibredadjlemhom} to the lift of the
  adjunction to a fibred adjunction from $V(\cat{E})$ to $V(\cat{F})$
  and the usual properties of an adjunction.
\end{proof}

\begin{definition}
  Let $p \colon \cat{E} \rightarrow \cat{B}$ be a fibration. An
  \emph{endofunctor over $\cat{B}$} is a functor
  $T \colon \cat{E} \rightarrow \cat{E}$ such that $p \circ T = p$. A
  \emph{pointed endofunctor over $\cat{B}$} is a an endofunctor over
  $\cat{B}$ together with a natural transformation
  $\eta \colon 1 \Rightarrow T$ which is pointwise vertical. We define
  \emph{monad over $\cat{B}$} to be a pointed endofunctor with
  pointwise vertical multiplication $\mu \colon T^2 \Rightarrow T$
  satisfying the usual axioms for a monad. We dually define
  \emph{copointed endofunctor over $\cat{B}$} and \emph{comonad over
    $\cat{B}$}.
\end{definition}

\begin{remark}
  We again don't require an endofunctor necessarily to preserve
  cartesian maps. If it does we say it is \emph{fibred}.
\end{remark}

\begin{definition}
  Let $T$ be an endofunctor over $\cat{B}$. An \emph{algebra} over $T$
  is an object $X$ of $\cat{E}$ together with a vertical map
  $\alpha \colon T X \rightarrow X$.

  If $T, \eta$ is a pointed endofunctor, we define an \emph{algebra}
  over $T, \eta$ to be an algebra over the endofunctor which
  additionally satisfies the \emph{unit law}:
  \begin{equation*}
    \xymatrix{ X \ar@{=}[dr] \ar[r]^{\eta_X} & T X \ar[d]^\alpha \\
      & X}
  \end{equation*}

  If $T, \eta, \mu$ is a monad over $\cat{B}$, we define an
  \emph{algebra} over $T, \eta, \mu$ to be an algebra over the pointed
  endofunctor which additionally satisfies the \emph{multiplication
    law}:
  \begin{equation*}
    \xymatrix{ T^2 X \ar[r]^{T \alpha} \ar[d]_{\mu_X} & T X
      \ar[d]^\alpha \\
      T X \ar[r]_\alpha & X}
  \end{equation*}
\end{definition}

\begin{remark}
  \label{rmk:novertinalg}
  For algebras over a pointed endofunctor and over a monad we can drop
  the requirement that $\alpha$ is vertical, since it follows
  automatically from the unit law together with the axiom that the
  unit is vertical.
\end{remark}

\subsection{Functorial Factorisations and Lawfs's over a Fibration}
\label{sec:fibr-funct-fact}

In this section, we give generalisations of the usual definitions of
functorial factorisation, lawfs and awfs, as appear for instance in
\cite{garnersmallobject}.

We write $\cat{E}^\bthree_\vtcl$ for the full subcategory of
$\cat{E}^\bthree$ where both maps in an object are vertical.

\begin{definition}
  Let $p \colon \cat{E} \rightarrow \cat{B}$. A \emph{functorial
    factorisation over $\cat{B}$} is a functor
  $\vrt{\cat{E}} \rightarrow \cat{E}^\bthree_\vtcl$ over
  $\cat{B}$ that is a section of the composition functor
  $\cat{E}^\bthree_\vtcl \rightarrow \vrt{\cat{E}}$. If
  $f \colon X \rightarrow Y$ is a vertical map in $\cat{E}$ we will
  usually write the factorisation of $f$ as $X \stackrel{L
    f}{\longrightarrow} K f \stackrel{R f}{\longrightarrow} Y$.
\end{definition}

\begin{lemma}
  Let $p \colon \cat{E} \rightarrow \cat{B}$ be a fibration. Assume
  further that $p$ has all finite limits, and so $\cod \colon
  \vrt{\cat{E}} \rightarrow \cat{E}$.
  is also a fibration. The following categories are isomorphic.
  \begin{enumerate}
  \item Functorial factorisations $(L, R)$ over $p$.
  \item Pointed endofunctors $(R, \lambda)$ over
    $\cod \colon \vrt{\cat{E}} \rightarrow \cat{E}$.
  \item Copointed endofunctors $(L, \rho)$ over
    $\dom \colon \vrt{\cat{E}} \rightarrow \cat{E}$.
  \end{enumerate}
\end{lemma}

\begin{proof}
  We can take $\lambda_f$ to be $L f$ and vice versa to show 1 and 2
  are isomorphic.

  Similarly $\rho_f$ and $R f$ can be swapped to show 1 and 3 are
  isomorphic.
\end{proof}

\begin{definition}
  A \emph{left half of an algebraic weak factorisation system over
    $\cat{B}$} (lawfs) is a comonad over
  $\dom \colon \vrt{\cat{E}} \rightarrow \cat{E}$.

  We dually define \emph{right half of an algebraic weak factorisation
    system over $\cat{B}$} (rawfs) as a monad over
  $\cod \colon \vrt{\cat{E}} \rightarrow \cat{E}$
\end{definition}

\begin{proposition}
  Lawfs's correspond precisely to a functorial factorisations over
  $\cat{B}$ with a vertical natural transformation
  $\Sigma \colon L \Rightarrow L^2$ making $L$ into a comonad over
  $\dom \colon \vrt{\cat{E}} \rightarrow \cat{E}$.

  Rawfs's correspond precisely to a functorial factorisations over
  $\cat{B}$ with a vertical natural transformation
  $\Pi \colon R^2 \Rightarrow R$ making $R$ into a monad over
  $\cod \colon \vrt{\cat{E}} \rightarrow \cat{E}$.
\end{proposition}

\begin{proposition}
  \label{prop:fibredchar}
  Let $(L, R)$ be a functorial factorisation over $\cat{B}$. The
  following are equivalent.
  \begin{enumerate}
  \item The underlying functor between fibrations
    $\vrt{\cat{E}} \rightarrow \cat{E}^\bthree_\vtcl$ is fibred over
    $\cat{B}$.
  \item The endofunctor $L$ is fibred over
    $p \circ \dom$.
  \item The endofunctor $R$ is fibred over
    $p \circ \cod$.
  \end{enumerate}
\end{proposition}

\begin{definition}
  We say a functorial factorisation is \emph{fibred} if one of the
  equivalent conditions in proposition \ref{prop:fibredchar} holds.

  We say it is \emph{strongly fibred} if it is fibred
  as a pointed endofunctor over $\cod$.
\end{definition}

\begin{lemma}
  \label{lem:strongfibrelem}
  A functorial factorisation $(L, R)$ is strongly fibred if and only
  if it is both fibred and for each $I \in \cat{B}$, the restriction
  $R_I \colon V(\cat{E})_I \rightarrow V(\cat{E})_I$ preserves
  cartesian maps over $\cat{E}_I$. (Note that $V(\cat{E})_I$ is just
  $\cat{E}_I^\btwo$, and the fibration $V(\cat{E})_I \rightarrow
  \cat{E}_I$ is just codomain, so this says $R_I$ preserves pullbacks.)
\end{lemma}

\begin{proof}
  By lemma \ref{lem:fibredfunctorcomp}.
\end{proof}

\begin{definition}
  We say an lawfs or rawfs is \emph{(strongly) fibred} if the
  underlying functorial factorisation is (strongly) fibred.
\end{definition}

We will see that it is often easy to show functorial factorisations
are fibred under mild assumptions. Strongly fibred functorial
factorisations are much rarer. We will however see later some useful
lemmas for producing a few interesting examples of strongly fibred
awfs's.

\subsection{An Abstract Description of Step-one}
\label{sec:an-abstr-descr}

Ultimately we want to construct an lawfs over a fibration, which is a
comonad over the opfibration
$\dom \colon V(\cat{E}) \rightarrow \cat{E}$. We will first give a
more general construction of a comonad where we replace
$\dom \colon V(\cat{E}) \rightarrow \cat{E}$ with an arbitrary
opfibration $q$ in the diagram below.
\begin{equation}
  \label{eq:10}
  \begin{gathered}
    \xymatrix{ \cat{F} \ar[dr]_r \ar[rr]^q & & \cat{E} \ar[dl]^p \\
      & \cat{B} &}
  \end{gathered}
\end{equation}

This will allow us to give a clean proof that our claimed comonad
really is a comonad, including the construction of comultiplication
and the proof that it satisfies the ``square'' comonad law, which is
somewhat cumbersome to do directly. This description will also be
useful when we show that the generating family of left maps has a
coalgebra structure and (under suitable conditions) that the resulting
lawfs is fibred. However, we will see that the abstract description
can easily be reduced to give an explicit description of the action of
the comonad on objects in the special case we are interested in. The
reader may prefer to skip ahead to section \ref{sec:apply-abstr-descr}
to see how the abstract description is used before looking at this
section in detail.

Assume that we are given functors as in diagram \eqref{eq:10}.
Suppose further that $r$ is a locally small fibration and $q$ is an
opfibration. We will use this to construct a comonad $L_1$ over $q$. The
overall idea is to first construct an adjunction $F \dashv G$, and
then take $L_1$ to be $FG$.

\begin{definition}
  Let $X \in \cat{F}$. We define the category, $\bicart(X)$ of
  \emph{bicartesian spans from $X$} as follows. An object consists of
  $Y, Z \in \cat{F}$ together with $h \colon Z \rightarrow X$ cartesian
  over $r$ and $k \colon Z \rightarrow Y$ opcartesian over $q$. A
  morphism from
  $X \stackrel{h}{\leftarrow} Z \stackrel{k}{\rightarrow} Y$ to
  $X \stackrel{h'}{\leftarrow} Z' \stackrel{k'}{\rightarrow} Y'$
  consists of diagrams of the following form.
  
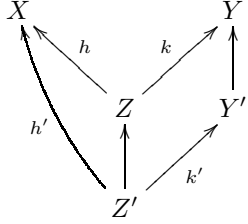
\begin{equation*}
    \begin{gathered}
      \xymatrix{ X & & Y \\
        & Z \ar[ul]_h \ar[ur]^k & Y' \ar[u] \\
        & Z' \ar@/^/[uul]^{h'} \ar[ur]_{k'} \ar[u] & }
    \end{gathered}
  \end{equation*}  
\end{definition}

We define $F \colon \bicart(X) \rightarrow \cat{F}$ to be
the functor sending
$X \stackrel{h}{\leftarrow} Z \stackrel{k}{\rightarrow} Y$ to $Y$. We
will define a right adjoint $G$ to $F$, giving us a comonad $FG$ on
$\cat{F}$.

Given $Y \in \cat{F}$, we construct $G(Y)$ as follows. Using the local
smallness of $r$ we have a span
as below, where $l$ is cartesian.
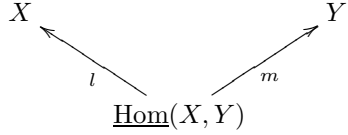
\begin{equation*}
  \begin{gathered}
    \xymatrix{ X & & Y \\
      & \fhom(X, Y) \ar[ul]^{l} \ar[ur]_{m} &}
  \end{gathered}
\end{equation*}
We then get an opcartesian map
$\fhom(X, Y) \rightarrow \coprod_{q(m)} \fhom(X, Y)$ over $q(m)$ given
by the opfibration structure on $q$. We then take $G(Y)$ to be
$X \stackrel{h}{\leftarrow} \fhom(X, Y) \stackrel{k}{\rightarrow}
\coprod_{q(m)} \fhom(X, Y)$. Given $n \colon Y \rightarrow Y'$, we
define $G(n)$ by first applying the universal property of $\fhom(X,
Y')$, and then applying the universal property of the opcartesian map
$\fhom(X, Y') \rightarrow \coprod_{q(m)} \fhom(X, Y')$.

We now check that $F \dashv G$. Suppose that we are given an object of
$\bicart(X)$ of the form $X \stackrel{h}{\leftarrow} Z
\stackrel{k}{\rightarrow} Y$ an object $Y'$ of $\cat{F}$, and a map $n
\colon Y \rightarrow Y'$. This gives us the following diagram.
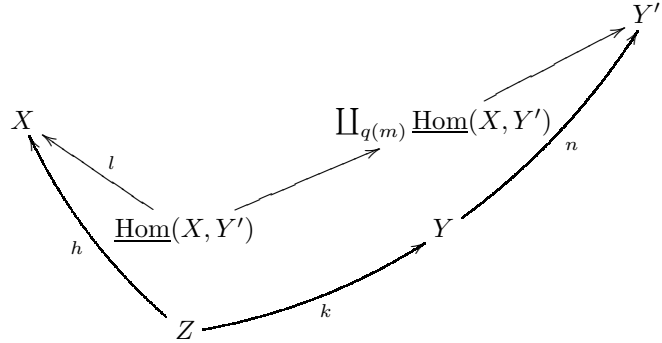
\begin{equation*}
  \begin{gathered}
    \xymatrix{ & & & Y' \\
      X & & \coprod_{q(m)} \fhom(X, Y') \ar[ur] & \\
      & \fhom(X, Y') \ar[ul]_l \ar[ur] & Y \ar@/_/[uur]_n \\
      & Z \ar@/^/[uul]^{h} \ar@/_/[ur]_k & & }
  \end{gathered}
\end{equation*}
We then apply the universal property of $\fhom(X, Y')$ followed by the
opcartesianess of $k$, to extend the diagram as below (and the right
hand triangle commutes, again using the opcartesianess of $k$).
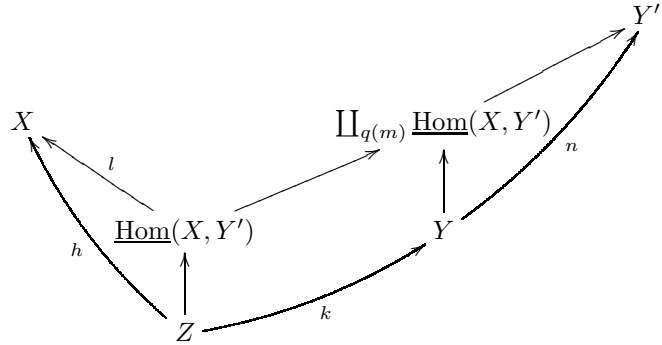
\begin{equation*}
  \begin{gathered}
    \xymatrix{ & & & Y' \\
      X & & \coprod_{q(m)} \fhom(X, Y') \ar[ur] & \\
      & \fhom(X, Y') \ar[ul]_l \ar[ur] & Y \ar@/_/[uur]_n \ar[u] \\
      & Z \ar@/^/[uul]^{h} \ar@/_/[ur]_k \ar[u] & & }
  \end{gathered}
\end{equation*}

But we now have a morphism from $X \stackrel{h}{\leftarrow} Z
\stackrel{k}{\rightarrow} Y$ to $G(Y')$. This operation is invertible
by composition with the canonical vertical map $\coprod_{q(m)}
\fhom(X, Y')
\rightarrow Y'$, and one can check that it is natural. Hence this does
give us an adjunction $F \dashv G$.

\begin{definition}
  \label{def:absstep1}
  We assume that we are given functors in the diagram below.
  
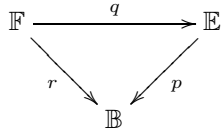
\begin{equation*}
    \begin{gathered}
      \xymatrix{ \cat{F} \ar[dr]_r \ar[rr]^q & & \cat{E} \ar[dl]^p \\
        & \cat{B} &}
    \end{gathered}
  \end{equation*}
  Suppose further that $r$ is a locally small fibration and $q$ is an
  opfibration, and we are given $X \in \cat{F}$. We will use this to
  construct a comonad over $q$.

  Let $F \dashv G$ be the adjunction above. We write $L_1$ for the
  resulting comonad $FG$. Note that in fact $F$ and $G$ can be viewed as
  functors over $\cat{E}$ and the unit and counit are both vertical,
  giving us an adjunction over $\cat{E}$. Hence the counit and
  comultiplication are vertical over $\cat{E}$ and so $L_1$ is a
  comonad over $\cat{E}$.
\end{definition}

We also get an explicit description of the counit of the
comonad. Namely, it is the map corresponding to the identity on $G Y$
under the adjunction, which unfolding the description above, is the
vertical map in the factorisation of the map
$\fhom(X, Y) \rightarrow Y$ as opcartesian map followed by vertical
map.

We now show that $X$ admits an $FG$-coalgebra structure.
\begin{lemma}
  \label{lem:genhascoalgstr}
  Let $X$ be the object appearing in the definition of $L_1$. Then $X$
  admits an $L_1$-coalgebra structure.
\end{lemma}

\begin{proof}  
  Define $\bar{X}$ to be the span
  $X \stackrel{1_X}{\longleftarrow} X \stackrel{1_X}{\longrightarrow}
  X$. The identity map is both cartesian and opcartesian, so this
  gives an object of $\bicart(X)$. We clearly have $X =
  F(\bar{X})$.

  Recall that for any adjunction
  $F \dashv G \colon \cat{C} \rightarrow \cat{D}$ and any object $Z$
  of $\cat{C}$, $F(Z)$ admits an $F G$-coalgebra structure, given by
  $F \eta_Z$, where $\eta$ is the unit of the adjunction.

  Applying this to $\bar{X}$ gives us an $FG$-coalgebra structure on
  $F(\bar{X}) = X$.
\end{proof}

Finally, we show a general result about when the comonad is
fibred.
\begin{definition}
  We say \emph{$q$-opcartesian maps are $r$-fibred} if the following
  holds. Suppose we are given a diagram of the following form in
  $\cat{F}$.
  \begin{equation*}
    \xymatrix{ X \ar[r]^f \ar[d]_h & Y \ar[d]^k \\
      W \ar[r]_g & Z}
  \end{equation*}
  Suppose further that $k$ and $h$ are cartesian over $r$, that $g$ is
  both vertical over $r$ and opcartesian over $p$, and that $f$ is
  vertical over $r$. Then the condition states that for every such
  square $f$ is also opcartesian over $q$.
\end{definition}

\begin{lemma}
  \label{lem:abststep1fibrd}
  Suppose that $r$ is a bifibration satisfying the Beck-Chevalley
  condition, that $q$-opcartesian maps are $r$-fibred and that
  $f \colon Y' \rightarrow Y$ is cartesian over $r$. Then
  $L_1 f \colon L_1 Y' \rightarrow L_1 Y$ is also cartesian over $r$.
\end{lemma}

\begin{proof}
  We consider the following square.
  \begin{equation}
    \label{eq:2}
    \begin{gathered}
      \xymatrix{ \sigma'^\ast(X) \ar[r] \ar[d] & \coprod_{q(m')}
        \fhom(X, Y') \ar[d] \\
        \sigma^\ast(X) \ar[r] & \coprod_{q(m)} \fhom(X, Y)}
    \end{gathered}
  \end{equation}
  The horizontal maps are opcartesian over $q$ so by lemma
  \ref{lem:opcartincomp} we can factorise each of them as maps
  opcartesian over $r$ followed by maps opcartesian over $q$ and
  vertical over $r$. Using the opcartesianess of the maps on the
  left, we also get a canonical map $t$ splitting the square into two,
  as below.
  \begin{equation*}
    \begin{gathered}
      \xymatrix{ \sigma'^\ast(X) \ar[r] \ar[d] & \cdot \ar[r] \ar[d]^t &
        \coprod_{q(m')}
        \fhom(X, Y') \ar[d] \\
        \sigma^\ast(X) \ar[r] & \cdot \ar[r] & \coprod_{q(m)} \fhom(X, Y)}
    \end{gathered}
  \end{equation*}
  
  Since the horizontal maps on the right are both vertical over $r$,
  the left hand square lies over the same square in $\cat{B}$ as
  \eqref{eq:2} did, which is the one below.
  \begin{equation*}
    \xymatrix{ \fhom(X, Y') \ar[r] \ar[d] & r(Y') \ar[d] \\
      \fhom(X, Y) \ar[r] & r(Y) }
  \end{equation*}
  However, by lemma \ref{lem:homprojpb} and the cartesianness of $f$
  this is a pullback square. Using this together with the
  Beck-Chevalley condition and proposition \ref{prop:opfibcart} we see
  that the $t$ must be cartesian.

  Since we now know $t$ is cartesian, we can use the assumption that
  $q$-opcartesian maps are $r$-fibred together with the same argument
  that we used in the proof of proposition \ref{prop:opfibcart} to
  show that the right hand map in \eqref{eq:2} is cartesian over $r$,
  as required.
\end{proof}

\subsection{Applying the Abstract Description}
\label{sec:apply-abstr-descr}

Suppose we are given a vertical map $m \colon U \rightarrow V$ over $I
\in \cat{B}$.

We will use this to define a lawfs over $p$ denoted \emph{step-one of
  the small object argument}, $(L_1, R_1)$.

Suppose we are given a vertical map $f \colon X \rightarrow Y$ over
$J$.

First we view $m$ and $f$ both as elements of $\vrt{\cat{E}}$
over $I$ and $J$ respectively. We then apply the abstract version of
step-one to the following diagram.
\begin{equation}
  \label{eq:9}
  \begin{gathered}
    \xymatrix{ \vrt{\cat{E}} \ar[dr]_{p \circ \cod = p \circ \dom}
      \ar[rr]^{\dom} & & \cat{E} \ar[dl]^p \\
      & \cat{B} & } 
  \end{gathered}
\end{equation}

Unfolding the abstract definition of $L_1$ in this case, we see that
for a fixed $m$ over $I$, we define the comonad $L_1$ over $\dom$ as
follows. Given a vertical $f$ we construct $L_1 f$ by first taking the
hom object from $m$ to $f$, which is just the universal lifting
problem from $m$ to $f$ and consists of the object $\fhom(m, f)$, maps
$\sigma \colon \fhom(m, f) \rightarrow I$ and
$\tau \colon \fhom(m, f) \rightarrow J$, and a map
$h_1 \colon \sigma^\ast(m) \rightarrow f$ in $\vrt{\cat{E}}$ over
$\tau$ (which recall we can also view as a square in $\cat{E}$ where
the horizontal maps lie over $\tau$). We then factor $h_1$ as an
opcartesian map followed by a vertical map, and the vertical map is
the value of $L_1 f$ together with the counit at $f$. Since $L_1$ is a
comonad over $\dom$ it is an lawfs.

\begin{lemma}
  \label{lem:step1algrmap}
  Let $(L_1, R_1)$ be the lawfs obtained by applying the abstract
  version of step-one to \eqref{eq:9}.
  Then $R_1$-algebra structures on $f$ correspond precisely to solutions of
  the universal lifting problem of $m$ against $f$.
\end{lemma}

\begin{proof}
  As stated above, in the first part of the abstract version of step-one
  we take the universal lifting problem and in the second part we
  factor the map $\fhom(m, f) \rightarrow f$ as a map which is
  opcartesian over $\dom$, followed by a map which is vertical over
  $\dom$. By lemma \ref{lem:domopcartchar} this gives us the following
  diagram in $\cat{E}$ where the rectangle is the universal lifting
  problem, the right hand square is the factorisation of $f$ given by
  step-one, and the left hand square is a pushout in the category
  $\cat{E}$.
  \begin{equation*}
    \xymatrix{ \sigma^\ast(U) \ar[r] \ar[d]_{\sigma^\ast(m)} & X \ar[d]^{L_1 f} \ar@{=}[r]
      & X \ar[d]^f \\
      \sigma^\ast(V) \ar[r] & K_1 f \pushoutcorner \ar[r]_{R_1 f} & Y
    }
  \end{equation*}
  It is now easy to see from the universal property of the pushout
  that diagonal fillers of the whole rectangle correspond to diagonal
  fillers of the right hand square. However, diagonal fillers of the
  right hand square are necessarily vertical over $p$ and are
  precisely $R_1$-algebra structures on $f$.
\end{proof}

\begin{lemma}
  \label{lem:miscoalg}
  Let $m$ be the family of generating left maps for $(L_1, R_1)$. $m$
  has the structure of an $L_1$-coalgebra.
\end{lemma}

\begin{proof}
  By lemma \ref{lem:genhascoalgstr}.
\end{proof}

We summarise the above as the following theorem.
\begin{theorem}[Step-one of the small object argument]
  \label{thm:step1}
  Let $p \colon \cat{E} \rightarrow \cat{B}$ be a locally small
  bifibration. Let $m$ be a vertical map over $p$. Then there is an
  lawfs $(L_1, R_1)$ such that $R_1$-algebra structures on a vertical
  map $f$ correspond precisely to solutions of the universal lifting
  problem from $m$ to $f$, and $m$ can be given the structure of an
  $L_1$-coalgebra.
\end{theorem}

\begin{proof}
  An lawfs is a comonad over $p \circ \dom$. We construct this as the
  comonad obtained by applying the abstract version of step-one
  (definition \ref{def:absstep1}) to the fibrations in diagram
  \eqref{eq:9}. The rest of the theorem is then lemmas
  \ref{lem:step1algrmap} and \ref{lem:miscoalg}.
\end{proof}

\begin{theorem}
  \label{thm:step1fibred}
  Suppose that $p\colon \cat{E} \rightarrow \cat{B}$ is a locally
  small bifibration that satisfies the Beck-Chevalley condition and
  has fibred pushouts. Then step-one of the small object argument,
  $(L_1, R_1)$ is a fibred lawfs.
\end{theorem}

\begin{proof}
  We aim to apply lemma \ref{lem:abststep1fibrd}.
  
  Since cartesian maps and opcartesian maps in $\vrt{\cat{E}}$ are
  just maps that are levelwise cartesian and opcartesian respectively,
  the Beck-Chevalley condition for $p$ implies the same for $\dom
  \circ p$. Next, note that the condition that $\dom$-opcartesian maps
  are $p$-fibred is exactly the assumption that pushouts are fibred.

  Hence by lemma \ref{lem:abststep1fibrd} we can deduce that $L_1$ is
  fibred over $\cat{B}$ and so is a fibred lawfs.
\end{proof}

\section{Criteria for the Existence of Algebraically Free Rawfs's and
  Awfs's}
\label{sec:crit-exist-algebr}

In this section we show that the existence of algebraically free
awfs's is equivalent to the existence of initial algebras for certain
pointed endofunctors. The motivation for doing this is that initial
algebras allow us to formalise in category theory the notion of
inductively generated object. We can then focus on arguments that
produce objects that are intuitively the ``least'' objects satisfying
some definition, then formalise this idea by showing they are initial
algebras, then deduce that we get awfs's by applying the results of
this section. This will be used by the author in a future paper alongside
a generalisation of dependent $W$-types, in which dependent polynomial
endofunctors are generalised to a certain class of pointed
endofunctors that will include as a special case the pointed
endofunctors appearing in this section in codomain fibrations.

\subsection{The Construction of Adjunctions and Monads}

The following lemma is based on an observation by Van den Berg and
Garner in \cite[Proposition 3.3.6]{vdberggarnerpathobj}.

\begin{lemma}
  \label{lem:algfibred}
  Suppose that $T \colon \cat{E} \rightarrow \cat{E}$ is a (not
  necessarily fibred) pointed endofunctor over the fibration
  $p \colon \cat{E} \rightarrow \cat{B}$. Write
  $U \colon \xalg{T} \rightarrow \cat{E}$ for the forgetful
  functor. Then for every cartesian map $f \colon X \rightarrow U(Y)$
  in $\cat{E}$ there is a unique cartesian map $g$ in $\xalg{T}$ such
  that $U(g) = f$ and $\cod{g} = Z$.

  Consequently, $p \circ U$ is a fibration, and $U$ is a fibred
  functor, regardless of whether or not $T$ is fibred.
\end{lemma}

\begin{proof}
  Let $f \colon X \rightarrow Y$ be cartesian in in $\cat{E}$ and let
  $\alpha \colon T Y \rightarrow Y$ be a $T$-algebra on $Y$. Then we
  have the following diagram.
  \begin{equation*}
    \begin{gathered}
      \xymatrix{T(X) \ar[rr]^{T(f)} & & T Y
        \ar@/^/[dr]^\alpha & \\
        & X \ar[rr]_{f} & & Y}
    \end{gathered}
  \end{equation*}
  Using the fact that $f$ is cartesian, there is a
  unique vertical map $T(X) \rightarrow X$
  making a commutative square. It is straightforward to check that
  this satisfies the unit law, and so is a $T$-algebra structure on
  $X$ and that the resulting morphism of $T$-algebras is
  cartesian in $\xalg{T}$.
\end{proof}

\begin{lemma}
  \label{lem:inittoadj}
  Suppose we are given a fibred functor
  $G \colon \cat{F} \rightarrow \cat{E}$ over $\cat{B}$. Suppose
  further that we are given a (not necessarily fibred) choice
  of initial object of $(X \downarrow G_I)$ for each $I \in \cat{B}$
  and each $X \in \cat{E}_I$ (where $G_I$ is the restriction
  $\cat{F}_I \rightarrow \cat{E}_I$). Then we have a (not necessarily
  fibred) adjunction $F \dashv G$.

  Furthermore, if the initial object at $X$ is
  $(FX, X \stackrel{\eta_X}{\longrightarrow} G F X)$, then these form
  the action on objects of the left adjoint and the unit of the
  adjunction respectively.
\end{lemma}

\begin{proof}
  This is a straightforward variation on the standard result for
  adjunctions.
\end{proof}

\begin{remark}
  It is necessary in lemma \ref{lem:inittoadj} to assume that $G$ is
  fibred, since this is the case for any adjunction by remark
  \ref{rmk:adjnotfibred}. If we do have a fibred choice of initial
  object then the same construction clearly gives us a fibred left
  adjoint (see \cite[Proposition 2.2.2]{pareschumacheradjfun}).
\end{remark}

\begin{definition}
  Let $T$ be a pointed endofunctor over a fibration
  $p \colon \cat{E} \rightarrow \cat{B}$. A monad $M$ over $p$ is
  \emph{algebraically free on $T$} if there is a morphism of pointed
  endofunctors $\xi \colon T \rightarrow M$ such that $\xi$ is
  levelwise vertical and the canonical map
  $\bar{\xi} \colon \xalg{M} \rightarrow \xalg{T}$ is an isomorphism.
\end{definition}

\begin{proposition}
  \label{prop:monadictoalgfree}
  Let $T$ be a pointed endofunctor over a fibration $p \colon \cat{E}
  \rightarrow \cat{B}$. If the forgetful functor $U \colon \xalg{T}
  \rightarrow \cat{E}$ has a left adjoint $F$ over $p$, then the monad
  $U F$ over $p$ is algebraically free on $T$.
\end{proposition}

\begin{proof}
  Except for the requirement that we are working over $\cat{B}$, the
  result is exactly the one proved by Kelly in \cite[Theorem
  22.3]{kellytransfinite}.

  However, it is easy to check that if the adjunction $F \dashv U$ is
  an adjunction over $p$ then $U F$ is a monad over $p$, and that the
  map constructed by Kelly is levelwise vertical.
\end{proof}

\begin{lemma}
  \label{lem:buildalgfreemonad}
  Suppose that $T, \eta$ is a (not necessarily fibred) pointed
  endofunctor over $p \colon \cat{E} \rightarrow \cat{B}$. Suppose
  further that we are given a choice of initial object of
  $(X \downarrow U_I)$ for each $I \in \cat{B}$ and $X \in
  \cat{E}_I$ (where $U_I$ is the forgetful functor, restricted to the
  fibre of $I$). Then we can construct a monad algebraically free on
  $T$.
\end{lemma}

\begin{proof}
  The basic idea is the follow \cite[Section 6.1]{gambinohylanddepw}
  with a few modifications.

  First, note that by lemma \ref{lem:algfibred}, $U$ is a fibred
  functor. We can therefore apply lemma \ref{lem:inittoadj} to
  construct a left adjoint $F$ to $U$ over $\cat{B}$. Note that this
  is in particular an ordinary adjunction $F \dashv U$. Furthermore,
  note that forgetful functors $U$ from algebras over a pointed
  endofunctor always create coequalizers for $U$-split pairs. Hence we
  can apply (the ordinary version of) Beck's monadicity theorem to
  show that the canonical functor $\xalg{T} \rightarrow \xalg{U F}$ is
  an isomorphism, and so the monad $U F$ is algebraically free on
  $T$.
\end{proof}

\begin{lemma}
  Suppose that $T, \eta$ is a (not necessarily fibred) pointed
  endofunctor over $p \colon \cat{E} \rightarrow \cat{B}$. Then
  algebraically free monads on $T$ are unique up to isomorphism.

  (In particular every algebraically free monad on $T$ is isomorphic
  to the one constructed in lemma \ref{lem:buildalgfreemonad}.)
\end{lemma}

\begin{proof}
  Let $(M, \lambda, \mu)$ be a monad and $\xi \colon T \rightarrow M$
  a morphism of pointed endofunctors witnessing that $M$ is
  algebraically free on $T$. For each $I$, the forgetful functor
  $\xalg{M}_I \rightarrow \cat{E}_I$ has a left adjoint,
  which sends $X$ to $M X$ together with the $M$-algebra structure
  $\mu_X \colon M^2 X \rightarrow M X$. Composing with the isomorphism
  $\xalg{T} \cong \xalg{M}$ gives a left adjoint to the forgetful
  functor $\xalg{T}_I \rightarrow \cat{E}_I$. The result now follows
  from uniqueness of adjoints.
\end{proof}

\begin{lemma}
  \label{lem:freemonadfibred}
  Suppose that $T, \eta$ is a fibred pointed endofunctor and $p \colon
  \cat{E} \rightarrow \cat{B}$ has all dependent products (which do
  not need to satisfy Beck-Chevalley). Then the
  algebraically free monad on $T$ is fibred.
\end{lemma}

\begin{proof}
  Let $\sigma \colon I \rightarrow J$ in $\cat{B}$. Then we have seen
  that the reindexing functor
  $\sigma^\ast \colon \cat{E}_J \rightarrow \cat{E}_I$ lifts to a
  functor $\sigma^\ast \colon (\xalg{T})_J \rightarrow
  (\xalg{T})_I$. However, when $T$ is fibred the dependent product
  $\Pi_\sigma \colon \cat{E}_I \rightarrow \cat{E}_J$ also lifts to a
  functor $(\xalg{T})_I \rightarrow (\xalg{T})_J$, and this is right
  adjoint to $\sigma^\ast$ (see \cite[Lemma B1.4.15(i)]{theelephant}
  for the result when $T$ has the structure of a monad and note the
  same proof holds for pointed endofunctors).

  Then we can lift this to get an adjunction between
  $(Y \downarrow U_J)$ and $(\sigma^\ast(Y) \downarrow U_I)$ for
  $Y \in \cat{E}_J$. Since
  $\sigma^\ast \colon (Y \downarrow U_J) \rightarrow (\sigma^\ast(Y)
  \downarrow U_I)$ is a left adjoint it preserves colimits and in
  particular initial objects. Hence any choice of initial objects is
  fibred, and so the $F$ constructed using such initial objects is
  also fibred.
\end{proof}

\subsection{Algebraic Weak Factorisation Systems}
\label{sec:defin-algebr-free}

We now the definition of algebraic weak factorisation over a
fibration. This is a generalisation of the well known definition due
to Grandis and Tholen \cite{grandistholennwfs} (originally referred to
as natural weak factorisation system).

\begin{definition}
  An \emph{algebraic weak factorisation system over $p$} (awfs)
  is a functorial factorisation over $p$ together with a both a vertical
  natural transformation $\Sigma \colon L \Rightarrow L^2$ making $L$
  into a comonad over
  $\dom \colon \vrt{\cat{E}} \rightarrow \cat{E}$ and a vertical
  transformation $\Pi \colon R^2 \rightarrow R$ making $R$ into a
  monad over $\cat{B}$, and the canonical natural transformation $LR
  \Rightarrow RL$ is a distributive law (see e.g. \cite[Definition
  4.12]{riehlams} for an elaboration of this part of the definition).
\end{definition}

\begin{definition}
  An awfs is \emph{(strongly) fibred} if the underlying functorial
  factorisation is (strongly) fibred.
\end{definition}

\subsection{Algebraically Free (R)awfs's}

\begin{definition}
  A \emph{morphism of functorial factorisations over $\cat{B}$} is a
  vertical natural transformation $\xi$ between the functors
  $\vrt{\cat{E}} \rightarrow \cat{E}^\bthree_\vtcl$. A
  \emph{morphism of lawfs's} is a morphism of functorial
  factorisations that also respects the comultiplication. A
  \emph{morphism of awfs's} is a morphism of lawfs's that also
  respects the multiplications.

  We define morphisms of rawfs's dually to those of lawfs's. We
  write the corresponding categories as $\ffcat(p)$, $\lawfscat(p)$,
  $\awfscat(p)$ and $\rawfscat(p)$ respectively.
\end{definition}

\begin{proposition}
  Any morphism of lawfs's $\xi \colon (L, R) \Rightarrow (L', R')$
  induces a functor $\xalg{R'} \rightarrow \xalg{R}$ over $\cat{B}$.
\end{proposition}

\begin{definition}
  Let $(L_0, R_0)$ be a functorial factorisation over a fibration
  $p$. An rawfs $(L, R)$ over $p$ is \emph{algebraically free} on
  $(L_0, R_0)$ if the corresponding monad over $p$ is algebraically
  free on the corresponding pointed endofunctor.
\end{definition}

\begin{definition}
  Let $(L, R, \Sigma)$ be an lawfs over $\cat{B}$. An
  \emph{algebraically free awfs on $(L, R, \Sigma)$} is an awfs $(L',
  R', \Sigma', \Pi')$ over $\cat{B}$ together with a morphism of
  lawfs's $\xi \colon (L, R, \Sigma) \Rightarrow (L', R', \Sigma')$
  such that the composition of the functor $\xalg{R'} \rightarrow
  \xalg{R}$ with the forgetful functor $\xalg{(R', \lambda, \Pi)}
  \rightarrow \xalg{(R', \lambda)}$ is an isomorphism.
\end{definition}

\begin{definition}
  Let $m$ be a vertical map over $p$. We say that an awfs $(L, R)$ is
  \emph{cofibrantly generated by $m$} if it is algebraically free on
  step-one.
\end{definition}

\begin{theorem}
  \label{thm:initialalgthmrawfs}
  Suppose that we are given a functorial factorisation $(L_0,
  R_0)$. Write $U \colon \xalg{R_0} \rightarrow \vrt{\cat{E}}$ for the
  forgetful functor. Suppose we are given an initial object of
  $(f \downarrow U_Y)$ for every object $Y$ of $\cat{E}$, and every
  vertical map $f \colon X \rightarrow Y$. Then there is an rawfs
  $(L, R)$ algebraically free on $(L_0, R_0)$.
\end{theorem}

\begin{proof}
  This is a special case of lemma \ref{lem:buildalgfreemonad}.
\end{proof}

\begin{theorem}
  \label{thm:algfreerawfsfib}
  Suppose that $p \colon \cat{E} \rightarrow \cat{B}$ is a fibration
  with dependent products (right adjoints to reindexing functors) and
  $(L_0, R_0)$ is a fibred functorial factorisation. If the
  algebraically free rawfs on $(L_0, R_0)$ exists (say $(L, R)$), then
  it is also fibred.
\end{theorem}

\begin{proof}
  First note that the fibration
  $p \circ \cod \colon \vrt{\cat{E}} \rightarrow \cat{B}$ also has
  dependent products, which are just defined levelwise. Recall that an
  rawfs $(L, R)$ is algebraically free on a functorial factorisation
  $(L_0, R_0)$ if $R$ is algebraically free on $R_0$ over
  $\cat{E}$. However, this implies that $R$ is also algebraically free
  on $R_0$ over $\cat{B}$. (In fact the algebras over $\cat{B}$ are
  the same as the corresponding algebras over $\cat{E}$ by remark
  \ref{rmk:novertinalg}.)

  Therefore we can apply lemma \ref{lem:freemonadfibred} to show that
  $R$ is fibred.
\end{proof}

\begin{theorem}
  \label{thm:algfreerawfsstrfib}
  Suppose that $p \colon \cat{E} \rightarrow \cat{B}$ is a fibration
  with dependent products (right adjoints to reindexing functors),
  that each fibre category $\cat{E}_I$ is locally cartesian closed and
  $(L_0, R_0)$ is a strongly fibred functorial factorisation. If the
  algebraically free rawfs on $(L_0, R_0)$ exists (say $(L, R)$), then
  it is also strongly fibred.
\end{theorem}

\begin{proof}
  We can construct dependent products for
  $\cod \colon \vrt{\cat{E}} \rightarrow \cat{E}$ using the dependent
  products for $p$ and the dependent products in the fibre categories
  as follows. We need to show that for every map
  $u \colon X \rightarrow Y$ in $\cat{E}$, the reindexing map $u^\ast$
  has a right adjoint. However, we know by the characterisation of
  cartesian maps in lemma \ref{lem:codcartchar} that this is the same
  as reindexing along $p(u)$ over $p \circ \cod$, followed by pulling
  back along the vertical component of $u$. The former has a right
  adjoint by the same argument as in theorem \ref{thm:algfreerawfsfib}
  and the latter has a right adjoint by the existence of dependent
  products. Hence each $u^\ast$ has a right adjoint by the composition
  of adjoints.
  
  Therefore we can apply
  lemma \ref{lem:freemonadfibred} to show $R$ is fibred over $\cod$,
  and so $(L, R)$ is strongly fibred.
\end{proof}

\subsection{Criteria for the Existence of Algebraically Free Awfs's}
\label{sec:crit-exist-algebr-1}

We now use the observation by Garner that awfs's can be characterised
in terms of double categories. We will follow the description by Bourke and
Garner in \cite[Section 2 and 3]{bourkegarnerawfs1} and by Riehl
in \cite[Section 2.5 and Section 6.2]{riehlams}.

The result is essentially that for a given rawfs $R$, the additional
structure making it an awfs corresponds precisely to a natural
composition operation, assigning an $R$-algebra structure to
$g \circ f$ given $R$-algebra structures on $f$ and $g$. We express
this composition operation as an extension of the category $\xalg{R}$
to a double category.

In this paper we adapt the results as follows. First of all, we are of
course working over a fibration rather than a category. Secondly, the
descriptions in \cite{riehlams} and \cite{bourkegarnerawfs1} consider
awfs's over different base categories, which requires a more
sophisticated notion of morphism (so called \emph{lax morphisms of
  awfs's}). On the other hand, we fix a single fibration throughout,
since that is all we need here. Finally, we make a minor observation
that does not appear in those papers. If instead of an rawfs we are
given a functorial factorisation, then one direction of the
correspondence still holds. Namely, given an lawfs we can produce a
composition operation on the $R$-algebra structures, where $R$ is now
just a pointed endofunctor rather than a monad.

Note that $\vrt{\cat{E}}$ has the structure of a double category
by taking vertical maps to be vertical maps in $\cat{E}$, taking
squares to be squares in $\cat{E}$ (with left and right maps
vertical), and taking composition to be composition in $\cat{E}$ in
each direction. (Note that we chose to orientate the double category
so that vertical maps in the double category sense are also vertical
maps in the Grothendieck fibration sense.)

\begin{lemma}
  \label{lem:awfstocomp}
  If $L$ is an awfs, and $R$ is the corresponding monad, then
  $R$-algebras can be composed naturally. Formally, there is a double
  category $\cat{T}$ whose vertical maps are $R$-algebras, whose
  squares are morphisms of $R$-algebras such that the forgetful
  functor $\xalg{R} \rightarrow \vrt{\cat{E}}$ is a double
  functor (i.e. it preserves the vertical composition).
\end{lemma}

\begin{proof}
  The usual proof generalises to fibrations without modification. See
  e.g. \cite[Definition 2.21 and lemma 2.22]{riehlams}.
\end{proof}

\begin{lemma}
  \label{lem:lawfstocomp}
  If $(L, R)$ is an lawfs, then $R$-algebras can be composed
  naturally. Formally, there is a double category $\cat{T}$ whose
  vertical maps are $R$-algebras (where these are now just pointed
  endofunctor algebras), whose squares are morphisms of $R$-algebras
  such that the forgetful functor
  $\xalg{R} \rightarrow \vrt{\cat{E}}$ is a double functor
  (i.e. it preserves the vertical composition).
\end{lemma}

\begin{proof}
  Observe that the explicit description in \cite[Definition
  2.21]{riehlams} makes no use of the multiplication on $R$. A
  lengthy but straightforward diagram chase verifies that this gives a
  functorial composition operation for $R$-algebras using only the
  lawfs axioms. The generalisation to fibrations is again
  straightforward.
\end{proof}

\begin{lemma}
  \label{lem:comptoawfs}
  If $R$ is an rawfs, and $R$-algebras can be composed naturally,
  then the corresponding copointed endofunctor $L$ can be given the
  structure of a comonad (and so $(L, R)$ the structure of an awfs).
\end{lemma}

\begin{proof}
  Again the usual proof generalises easily to fibrations. See
  e.g. \cite[Theorem 2.24]{riehlams} or \cite[Proposition
  4]{bourkegarnerawfs1}.
\end{proof}

\begin{lemma}
  \label{lem:doublefunctolawfsmorph}
  Suppose that we are given an lawfs $(L_1, R_1)$ and an awfs $(L, R)$
  together with a morphism of functorial factorisations (i.e. morphism
  of pointed endofunctors) $\xi \colon (L_1, R_1) \rightarrow (L,
  R)$. Suppose that the corresponding map $\bar{\xi} \colon \xalg{R}
  \rightarrow \xalg{R_1}$ preserves vertical composition (i.e. it is a
  double functor). Then $\xi$ respects comultiplication, and so is a
  morphism of lawfs's.
\end{lemma}

\begin{proof}
  In the second half of \cite[Theorem 6.9]{riehlams}, Riehl gives an explicit
  proof of this when $(L_1, R_1)$ is also an awfs. However, the
  multiplication on $(L_1, R_1)$ is never needed and in fact the same
  proof applies when $(L_1, R_1)$ is just an lawfs (and once again the
  proof generalises to a fibration without problems).
\end{proof}

We say a functor $U \colon \cat{T} \rightarrow \vrt{\cat{E}}$ is
\emph{monadic over $\cat{E}$} if it is isomorphic to $\xalg{T}$ where
$T$ over
$\cod \colon \vrt{\cat{E}} \rightarrow \cat{E}$. We then get the
following.

\begin{theorem}
  \label{thm:buildalgfreeawfs}
  Let $(L_1, R_1)$ be an lawfs. Then there is an awfs $(L, R)$
  algebraically free on $(L_1, R_1)$ if and only if the forgetful
  functor $U \colon \xalg{R_1} \rightarrow \vrt{\cat{E}}$ is
  monadic over $\cod \colon \vrt{\cat{E}} \rightarrow \cat{E}$.
\end{theorem}

\begin{proof}
  By proposition \ref{prop:monadictoalgfree} there is a monad $R$ over
  $\cod \circ p$ and a morphism of pointed endofunctors
  $\xi \colon R_1 \rightarrow R$ such that the corresponding map
  $\bar{\xi} \colon \xalg{R} \rightarrow \xalg{R_1}$ is an
  isomorphism. However, we can also view $R$ as an rawfs $(L, R)$ and
  $\xi$ as a morphism of functorial factorisations $(L_1, R_1)
  \rightarrow (L, R)$.

  We saw in lemma \ref{lem:lawfstocomp} that we can extend $\xalg{T}$
  to make a double category such that the forgetful functor to
  $\vrt{\cat{E}}$ is a double functor. However, $\bar{\xi}$ is an
  isomorphism between $\xalg{R_1}$ and $\xalg{R}$ over
  $\cat{E}$. Hence we can also give $\xalg{R}$ the structure of a
  double category by passing back and forth across the isomorphism.
  We then use lemma \ref{lem:comptoawfs} to assign a comultiplication
  to $R$, making it an awfs. By definition vertical composition is
  preserved by $\bar{\xi}$, and so by lemma
  \ref{lem:doublefunctolawfsmorph}, $\xi$ must preserve
  comultiplication, making it an lawfs morphism
  $(L_1, R_1) \rightarrow (L, R)$. But we can now deduce that $(L, R)$
  is algebraically free on $(L_1, R_1)$.
\end{proof}

\begin{theorem}
  \label{thm:initialalgthmawfs}
  Let $(L_1, R_1)$ be an lawfs.
  Write $U \colon \xalg{R_1} \rightarrow \vrt{\cat{E}}$ for the
  forgetful functor. Suppose we are given an initial object of
  $(f \downarrow U_Y)$ for every object $Y$ of $\cat{E}$, and every
  vertical map $f \colon X \rightarrow Y$. Then there exists an
  algebraically free awfs on $(L_1, R_1)$.
\end{theorem}

\begin{proof}
  By lemma \ref{lem:buildalgfreemonad} and theorem
  \ref{thm:buildalgfreeawfs}.
\end{proof}

\begin{corollary}
  \label{cor:initialalg}
  Let $p \colon \cat{E} \rightarrow \cat{B}$ be a locally small
  bifibration.

  Fix a vertical map $m$. Suppose that for every $Y \in \cat{E}$ and
  every vertical map $f \colon X \rightarrow Y$, the following pointed
  endofunctor on $\cat{E}/Y$ has an initial algebra. Send
  $g \colon Z \rightarrow Y$ to
  $\langle f, \rho_g \rangle \colon X \amalg R_1 g \rightarrow Y$,
  where $R_1$ is given by step-one. The point of the pointed
  endofunctor, at $g$ is the composition
  $Z \stackrel{\lambda_g}{\rightarrow} R g \hookrightarrow X \amalg R
  g$. Then the awfs cofibrantly generated by $m$ exists.
\end{corollary}

\begin{proof}
  One can check that for each $Y \in \cat{E}$ and each vertical map
  $f$, the category of algebras for the pointed endofunctor defined is
  isomorphic to $(f \downarrow U_Y)$. The result now follows from
  theorem \ref{thm:initialalgthmawfs}.
\end{proof}

\subsection{Fibred and Strongly Fibred Algebraically Free Awfs's}
\label{sec:fibr-strongly-fibr}

\begin{theorem}
  \label{thm:algfreeawfsfib}
  Suppose that $p \colon \cat{E} \rightarrow \cat{B}$ is a fibration
  with dependent products (right adjoints to reindexing functors) and
  $(L_1, R_1)$ is a fibred lawfs. If the algebraically free awfs on
  $(L_1, R_1)$ exists (say $(L, R)$), then it is also fibred.
\end{theorem}

\begin{proof}
  Same as theorem \ref{thm:algfreerawfsfib}
\end{proof}

\begin{theorem}
  \label{thm:algfreeawfsstrfib}
  Suppose that $p \colon \cat{E} \rightarrow \cat{B}$ is a fibration
  with dependent products (right adjoints to reindexing functors),
  that each fibre category $\cat{E}_I$ is locally cartesian closed and
  $(L_1, R_1)$ is a strongly fibred lawfs. If the algebraically free
  awfs on $(L_1, R_1)$ exists (say $(L, R)$), then it is also strongly
  fibred.
\end{theorem}

\begin{proof}
  Same as theorem \ref{thm:algfreerawfsstrfib}.
\end{proof}

\section{Fibred Leibniz Construction}

\subsection{Review of Monoidal Fibrations}
\label{sec:revi-fibr-mono}

We recall the following from \cite[Section
12 and 13]{shulmanfbmf}.

\begin{definition}[Shulman]
  A \emph{monoidal fibration} is a Grothendieck fibration $p \colon
  \cat{E} \rightarrow \cat{B}$ together with monoidal products on
  $\cat{E}$ and $\cat{B}$ such that $p$ strictly preserves monoidal
  products and the monoidal product $\otimes$ on $\cat{E}$ preserves
  cartesian maps.

  A monoidal fibration is \emph{cartesian} if ($\cat{B}$ has products
  and) the monoidal product on $\cat{B}$ is cartesian product
  $\times$.
\end{definition}

\begin{proposition}[Shulman]
  Let $p \colon \cat{E} \rightarrow \cat{B}$ be a cartesian monoidal
  fibration. Then for each $I \in \cat{B}$ we can define a monoidal
  product $\otimes_I$ in each fibre category, and these are preserved
  (up to natural isomorphism) by reindexing. (Moreover this forms part
  of an equivalence of $2$-categories between cartesian fibrations on
  $\cat{B}$ and pseudofunctors from $\opcat{\cat{B}}$ to monoidal
  categories - see \cite[Theorem 12.7]{shulmanfbmf} for details.)
\end{proposition}

\begin{example}
  \label{ex:codismonoidal}
  If $\catc$ is a category with pullbacks, then we can define a
  cartesian fibred monoidal product on the codomain fibration. We
  define the monoidal product on $\catc^\btwo$ to just be the
  cartesian product on $\catc^\btwo$, or in other words pointwise the
  cartesian product on $\catc$. The product in each fibre category
  $\catc/I$ is then just the cartesian product in $\catc/I$, which is
  just pullback in $\catc$.
\end{example}

\begin{example}
  Let $\catc$ be any category and $\otimes$ a monoidal product on
  $\catc$. We lift $\otimes$ to a fibred monoidal product on the
  category indexed families fibration $\fmlyc \rightarrow
  \cats$. Given $X \colon \smcat{A} \rightarrow \catc$ and
  $Y \colon \smcat{B} \rightarrow \catc$ we define $X \otimes Y$
  pointwise. That is, we define
  $X \otimes Y \colon \smcat{A} \times \smcat{B} \rightarrow \catc$ by
  $(X \otimes Y)(A, B) := X(A) \otimes Y(B)$. The resulting monoidal
  product $\otimes_\smcat{A}$ in each fibre category
  $[\smcat{A}, \catc]$ is of course also just defined pointwise.
\end{example}

\begin{definition}[Shulman]
  Let $p \colon \cat{E} \rightarrow \cat{B}$ be a cartesian monoidal
  fibration. We say that $p$ is \emph{internally closed} if each fibre
  $\cat{E}_I$ is closed monoidal (i.e. $- \otimes_I X$ and $X
  \otimes_I -$ have right adjoints for all $X$), and reindexing is
  closed monoidal (i.e. preserves these right adjoints).
\end{definition}

\begin{example}
  \label{ex:codismonoidalclosed}
  If $\catc$ is locally cartesian closed then the codomain fibration
  with monoidal structure from example \ref{ex:codismonoidalclosed} is
  internally monoidal closed.
\end{example}

We will also use the following proposition.
\begin{proposition}
  \label{prop:monoidalclosedtoadj}
  Let $p \colon \cat{E} \rightarrow \cat{B}$ be an internally closed
  cartesian monoidal fibration. Suppose further that $\cat{B}$ has a
  terminal object, and that $X$ is an object of $\cat{E}_1$. For $I
  \in \cat{B}$, write $\{ Y, - \}_I$ for the right adjoint to the
  functor $Y \otimes_I -$.
  
  Then there is an adjunction from $\cat{E}$ to $\cat{E}$ over
  $\cat{B}$ defined as follows. The left adjoint $F$ is defined by
  $F(Y) := I^\ast(X) \otimes_I Y$ for $Y \in \cat{E}_I$. The right
  adjoint $G$ is defined by $G(Z) := \{J^\ast(X) , Z\}$ for $Z \in
  \cat{E}_J$.
\end{proposition}

\begin{proof}
  It is straightforward to check that the definition given is an
  adjunction over $\cat{B}$.
\end{proof}

\subsection{Definition and Existence of the Fibred Leibniz Construction}
\label{sec:defin-exist-fibr}

The Leibniz construction is a standard construction in homotopical
algebra (see e.g. \cite[Construction 11.1.7]{riehlcht}), which takes a
monoidal closed structure on a category $\catc$, and produces a
monoidal closed structure on the arrow category $\catc^\btwo$. It was
applied to the semantics of homotopy type theory by Gambino and
Sattler in \cite{gambinosattlerpi} who showed how to construct
dependent products for certain awfs's cofibrantly generated by maps
defined using pushout product, including as a special case the Kan
fibrations in the CCHM cubical set model (from
\cite{coquandcubicaltt}).

In this section we show how to extend the Leibniz construction to monoidal
fibrations, which will be used later in some of our examples.

\begin{definition}
  Let $(\catc, \otimes)$ be a monoidal category with pushouts. The
  \emph{pushout product} is the monoidal product $\hat{\otimes}$
  defined on $\catc^\btwo$ as follows. Given $f \colon U \rightarrow
  V$ and $g \colon X \rightarrow Y$, we define $f \hat{\otimes} g$ as
  the map given by the universal property of the pushout below.
  \begin{equation*}
    \xymatrix{ U \otimes X \ar[r]^{f \otimes X} \ar[d]_{U \otimes g} &
      V \otimes X \ar[d] \ar@/^/[ddr]^{V \otimes g} & \\
      U \otimes Y \ar[r] \ar@/_/[drr]_{f \otimes Y}
      & \cdot \pushoutcorner \ar[dr]|{f \hat \otimes g} & \\
      & & V \otimes Y}
  \end{equation*}
\end{definition}

\begin{proposition}
  \label{prop:pushoutprodfib}
  Suppose we are given a monoidal fibration
  $p \colon \cat{E} \rightarrow \cat{B}$ with fibred pushouts. Then
  pushout product restricts to a monoidal product on $V(\cat{E})$ and
  this makes $\vrt{\cat{E}} \rightarrow \cat{B}$ into a monoidal
  fibration with the same monoidal product on $\cat{B}$.
\end{proposition}

\begin{proof}
  First note that all the objects and maps in the pushout in the
  definition of pushout product lie in the fibre of
  $p(f) \otimes p(g)$, and so $f \hat \otimes g$ does too. Then using
  the fact that $\otimes$ and pushouts preserve cartesian maps,
  pushout product must too. But this is enough to show we have a
  monoidal fibration.
\end{proof}

\begin{proposition}
  \label{prop:fibpushoutproddescr}
  Let $p \colon \cat{E} \rightarrow \cat{B}$ be a cartesian monoidal
  fibration where $\cat{B}$ has a terminal object. Suppose that $m_0$
  is a vertical map $\cat{E}$ in the fibre of the terminal object of
  $\cat{B}$ and $m$ is a vertical map over $I \in \cat{B}$. Let
  $f \colon X \rightarrow Y$ be a vertical map over $J \in \cat{B}$.
  Every family of lifting problems of $m_0 \hat{\otimes} m$ against
  $f$ over $K$ is isomorphic to a lifting problem of the following
  form in $\cat{E}$, where $\sigma'$ is the composition of $\sigma$
  with the canonical isomorphism $1 \times I \cong I$, and the
  horizontal maps lie over $\tau \colon K \rightarrow J$.
  \begin{equation*}
    \xymatrix{ \cdot \ar[d]_{m_0 \hat{\otimes} \sigma'^\ast(m)} \ar[r] &
      X \ar[d]^f \\
      \cdot \ar[r] & Y}
  \end{equation*}
\end{proposition}

\begin{proof}
  From the definition of lifting problem together with the fact that
  pushout product is fibred.
\end{proof}

\begin{definition}
  Let $(\catc, \otimes)$ be a monoidal category with pushouts and
  pullbacks. Suppose that for each $X$, $X \otimes -$ has a right
  adjoint $\{X, - \}$ (referred to as \emph{cotensor}). Then for each
  map $f$, $f \hat \otimes -$ has a right adjoint,
  $\widehat{\{f, -\}}$ referred to as \emph{pullback cotensor}, which
  is defined explicitly as the map given by the universal property of
  the pullback below. Let $f \colon U \rightarrow V$ and
  $g \colon X \rightarrow Y$.
  \begin{equation*}
    \xymatrix{ \{V, X\} \ar[dr]|{\widehat{\{f, g\}}}
      \ar@/^/[drr]^{\{V, g\}} \ar@/_/[ddr]_{\{f, X\}} & & \\
      & \cdot \pullbackcorner \ar[r] \ar[d] & \{V, Y\} \ar[d]^{\{f, Y\}} \\
      & \{U,X\} \ar[r]_{\{U, g\}} & \{U,Y\}
    }
  \end{equation*}
\end{definition}

We similarly can define a right adjoint to $- \hat \otimes f$ referred
to as \emph{pullback hom}.

\begin{proposition}
  Suppose we are given an internally closed monoidal fibration
  $p \colon \cat{E} \rightarrow \cat{B}$ with fibred pushouts.
  Then $\vrt{\cat{E}} \rightarrow \cat{B}$ is
  a internally closed monoidal fibration.
\end{proposition}

\begin{proof}
  Construct the pullback cotensor in each fibre category $\cat{E}_I$
  to get a right adjoint to $f \hat\otimes_I -$, which we'll write as
  $\{f, -\}_I$. This is fibred since cotensor and pullback are both
  fibred.
\end{proof}

\begin{proposition}
  \label{prop:fibleibnizulp}
  Let $p \colon \cat{E} \rightarrow \cat{B}$ be a monoidal closed
  fibration where $\cat{B}$ has a terminal object. Suppose that $m_0$
  is a vertical map $\cat{E}$ in the fibre of the terminal object of
  $\cat{B}$. Let $m$ be vertical over $I$ and $f$ vertical over $J$.

  Solutions of the universal lifting problem from
  $I^\ast(m_0) \hat \otimes_I m$ to $f$ correspond precisely to
  solutions of the universal lifting problem from $m$ to
  $\widehat{\{m_0, f\}}$.
\end{proposition}

\begin{proof}
  We apply the adjunction constructed in proposition
  \ref{prop:monoidalclosedtoadj} to pushout product and pullback
  cotensor. The result then follows from proposition
  \ref{prop:fibpushoutproddescr} and the characterisation of adjoints
  of hom objects in lemma \ref{lem:fibredadjlemhom}.
\end{proof}

\section{Examples}
\label{sec:examples}

We now give several examples of Grothendieck fibrations together with
explanations of what the general constructions look like in each
instance.

A theme throughout these examples is that our general construction was
defined in terms of hom objects and opcartesian maps, which are both
unique up to isomorphism. Therefore we can characterise what step-one of
the small argument looks like in each fibration by asking what are the
hom objects and what are the opcartesian maps.

For example, we will see that we recover a definition due to Garner by
applying our construction to fibrations of \emph{category indexed
  families} on a category $\catc$. In these fibrations, opcartesian
maps are described explicitly as left Kan extensions, which require
cocompleteness of $\catc$ to construct. This gives an explanation for
why cocompleteness of $\catc$ plays an important role in Garner's
small object argument, even for the relatively simple step-one part.

In contrast we will also look at codomain fibrations on a category
$\catc$. In this case the opcartesian maps are simply given by
composition, and so infinite colimits are not required for step-one. On
the other hand hom objects are now given by local exponentials in
slice categories, and so local cartesian closedness is a necessary
condition for step-one.

\subsection{Trivial Fibrations}
\label{sec:trivial-fibrations}

Let $\catc$ be a category. Then the unique functor
$\catc \rightarrow 1$ is a fibration. In this case a family of lifting
problems is just a single lifting problem and a choice of diagonal
fillers is just a diagonal filler. These fibrations are not locally
small and in fact universal lifting problems do not exist.

\subsection{Set Indexed Families}
\label{sec:set-indexed-families}

Set indexed families are the simplest nontrivial examples of
Grothendieck fibrations that we will consider. Although they are
simple, we can use them to illustrate the ideas that will turn up
again in other definitions. We will give a fairly brief
descriptions. See e.g. \cite{jacobs} for a more in depth reference on
set indexed families fibrations.

Recall the definition of set indexed families fibrations from example
\ref{ex:setindexfmly}.

A family of lifting problems from
$(F_i \colon U_i \rightarrow V_i)_{i \in I}$ to $(G_j \colon X_j
\rightarrow Y_j)_{j \in J}$
consists of a set $K$, together with maps $f \colon K \rightarrow I$
and $g \colon K \rightarrow J$ and for each $k \in K$, a commutative
square in $\catc$ of the following form:
\begin{equation}
  \label{eq:7}
  \begin{gathered}
    \xymatrix{ U_{f(k)} \ar[r] \ar[d]_{F_{f(k)}} & X_{g(k)}
      \ar[d]^{G_{g(k)}} \\
      V_{f(k)} \ar[r] & Y_{g(k)} }
  \end{gathered}
\end{equation}
A solution consists of a choice of diagonal filler for each such
square.

If $\catc$ is locally small, then so is $p$, with $\fhom_I(X, Y)$
given by the disjoint union of sets $\coprod_{i \in I} \hom(X_i,
Y_i)$.

Hence the universal family of lifting problems is defined as
follows. The indexing set (up to isomorphism) consists of triples
$\langle i, j, S \rangle$ where $i \in I$, $j \in J$ and $S$ is a
commutative square with left side equal to $F_i$ and right side equal
to $G_j$. The square in $\catc$ indexed at $\langle i, j, S \rangle$
is just $S$ itself.

\subsection{Category Indexed Families}
\label{sec:categ-index-famil}

Recall from example \ref{ex:catindexfmly} that $p$ from the previous
section can be extended to a fibration $\fmlyc \rightarrow \cats$. An
object of the larger $\fmlyc$ consists of a small category
$\smcat{C} \in \cats$ together with a functor
$X \colon \smcat{C} \rightarrow \catc$.

A vertical morphism over $\smcat{C}$ is just an object of
$[\btwo, [\smcat{C}, \catc]]$, but since $[\btwo, [\smcat{C}, \catc]] \cong
[\smcat{C} \times \btwo ,\catc] \cong
[\smcat{C}, [\btwo, \catc]]$, we can instead think of it as a functor from
$\smcat{C}$ to $\catc^\btwo$.

A family of lifting problems from
$F \colon \smcat{C} \rightarrow \catc^\btwo$ to
$G \colon \smcat{C} \rightarrow \catc^\btwo$ consists of a commutative
square for each object of the indexing category together with
commutative cubes for each morphism. A choice of fillers consists of a
choice of filler for each commutative square such that the resulting
``diagonal squares'' across each cube commute. We recover in this way
Garner's notion of lifting problem for $\catc$ from
\cite{garnersmallobject}.

$p \colon \fmlyc \rightarrow \cats$ is locally small if $\catc$ is
locally small (in the usual category theoretic sense). In fact, given
category indexed families $X \colon \smcat{A} \rightarrow \catc$ and
$Y \colon \smcat{B} \rightarrow \catc$, $\fhom(X, Y)$ is simply the
comma category $(X \downarrow Y)$. It is a bifibration if and only if
$\catc$ is cocomplete, with the opcartesian maps given by left Kan
extensions. We will see that although the Beck-Chevalley condition
does not hold in general for category indexed families we can
still show that step-one is fibred.

We now work towards a proof that Garner's definition of step-one of the
small object argument is the same as the result of applying the
general framework here to $\fmlyc$. As a corollary of this we get an
interesting new insight into Garner's definition. The original
definition of step-one is split into two pieces, first taking a certain
colimit, and then a pushout. On the other hand in the new definition
here, step-one is defined as a single opcartesian lift, or even just a
single pushout in $\fmlyc$. The key to seeing the link between the two
definitions is lemma \ref{lem:domopcartchar}, where we proved that the
single opcartesian lift over
$\dom \colon \vrt{\fmlyc} \rightarrow \fmlyc$ is equivalent to
an opcartesian lift over $\fmlyc \rightarrow \cats$, followed by a
pushout. Then for the special case where $f$ is a vertical map over
$1$, we are taking the opcartesian lift along a map
$\smcat{A} \rightarrow 1$, which is a colimit over a diagram of shape
$\smcat{A}$. This is now indeed of the same form as Garner's
definition, although we still need to check that it is the same
colimit, which will appear in lemma \ref{lem:steponematches}.

\begin{lemma}
  \label{lem:beckchevcif}
  Let $\cat{C}$ be a locally small category. Let $\smcat{A}$,
  $\smcat{B}$ and $\smcat{B}'$ be small categories. Suppose we are
  given $X$, $Y$ and $Y'$ over $\smcat{A}$, $\smcat{B}'$ and
  $\smcat{B}$ respectively. Suppose further that we have a map
  $\chi \colon \smcat{B}' \rightarrow \smcat{B}$. Write $\pi_0$ for
  the projection $(X \downarrow Y) \rightarrow \smcat{A}$ and $\pi_1$
  for the projection $(X \downarrow Y) \rightarrow \smcat{B}$. Write
  $(X \downarrow \chi)$ for the canonical map
  $(X \downarrow \chi^\ast(Y)) \rightarrow (X \downarrow Y)$ and
  $\pi_1'$ for the projection
  $(X \downarrow \chi^\ast(Y)) \rightarrow \smcat{B}'$. Then the
  canonical morphism $\coprod_{\pi_1'} (X \downarrow \chi)^\ast
  \rightarrow \chi^\ast \coprod_{\pi_1}$ is an isomorphism.
\end{lemma}

\begin{proof}
  In general if we are given a pullback square where the lower map is
  an opfibration then Beck-Chevalley holds for that square. (This
  appears to be a folklore result, see
  e.g. \cite{nlabbeckchevalley}
  for a proof.) However it is easy to check that the projection $(X
  \downarrow Y) \rightarrow \smcat{B}$ is an opfibration.
\end{proof}

\begin{lemma}
  \label{lem:steponefibrd}
  Step-one of the small object argument is fibred for category indexed
  families fibrations.
\end{lemma}

\begin{proof}
  Note that $\vrt{\fmlyc}$ is isomorphic to $\fmly(\catc^\btwo)$ over
  $\cats$, so is itself a category indexed family fibration. We then
  apply lemma \ref{lem:beckchevcif} not to $\catc$ itself, but to
  $\catc^\btwo$. We then note that for the proof of theorem
  \ref{thm:step1fibred} to hold, we don't need the entire
  Beck-Chevalley condition, but only certain instances that are
  precisely covered by this case.
\end{proof}

\begin{lemma}
  \label{lem:funtocatindfun}
  Let $\catc$ be a category. Suppose that we are given an
  endofunctor $T_1 \colon \catc \rightarrow \catc$. Then there is a
  fibred endofunctor $T$ over $\fmlyc \rightarrow \cats$ with the
  property that $T(X) = T_1(X)$ whenever
  $X \in \cat{C} \cong [1, \catc]$ and $T$ is unique up to
  isomorphism. Furthermore, we can ensure that $T$ strictly preserves
  reindexing (i.e. that for all $\sigma$ and $X$ we have
  $T(\sigma^\ast(X)) = \sigma^\ast(T(X))$ and
  $T(\bar{\sigma}(X)) = \bar{\sigma}(T(X))$), and is unique with this
  property.
\end{lemma}

\begin{proof}
  We simply define $T$ pointwise. That is, given
  $X \in [ \smcat{A}, \catc ]$ we define $T(X) = T_1 \circ X$. It is
  easy to see that this extends to an endofunctor over
  $\fmlyc \rightarrow \catc$. We recall that we defined reindexing as
  composition, and it is clear that this is strictly preserved by
  $T$.

  It remains to show uniqueness up to isomorphism. Strict uniqueness
  when $T$ strictly preserves reindexing is similar but easier.

  For $Z \in \catc$, write $\overline{Z}$ for the corresponding
  element of $\fmlyc$ over the fibre of $1$ and for a morphism $f$ in
  $\catc$ write $\overline{f}$ for the corresponding vertical map over
  $1$. Let $T$ be an endofunctor over $\fmlyc$ such that for all $Z$
  in $\catc$ we have $T(\overline{Z}) = \overline{T_1(Z)}$ and for all
  $f$ in $\catc$ we have $T(\overline{f}) = \overline{T_1(f)}$. Then
  to show that $T$ is isomorphic to the endofunctor described above,
  it suffices to find for each $X \colon \smcat{A} \rightarrow \catc$
  in $\fmlyc$ and each $A \in \smcat{A}$, an isomorphism
  $\alpha^X_A \colon T(\overline{X(A)}) \rightarrow (T X)(A)$, such
  that $\alpha^X_A$ is natural in both $A$ and $X$.

  Note that for any $X \colon \smcat{A} \rightarrow \catc$ in $\fmlyc$
  and for any object $A$ of $\smcat{A}$, we have a canonical cartesian
  map $\overline{X(A)} \rightarrow X$ over the map
  $\ulcorner A \urcorner \colon 1 \rightarrow \smcat{A}$. Applying $T$
  gives us a cartesian map $T(\overline{X(A)}) \rightarrow T X$ over
  $\ulcorner A \urcorner$. However, such a map is exactly a map from
  the underlying object in $\catc$ of $T(\overline{X(A)})$ to $(T
  X)(A)$, which is an isomorphism, by cartesianness. We take this to
  be $\alpha^X_A$.
  
  Naturality in $X$ is straightforward, but naturality in $A$ is more
  difficult, so we now give a proof. Let $\sigma \colon A
  \rightarrow B$ in $\smcat{A}$. We need to verify that the following
  diagram commutes.
  \begin{equation}
    \label{eq:6}
    \begin{gathered}
      \xymatrix@=3pc{
        T(\overline{X(A)}) \ar[r]_{T(\overline{X(f)})}
        \ar[d]_{\alpha^X_A} &
        T(\overline{X(B)}) \ar[d]^{\alpha^X_B} \\
        (T X)(A) \ar[r]^{(T X)(f)}
        &
        (T X)(B)
      }
    \end{gathered}
  \end{equation}

  The issue that we need to deal with is that, for $\sigma \colon A
  \rightarrow B$ in $\catc$, the following diagram in $\fmlyc$
  does \emph{not} commute, since its image in $\cats$ does not
  commute.
  \begin{equation*}
    \xymatrix{ \overline{X(A)} \ar[r] \ar[d]_{\overline{X(\sigma)}} & X \\
      \overline{X(B)} \ar[ur] &}
  \end{equation*}
  
  For any map $g$ in $\catc$, write $\tilde{g}$ for the corresponding
  object in $\fmlyc$ in the fibre of $\btwo$. Then similarly to
  before, we have for each $\sigma \colon A \rightarrow B$ in
  $\smcat{A}$ a canonical cartesian map
  $\widetilde{X(\sigma)} \rightarrow X$ over the map
  $\ulcorner \sigma \urcorner \colon \btwo \rightarrow \smcat{A}$. As
  before, we have a canonical cartesian maps
  $T(\widetilde{X(\sigma)}) \rightarrow T X$, which we'll write as
  $\beta_{\sigma}$. We can view $\beta_\sigma$ as a commutative square in
  $\catc$ of the following form.
  \begin{equation}
    \label{eq:18}
    \begin{gathered}
    \xymatrix@=3pc{
      \cdot \ar[d]_{T(\widetilde{X(\sigma)})}  \ar[r]^{\beta_{\sigma,0}} &
      (T X)(A) \ar[d]^{(T X)(\sigma)} \\
      \cdot
      \ar[r]^{\beta_{\sigma, 1}} &
      (T X)(B)
    }      
    \end{gathered}
  \end{equation}

  Note that $\alpha^X_A$ factors through $\beta_{\sigma, 0}$ and
  $\alpha^X_B$ factors through $\beta_{\sigma,1}$, by applying $T$ to
  the appropriate commutative diagram, which gives us the following,
  where each $\gamma_{\sigma, i}$ lies over the corresponding map $1
  \rightarrow \btwo$ in $\cats$.
  \begin{equation}
    \label{eq:1}
    \begin{gathered}
      \xymatrix{ T (\overline{X(A)}) \ar@/^/[drr]^{\alpha^X_A}
        \ar[dr]_{\gamma_{\sigma,
            0}} & & \\
        & \cdot \ar[r]_{\beta_{\sigma, 0}} & (T X)(A)}
    \end{gathered}
    \qquad
    \begin{gathered}
      \xymatrix{ T (\overline{X(B)}) \ar@/^/[drr]^{\alpha^X_B}
        \ar[dr]_{\gamma_{\sigma,
            1}} & & \\
        & \cdot \ar[r]_{\beta_{\sigma, 1}} & (T X)(B)}
    \end{gathered}
  \end{equation}
  
  If we consider the special case where $A = B$ and $\sigma = 1_A$,
  then we have the following.
  \begin{align*}
    \beta_{1_A, 1} \circ T(\widetilde{1_{X A}}) \circ
    \gamma_{1_A, 0} &= \beta_{1_A, 0} \circ \gamma_{1_A, 0}
    \\
                       &= \alpha^X_A \\
                       &= \beta_{1_A, 1} \circ \gamma_{1_A, 1}
  \end{align*}
  Since $\beta_{1_A, 1}$ is an isomorphism, we deduce that the
  following diagram commutes.
  \begin{equation}
    \label{eq:23}
    \begin{gathered}
      \xymatrix@=2.5pc{ T(\overline{X(A)}) \ar[r]^{\gamma_{1_A, 0}}
        \ar[dr]_{\gamma_{1_A, 1}} & \cdot
        \ar[d]^{T(\widetilde{1_{X A}})} \\
        & \cdot}  
    \end{gathered}
  \end{equation}

  Next, suppose that we have a commutative square of the following
  form in $\catc$.
  \begin{equation*}
    \xymatrix{ \cdot \ar[r]^f \ar[d]_h & \cdot \ar[d]^k \\
      \cdot \ar[r]_g & \cdot}
  \end{equation*}
  Then we can view this as a vertical map in $\fmlyc$ in the fibre of
  $\btwo$ from $\tilde{f}$ to $\tilde{g}$, which we write as
  $\widetilde{(h, k)}$. Applying $T$ to $\widetilde{(1_{X A},
    X(\sigma))}$ gives us a commutative square of the following form.
  \begin{equation}
    \label{eq:20}
    \begin{gathered}
      \xymatrix{ \cdot \ar[r]^{T(\widetilde{1_{X A}})}
        \ar[d]_{T \widetilde{(1_{X A}, X(\sigma)})_0} & \cdot
        \ar[d]^{T \widetilde{(1_{X A}, X(\sigma)})_1} \\
        \cdot \ar[r]_{T(\widetilde{X(\sigma)})} & \cdot}
    \end{gathered}
  \end{equation}

  Furthermore such squares are compatible with $\overline{h}$
  and $\overline{k}$, in the sense that we have the following
  commutative diagrams in $\fmlyc$ over each of the corresponding maps
  $1 \rightarrow \btwo$ in $\cats$.
  \begin{equation*}
    \begin{gathered}
      \xymatrix{ \cdot \ar[r] \ar[d]_{\overline{h}} & \cdot
        \ar[d]^{\widetilde{(h, k)}} \\
        \cdot \ar[r] & \cdot}
    \end{gathered}
    \qquad
    \begin{gathered}
      \xymatrix{ \cdot \ar[r] \ar[d]_{\overline{k}} & \cdot
        \ar[d]^{\widetilde{(h, k)}} \\
        \cdot \ar[r] & \cdot}
    \end{gathered}
  \end{equation*}

  As before we apply $T$ to a special case of the diagrams
  above, to get the following commutative diagrams.
  \begin{equation}
    \label{eq:17}
    \begin{gathered}
      \xymatrix{ T(\overline{X(A)})
        \ar[r]^{\gamma_{1_{A}, 0}} \ar[dr]_{\gamma_{\sigma, 0}} &
        T(\widetilde{1_{X A}}) \ar[d]^{T \widetilde{(1_{X A},
            X(\sigma))}_0} \\
        &
        T(\widetilde{X \sigma})
      }
    \end{gathered}
    \quad
    \begin{gathered}
      \xymatrix{ T(\overline{X(A)}) \ar[d]_{T(\overline{X(\sigma)})}
        \ar[r]^{\gamma_{1_{A}, 1}} &
        T(\widetilde{1_{X A}}) \ar[d]^{T \widetilde{(1_{X A},
            X(\sigma))}_1} \\
        T(\overline{X(B)}) \ar[r]_{\gamma_{\sigma, 1}} & T(\widetilde{X \sigma})}
    \end{gathered}
  \end{equation}

  Finally, we can now calculate,
  \begin{align*}
    (T X)(\sigma) \circ \alpha^X_A &= (T X)(\sigma) \circ \beta_{\sigma, 0}
                                \circ \gamma_{\sigma, 0} &
                                                           \text{by
                                                           \eqref{eq:1}}
    \\
                                   &= \beta_{\sigma, 1} \circ T(\widetilde{X(\sigma)}) \circ
                                     \gamma_{\sigma, 0} &
                                                          \text{by
                                                          \eqref{eq:18}}
    \\
    &= \beta_{\sigma, 1} \circ T(\widetilde{X(\sigma)}) \circ T \widetilde{(1_{X A},
            X(\sigma))}_0 \circ \gamma_{1_A, 0} & \text{by
                                                  \eqref{eq:17}} \\
    &= \beta_{\sigma, 1} \circ T \widetilde{(1_{X A},
            X(\sigma))}_1 \circ T(\widetilde{1_{X A}}) \circ
      \gamma_{1_A, 0} & \text{by \eqref{eq:20}} \\
    &= \beta_{\sigma, 1} \circ T \widetilde{(1_{X A},
            X(\sigma))}_1 \circ \gamma_{1_A, 1} & \text{by
                                                  \eqref{eq:23}} \\
    &= \beta_{\sigma, 1} \circ \gamma_{\sigma, 1} \circ
      T(\overline{X(\sigma)}) & \text{by \eqref{eq:17}} \\
    &= \alpha^X_B \circ T(\overline{X(\sigma)}) & \text{by \eqref{eq:1}}
  \end{align*}

  But we can now deduce that the
  naturality square \eqref{eq:6} commutes, as required.
\end{proof}

\begin{lemma}
  \label{lem:steponematches}
  The restriction of step-one to the fibre over $1$ is isomorphic to
  Garner's definition of step-one (the composition of the left adjoints
  to $\mathcal{G}_2$ and $\mathcal{G}_3$ in
  \cite[Section 4]{garnersmallobject}).
\end{lemma}

\begin{proof}
  Suppose we are given a vertical map $m$ over a category
  $\smcat{A}$. Write $M$ for the corresponding functor
  $\smcat{A} \rightarrow \catc^\btwo$. Since we are just evaluating
  step-one on the fibre of $1$, we suppose we are given a vertical map
  $f$ over $1$, which is just a map in $\catc$, say $f \colon X
  \rightarrow Y$.
  
  Expanding the definition of step-one from section
  \ref{sec:apply-abstr-descr} using
  the explicit description of hom objects in $\fmlyc$, we see that
  $L_1 f$ is constructed as follows. $\fhom(m, f)$ is the comma
  category $(M \downarrow f)$ and reindexing (i.e. composing) $M$
  along the first projection gives us a functor
  $M' \colon (M \downarrow f) \rightarrow \catc^\btwo$. We have an
  opfibration $\dom \colon \vrt{\fmlyc} \rightarrow
  \fmlyc$. $L_1 f$ is then the opcartesian lift of $M'$ along the
  image of the canonical map $M' \rightarrow f$ under $\dom$. However,
  we have seen in lemma \ref{lem:domopcartchar} that this is
  equivalent to first taking the opcartesian lift of $M'$ along the
  map $(M \downarrow f) \rightarrow 1$ followed by a vertical pushout
  over $1$. The former is just the left Kan extension of $M'$ along
  $(M \downarrow f) \rightarrow 1$, which is the colimit of $M'$
  regarded as a diagram in $\catc^\btwo$, or alternatively as the
  canonical map between the levelwise colimits in $\catc$. The pushout
  over $1$ is an ordinary pushout in $\catc$. But this description now
  matches Garner's.
\end{proof}

\begin{remark}
  Step 0 of the small object argument over $1$ can be viewed as a left
  Kan extension in two different ways. Firstly, following our general
  construction we have already seen in the proof above that we take
  the left Kan extension of $M'$ along
  $(M \downarrow f) \rightarrow 1$. This is just the colimit of the
  diagram $M'$. Note, however that this diagram is, as observed by
  Garner, precisely the pointwise formula for the left Kan extension
  of $M$ along itself at $f$.
\end{remark}

\begin{theorem}
  \label{thm:steponematches2}
  Our definition of step-one of the small object argument is simply the
  pointwise lift of Garner's definition from \cite{garnersmallobject} to
  $\fmlyc$.
\end{theorem}

\begin{proof}
  By lemmas \ref{lem:steponematches}, \ref{lem:steponefibrd} and
  \ref{lem:funtocatindfun}
\end{proof}

\begin{theorem}[Garner]
  \label{thm:garnerssmallobj}
  Let $\catc$ be a category satisfying either one of the following two
  conditions.
  \begin{enumerate}
  \item For every $X \in \catc$ there is a regular cardinal $\alpha$
    such that $\catc(X, -)$ preserves $\alpha$-filtered colimits.
  \item $\catc$ possesses a proper, well-copowered strong
    factorisation system $(\mathcal{E}, \mathcal{M})$, and for every
    object $X$ of $\catc$ there is $\alpha$ such that $\catc(X, -)$
    preserves $\alpha$-filtered unions of $\mathcal{M}$-subobjects.
  \end{enumerate}
  Then for every vertical map $m$ in $\fmlyc$, the awfs cofibrantly
  generated by $m$ exists.
\end{theorem}

\begin{proof}
  The well known result due to Garner \cite[Theorem
  4.4]{garnersmallobject} is for ordinary awfs's rather than for
  awfs's fibred over $\fmlyc$, so we need to show how to extend the
  result to this case.

  We saw in theorem \ref{thm:steponematches2} that our definition of
  step-one is isomorphic to the pointwise lift of Garner's definition to
  $\fmlyc$. Garner's construction of an algebraically free awfs
  clearly lifts pointwise to a fibred awfs over $\fmlyc$ which is
  clearly algebraically free on the pointwise lift of step-one.
\end{proof}

\begin{remark}
  Garner's proof is not constructively valid as stated. We draw
  attention to two issues. Firstly for the small object argument to
  apply in many situations, such as categories of presheaves, we
  require the existence of uncountable regular ordinals. However Gitik
  has proved in \cite{gitik1980} that this is independent of
  $\mathbf{ZF}$. So this requires some form of the axiom of choice. On
  the other hand a very weak, and constructively acceptable form of
  choice such as AMC (as appears in \cite{moerdijkpalmgrenast1})
  should suffice.

  Secondly the axiom of excluded middle is used in places. The main
  example is the assumption that every ordinal is either a limit or
  successor. The treatment of ordinals in \cite[Section
  9.4]{aczelrathjenbookdraft} suggests that with care the argument can
  be rephrased to work in constructive set theory, although it would be
  necessary to assume the existence of
  inaccessibles (as in \cite[Chapter 18]{aczelrathjenbookdraft}), in
  addition to AMC.
\end{remark}

Note that cofibrantly generated awfs's are automatically fibred, since
they are unique up to isomorphism, and we can always find a fibred one
by lifting pointwise the cofibrantly generated awfs over $1$. We
can show something similar for strongly fibred.

\begin{theorem}
  \label{thm:catfmlystrfibrd}
  Suppose that $\catc$ is cocomplete and pullbacks exist and preserve
  all colimits. If $(L_1, R_1)$ is a strongly fibred lawfs and the
  algebraically free awfs on $(L_1, R_1)$ exists, say $(L, R)$. Then
  it is also strongly fibred.
\end{theorem}

\begin{proof}
  Since $(L, R)$ exists, it has to be the pointwise lift of the
  explicit description given by Garner. Hence it suffices to show that
  Garner's construction is preserved by pullbacks. However, this
  description is a colimit of iterations of $L_1$, which is preserved
  by pullbacks by the assumptions.
\end{proof}

Finally, we remark that in this case the fibred pushout product
construction from section \ref{sec:defin-exist-fibr} reduces down to
the usual pushout product for monoidal categories.

Unfortunately,
however when $\catc$ is monoidal closed, $\fmlyc \rightarrow \cats$ is
not necessarily an internally closed monoidal fibration. Consider, for
example, the case $\catc = \set$ with cartesian product. Then the
cotensor over a small category $\smcat{A}$ is the exponential in
$\set^\smcat{A}$, which is not preserved by reindexing.

\subsection{Internal Category Indexed Families of Diagrams}
\label{sec:categ-index-famil-1}

Recall that we can define a notion of \emph{internal category}, which
generalises small categories by replacing the set of objects and the
set of morphisms with objects in some (usually large)
category. Furthermore, given a fibration $p \colon \cat{E} \rightarrow
\cat{B}$ and an internal category in $\cat{B}$, we can define a notion
of \emph{diagram}, which generalises functors $\smcat{A} \rightarrow
\set$, where $\smcat{A}$ is a small category. See e.g. \cite[Chapter
7]{jacobs} for formal definitions of both of these.

Let $p \colon \cat{E} \rightarrow \cat{B}$ be a fibration and let
$\smcat{C}$ be an internal category in $\cat{B}$. Then we
define a new fibration as follows. We define $\cats_\cat{B}$ to be the
category of internal categories and internal functors in $\cat{B}$. We
define $\cat{E}_{\smcat{C}}$ to consist of pairs $(\smcat{D}, X)$ where
$\smcat{D}$ is an internal category in $\cat{B}$ and $X$ is a diagram
of type $\smcat{C} \times \smcat{D}$ in $\cat{E}$.

Note that if we apply this with
$p \colon \fmly(\set) \rightarrow \set$ the fibration of set indexed
families of sets, we get the following. A internal category
$\smcat{C}$ in $\set$ is just a small category. Given another small
category $\smcat{D}$, a diagram of type $\smcat{C} \times \smcat{D}$
is just a functor $\smcat{C} \times \smcat{D} \rightarrow \set$. Then
using the isomorphism
$[\smcat{C} \times \smcat{D}, \set] \cong [\smcat{D}, [\smcat{C},
\set]]$,
we see that the resulting fibration is a special case of category
indexed families from section \ref{sec:categ-index-famil} where
$\cat{C}$ is the category $[\smcatc, \set]$ of presheaves over
$\opcat{\smcatc}$.

\begin{theorem}
  Suppose that $p \colon \cat{E} \rightarrow \cat{B}$ is a locally
  small fibration, $\cat{B}$ is locally cartesian closed, and
  $\smcat{C}$ is an internal category in $\cat{B}$. Then
  $\cat{E}_{\smcat{C}} \rightarrow \cats_\cat{B}$ is also locally
  small.
\end{theorem}

\begin{proof}
  In \cite[Lemma B2.3.15 (i)]{theelephant} Johnstone shows how to do
  this for the base change of the fibration along the discrete
  category functor $\cat{B} \rightarrow \cats_\cat{B}$. We will show
  how to extend this to the fibration over $\cats_\cat{B}$. The basic
  idea is to mimic internally the construction for category indexed families
  applied to a presheaf category. If we are given functors
  $X, Y \colon \smcat{A} \rightarrow [\smcat{C}, \set]$, then the
  result cited above already gives us the internal version of the hom
  object in set indexed families over the objects of $\smcat{A}$. We
  recall that the hom object for category indexed families is the
  small category where the set of objects is the hom object for set
  indexed families, which we recall consists of morphisms in
  $[\smcat{C}, \set]$ (i.e. natural
  transformations), with morphisms being those morphisms in $\smcat{A}$
  that are compatible with the natural transformations.

  Suppose we are given an internal category
  $\smcat{A} = A_1 \rightrightarrows A_0$ in $\cat{B}$ together with
  diagrams $X$ and $Y$ over $\smcat{A}$. We first need to define the
  object in $\cats_\cat{B}$ indexing the hom. That is, we need to
  define an internal category $\fhom_\smcat{A}(X, Y)$. Note that we
  can also view $X$ and $Y$ as objects in $\cat{E}_{\smcat{C}}$ over
  $A_0$. By the result in loc. cit. we have a hom object
  $h \colon \fhom_{A_0}(X, Y) \rightarrow A_0$ in $\smcat{B}$. We will
  take this to be the object of objects in the internal category, so
  we just need to define the object of morphisms.

  Let $\mu \colon \partial_0^\ast(X) \rightarrow \partial_1^\ast(X)$
  be the action of $X$, and
  $\nu \colon \partial_0^\ast(Y) \rightarrow \partial_1^\ast(Y)$ the
  action of $Y$, which we view as vertical maps over
  $\cat{E}_\smcat{C} \rightarrow \cat{B}$ (strictly speaking these are
  morphisms over $C_1 \times A_1$, but we can view them as morphisms
  over $C_0 \times A_1$ by reindexing along the identity map for
  $\smcat{C}$). By lemma \ref{lem:vertarrhom} (together with
  the result from loc. cit. and local cartesian closedness) we know
  that $\vrt{\cat{E}_{\smcat{C}}} \rightarrow \cat{B}$ is
  locally small.  We take the object of morphisms to be
  $\fhom_{A_1}(\mu, \nu)$, which comes equipped with a map
  $k \colon \fhom_{A_1}(\mu, \nu) \rightarrow A_1$. By considering the
  domain and codomain of the universal square, we get canonical maps
  $\partial_0', \partial_1' \colon \fhom_{A_1}(\mu, \nu) \rightarrow
  \fhom_{A_0}(X, Y)$ such that
  $\partial_i \circ k = h \circ \partial_i'$ for each $i$. We can then
  define the identity and multiplication by lifting those for
  $\smcat{A}$. Note that the universal square from $\mu$ to $\nu$ is a
  vertical morphism between the diagrams $X$ and $Y$.

  We now check that this is a hom object. Suppose that we are given an
  internal category $\smcat{D}$ and an internal functor
  $F \colon \smcat{D} \rightarrow \smcat{A}$ together with a vertical
  map $F^\ast X \rightarrow F^\ast Y$. This is a vertical map
  $f \colon F_0^\ast(X) \rightarrow F_0^\ast(Y)$ over
  $\cat{E}_\smcat{C} \rightarrow \cat{B}$ such that the following
  diagram commutes.
  \begin{equation}
    \label{eq:intfunctor2}
    \begin{gathered}
    \xymatrix{ F_1^\ast \partial_0^\ast X \ar[r]^{\cong}
      \ar[d]_{F_1^\ast \mu} &
      \partial_0^\ast F_0^\ast X \ar[r]^{\partial_0^\ast f} & \partial_0^\ast F_0^\ast Y
      \ar[r]^{\cong} & F_1^\ast \partial_0^\ast Y \ar[d]^{F_1^\ast \nu}
      \\
      F_1^\ast \partial_1^\ast X \ar[r]^\cong & \partial_1^\ast
      F_0^\ast X \ar[r]_{\partial_1^\ast f} & \partial_1^\ast F_0^\ast Y \ar[r]^\cong &
      F_1^\ast \partial_1^\ast Y
    }      
    \end{gathered}
  \end{equation}
  (In other words $F$ respects the action of $\smcat{A}$. $F$ also
  respects the action of $\smcat{C}$, but this is already implicit in
  the fact that $f$ is a map in $\cat{E}_\smcat{C}$.)

  We now need an internal functor from $\smcat{D}$ to the hom object.
  The map $f$ already gives us a map
  $D_0 \rightarrow \fhom_{A_0}(X, Y)$. We take this to be action of
  the functor on objects. For the action of the functor on morphisms,
  we use \eqref{eq:intfunctor2} together with the universal property
  of $\fhom_{A_1}(\mu, \nu)$. It is straightforward to check that this
  does give an internal functor, and witnesses the universal property
  of the hom object.
\end{proof}

\begin{theorem}
  Suppose that $p \colon \cat{E} \rightarrow \cat{B}$ is a cocomplete
  fibration (i.e. has dependent coproducts satisfying Beck-Chevalley),
  and $\smcat{C}$ is an internal category in $\cat{B}$. Then
  $\cat{E}_{\smcat{C}} \rightarrow \cats_\cat{B}$ is a bifibration
  with all finite colimits.
\end{theorem}

\begin{proof}
  Write $\tilde{\cat{E}}$ for the category consisting of triples $(\smcat{A},
  X, \mu)$ where $\smcat{A}$ is an internal category in $\cat{B}$ and
  $(X, \mu)$ is a diagram over $\smcat{A}$. Then the projection
  $\tilde{\cat{E}} \rightarrow \cats_\cat{B}$ is a fibration. Furthermore,
  for a small category $\smcat{C}$, the fibration $\cat{E}_\smcat{C}
  \rightarrow \cats_\cat{B}$ is the base change of $\tilde{\cat{E}}$ along
  the functor $- \times \smcat{C}$.
  
  \cite[Proposition B2.3.20]{theelephant} tells us that
  $\tilde{\cat{E}} \rightarrow \cats_\cat{B}$ has finite colimits and
  dependent coproducts, and so is a bifibration. However, these
  properties are preserved by base change, so the same applies to
  $\cat{E}_\smcat{C} \rightarrow \cats_\cat{B}$.
\end{proof}

\subsubsection{Presheaf Assemblies}
\label{sec:presheaf-assemblies}

We recall some basic definitions in realizability. See e.g.
\cite[Section 1.2]{jacobs} for more details.

\begin{definition}
  An \emph{assembly over $\mathcal{K}_1$}, or \emph{$\omega$-set},
  consists of a pair $\langle X, E \rangle$ where $X$ is a set, and
  $E$ is a function $X \rightarrow \powset^\ast(\nat)$ (non empty
  subsets of $\nat$). We refer to $X$ as the \emph{underlying set} and
  to $E$ as the \emph{existence predicate} of the assembly.

  We will also refer to these just as assemblies, since we won't
  consider any other categories of assemblies in this paper.
\end{definition}

\begin{definition}
  Let $\langle X, E \rangle$ and $\langle X', E' \rangle$ be
  assemblies. We say a function $f \colon X \rightarrow X'$ is
  \emph{tracked}, or \emph{computable} if there exists $e \in \nat$
  satisfying the following. Write $\varphi_e$ for the $e$th computable
  function. For all $x \in X$ and all $n \in E(x)$, we have that
  $\varphi_e(n)$ is defined and $\varphi_e(n) \in E'(f(x))$.
\end{definition}

\begin{definition}
  Assemblies and computable functions form a category, which we denote
  $\asm$.
\end{definition}

\begin{proposition}
  $\asm$ has all finite limits and finite colimits and is locally
  cartesian closed.
\end{proposition}

Hence we see that the codomain fibration on $\asm$ is locally small.

We now give explicit descriptions of internal category, internal
functor and diagram in assemblies. These are all straightforward to
prove by unfolding the general definition.

\begin{proposition}
  An internal category over $\cod \colon \asm^\btwo \rightarrow \asm$
  is (up to equivalence) a small category
  $\smcat{C} = \langle C_1, C_0 \rangle$ together with existence
  predicates $E_0 \colon C_0 \rightarrow \powset^\ast(\nat)$ and
  $E_1 \colon C_1 \rightarrow \powset^\ast(\nat)$ such that the
  domain, codomain, identity and composition functions are all
  computable.
\end{proposition}

\begin{proposition}
  An internal functor between internal categories $\smcat{C}$ and
  $\smcat{D}$ in $\asm$ consists of a functor $F \colon \smcat{C}
  \rightarrow \smcat{D}$ between underlying small categories, such
  that its action on objects and action on morphisms are both
  computable functions.
\end{proposition}

\begin{proposition}
  Let $\smcat{C}$ be an internal category in assemblies. A diagram
  over $\smcat{C}$ is (up to equivalence) a functor $X \colon
  \smcat{C} \rightarrow \asm$ which is uniformly computable, in the
  following sense. For each object $A$ of $\smcat{C}$, write $X_0(A)$
  for the underlying set and $X_1(A)$ for the existence predicate of
  the assembly $X(A)$. Then there exists $e \in \nat$, such that for
  all morphisms $s \colon A \rightarrow B$ in $\smcat{C}$, for all $n
  \in E_1(s)$, for all $x \in X_0(A)$, for all $m \in X_1(A)(x)$,
  $\varphi_e(n, m)$ is defined, with $\varphi_e(n, m) \in
  X_1(B)(X(s)(x))$.
\end{proposition}

Now applying the internal category indexed family construction to
assemblies, we get the following description of lifting problems.

An object of $\cat{E}_\smcat{C}$ consist of a functor
$\smcat{C} \times \smcat{D} \rightarrow \asm$ that is uniformly
computable. In particular, $\cat{E}_{\smcat{C},1}$ consists of
presheaf assemblies over $\smcat{C}$. A vertical map can then be
viewed as a uniformly computable functor
$\smcat{D} \rightarrow [\smcat{C}, \asm]^\btwo$ (where we define
uniformly computable similarly to functors to $\asm$). A choice of
fillers is then a choice of fillers from section
\ref{sec:categ-index-famil} (i.e. a uniform choice of fillers in the
sense of \cite{garnersmallobject}) that satisfies the addition
requirement of being uniformly computable.

In this way, we see it is easy to develop realizability variants of
cofibrantly generated classes of maps in presheaf categories. To
illustrate this, we show that Kan fibrations in simplicial assemblies,
as defined by Stekelenburg in \cite{stekelenburgsimpass} are
a cofibrantly generated class in our general sense.

We define the computable simplex category $\Delta$ to be the following
internal category in $\asm$. We take the underlying category to be the
usual simplex category, $\Delta$. We define the existence predicate on
objects by $E_0([n]) = \{n\}$. The existence predicate on morphisms is
defined as follows. We can view $\sigma \colon [n] \rightarrow [m]$ as
an order preserving function between finite sets. We define
$E_1(\sigma)$ to be the set of natural numbers that track this
function. (This is the same as constructing the simplex category in
the internal logic of assemblies.) A simplicial assembly is then (up
to equivalence) a functor $\opcat{\Delta} \rightarrow \asm$ which is
uniformly computable.

We define a family of maps as follows. We take the underlying category
$\smcat{D}$ to be the discrete category defined as follows. The
underlying set consists of pairs $\langle n, k \rangle$ where
$k \leq n$. The existence predicate is defined taking
$E(\langle n, k \rangle)$ to be $\{(n, k)\}$, where $(-, -)$ is a
computable encoding for pairs. To define a vertical map in the
fibration over $\smcat{D}$ is to define a uniformly computable functor
$\opcat{\Delta} \times \smcat{D} \rightarrow \asm^\btwo$. We first
note that the Yoneda embedding is clearly computable, and so we get a
functor $\smcat{D} \rightarrow [\Delta, \asm]$ sending $(n, k)$ to
$\Delta^n$. We can then make the $k$th horn $\Lambda^k[n]$ into an
assembly by taking the existence predicate to be the restriction of
the existence predicate on $\Delta^n$. This then makes the functor
sending $\langle n, k \rangle$ to the horn inclusion $\Lambda^k[n]
\hookrightarrow \Delta^n$ uniformly computable.

Finally, we can see that a map with the right lifting property against
this functor is a computable variant of the usual notion of Kan
fibration. By unfolding the definition of universal lifting problem
for this case we can see that it is a morphism $f$ in simplicial
assemblies with a Kan filling operator providing a filler for every
lifting problem against a horn inclusion, which is uniformly
computable. Here uniformly computable means uniform both in the choice
of horn inclusion and the lifting problem. In other words, if we are
given $n$, $k$ and natural numbers tracking the horizontal maps in a
lifting problem of $\Lambda^k[n] \hookrightarrow \Delta^n$ against
$f$, then we can compute a number tracking the choice of diagonal
filler.

\subsection{Codomain Fibrations}
\label{sec:codomain-fibrations}

Let $\catc$ be a category with pullbacks. Then the codomain functor
$\cod \colon \catc^\btwo \rightarrow \catc$ is a fibration. We assume
that $\catc$ has all finite limits. Opcartesian lifts in $\cod$ always
exist, and are given by composition. The codomain fibration is locally
small exactly when $\catc$ is locally cartesian closed: given $I$ in
$\catc$, and $U \rightarrow I$ and $V \rightarrow I$ in $\catc/I$ we
define $\fhom_I(U, V)$ to be the exponential in $\catc/I$.

A family of maps, is just a morphism in a slice category. Suppose we
are given families of maps $m$ and $f$ over $I$ and $J$ respectively
as below:
\begin{equation*}
  \begin{gathered}
    \xymatrix{ U \ar[dr] \ar[rr]^m & & V \ar[dl] \\
      & I & }
  \end{gathered}\
  \qquad
  \begin{gathered}
    \xymatrix{ X \ar[dr] \ar[rr]^f & & Y \ar[dl] \\
      & J &}
  \end{gathered}
\end{equation*}

A family of lifting problems from $m$ to $f$ over $K$ is a diagram of
the following form, where both squares on the left are pullbacks.
\begin{equation*}
  \begin{gathered}
    \xymatrix{ U \ar[d] & \sigma^\ast(U) \ar[l] \ar[r] \ar[d]
      \pullbackcorner[dl] & X \ar[d] \\
      V \ar[d] & \sigma^\ast(V) \ar[l] \ar[r] \ar[d]
      \pullbackcorner[dl] & Y \ar[d] \\
      I & K \ar[l]_\sigma \ar[r] & J
    }
  \end{gathered}
\end{equation*}

From this description (and proposition \ref{prop:univlp}) we easily
get the following.
\begin{proposition}
  \label{prop:codfibredlpasordinary}
  Let $m$ be a vertical map over an object $I$ of $\catc$. Given a map
  $f$ in $\catc$, we view it as a vertical over $1$ (using the
  canonical isomorphism $\catc \cong \catc/1$). Then $f$ has the
  (fibred) right lifting property against $m$ if and only if it has
  the (ordinary) right lifting property against $\sigma^\ast(m)$ for
  every $K \in \catc$ and $\sigma \colon K \rightarrow I$.
\end{proposition}

The universal lifting problem is then defined by taking $K$ as
below. Recall that we can define $\fhom$ using the local exponential
in a slice category (in this case $\catc/I \times J$).
\begin{equation*}
  K := \fhom(U, X) \times_{\fhom(U, Y)} \fhom(V, Y)
\end{equation*}
The right hand maps to $X$ and $Y$ are then the evident evaluation
maps.

Over codomain fibrations we can easily show that lawfs's and awfs's
are fibred, as we show below.
\begin{theorem}
  Let $\cat{C}$ be locally cartesian closed category with
  pushouts. Let $m$ be a vertical map over $\cod \colon \catc^\btwo
  \rightarrow \catc$.
  \begin{enumerate}
  \item \label{codstep1exists} Step-one of the small object argument on
    $m$ is well defined, giving us an lawfs $(L_1, R_1)$, where the
    $R_1$ algebra structures on a family of maps $f$ correspond
    precisely to solutions of the universal lifting problem of $f$
    against $m$.
  \item \label{codstep1fibred} $(L_1, R_1)$ is a fibred lawfs.
  \item \label{codawfsfibred} If the awfs algebraically generated by
    $(L_1, R_1)$ exists, then it is also fibred.
  \end{enumerate}
\end{theorem}

\begin{proof}
  First note that $\cod$ is a bifibration, with opcartesian maps given
  by ``composition'' and the Beck-Chevalley condition is an easy
  diagram chase. As explained above $\cod$ is also locally
  small. Local cartesian closedness is used again to get dependent
  products in $\catc$. In particular reindexing preserves pullbacks
  (as long as they exist) since it has a right adjoint.

  Hence we get an lawfs, and it is fibred by theorem
  \ref{thm:step1fibred}. If the cofibrantly generated awfs exists then
  it is also fibred by theorem \ref{thm:algfreeawfsfib}.
\end{proof}

Sometimes useful to use an alternative formulation of
the universal lifting problem. When we gave the general definition of
universal lifting problem we were motivated by thinking of a lifting
problem from $m$ to $f$ as a map $\alpha \colon U \rightarrow X$
together with a map $\beta \colon V \rightarrow Y$ satisfying the
condition that $f \circ \alpha = \beta \circ m$. Using dependent
products we can instead formalise an alternative idea. A lifting
problem of $m$ against $f$ is a map $\beta \colon V \rightarrow Y$
together with a map $m^{-1}(\{v\}) \rightarrow f^{-1}(\{\beta(v)\})$
for each $v \in V$. We then no longer need to require an additional
commutativity condition. If we were working internally in type theory,
we might think of $U$ as a family of types indexed by $V$ and $X$ as a
family of types indexed by $Y$. We would then express the the above
using the type below
\begin{equation}
  \label{eq:piulptt}
  \Sigma_{i : I} \Sigma_{j : J} \Sigma_{\beta : V(i) \rightarrow Y(j)}
  \Pi_{v : V(i)} (U(i, v)
  \rightarrow X(j, \beta(v)))
\end{equation}

We can also formulate this idea in purely category theory terms using
dependent products as follows.

First write $e \colon V \times_I \fhom(V, Y) \rightarrow Y$ for the
right hand map of the hom object (which is the evaluation map of
the local exponential composed with the projection $\pi_1^\ast(Y)
\rightarrow Y$). Write $\tilde{m}$ for the map $\langle m, 1
\rangle \colon U \times_I \fhom(V, Y) \rightarrow V \times_I \fhom(V,
Y)$. Write $q$ for the projection $V \times_I \fhom(V, Y) \rightarrow
\fhom(V, Y)$. 

We view the objects $U \times_I \fhom(V, Y)$ and $e^\ast(X)$ as
objects of $\catc/(V \times_I \fhom(V, Y))$ and then take the
exponential in $\catc/(V \times_I \fhom(V, Y))$ from $U \times_I
\fhom(V, Y)$ to $e^\ast(X)$, corresponding to the term $U(v)
\rightarrow Y(\beta(v))$ in \eqref{eq:piulptt}. By the usual
description of local exponentials in terms of dependent products, we
can express this as $\Pi_{\tilde{m}} \tilde{m}^\ast(e^\ast(X))$. Note
that this comes equipped with an evaluation map $e' \colon
\tilde{m}^\ast \Pi_{\tilde{m}} \tilde{m}^\ast(e^\ast(X)) \rightarrow
\tilde{m}^\ast e^\ast(X)$ over $V \times_I \fhom(V, Y)$, and we have
$\tilde{m}^\ast \Pi_{\tilde{m}} \tilde{m}^\ast(e^\ast(X)) \cong U
\times_I \Pi_{\tilde{m}} \tilde{m}^\ast(e^\ast(X))$.

We then apply $\Pi_q$, corresponding to the $\Pi_{v : V}$ in
\eqref{eq:piulptt} to get
$\Pi_q \Pi_{\tilde{m}} \tilde{m}^\ast(e^\ast(X))$ in
$\catc/ \fhom(V, Y)$. Write $q'$ for the projection
$U \times_I \fhom(V, Y) \rightarrow \fhom(V, Y)$. Then since
$q' = q \circ \tilde{m}$, we have
$\Pi_q \Pi_{\tilde{m}} \tilde{m}^\ast(e^\ast(X)) \cong \Pi_{q'}
\tilde{m}^\ast(e^\ast(X))$

Finally note that we can view this as an object of $\catc/I \times J$
by composing with the map $\fhom(V, Y) \rightarrow I \times J$,
corresponding to the $\Sigma_{\beta : V \rightarrow Y}$ in
\eqref{eq:piulptt}. We have now constructed the indexing object for
the universal lifting problem. We summarise this as the lemma below.

\begin{lemma}
  \label{lem:ulpdepprod}
  The universal lifting problem is isomorphic to a diagram of the
  following form, where $q'$, $\tilde{m}$ and $e$ are as above.
  \begin{equation}
    \label{eq:ulpdepproddiag}
    \begin{gathered}
      \xymatrix{ U \ar[d] & U \times_I \Pi_{q'}
        \tilde{m}^\ast(e^\ast(X)) \ar[l] \ar[r] \ar[d]
        \pullbackcorner[dl] & X \ar[d] \\
        V \ar[d] & V \times_I \Pi_{q'}
        \tilde{m}^\ast(e^\ast(X)) \ar[l] \ar[r] \ar[d] \pullbackcorner[dl]
        & Y \ar[d] \\
        I & \Pi_{q'}
        \tilde{m}^\ast(e^\ast(X)) \ar[r] \ar[l] & J
      }
    \end{gathered}
  \end{equation}

  In type theoretic notation, $\Pi_{q'} \tilde{m}^\ast(e^\ast(X))$ is
  the type we started with in \eqref{eq:piulptt}. Then
  $V \times \Pi_{q'} \tilde{m}^\ast(e^\ast(X))$ contains additionally
  an element $v$ of $V(i)$. Similarly
  $U \times \Pi_{q'} \tilde{m}^\ast(e^\ast(X))$ contains an element of
  $U(i)$. The horizontal right hand maps are then given by
  evaluation.
\end{lemma}

\begin{proof}
  The two descriptions are equivalent by the reasoning preceding the
  lemma.
  
  Working in the internal logic of the category, one can 
  either verify directly using the type theoretic description that this
  satisfies the universal property of the universal lifting problem,
  or show that the description matches the usual definition of the
  universal lifting problem.
\end{proof}

We use the explicit description below to show that for an interesting
class of generating family of left maps, the resulting lawfs is
strongly fibred.

\begin{theorem}
  \label{thm:codlawfsstrfib}
  Suppose that the map $V \rightarrow I$ is an isomorphism, then step
  1 of the small object argument is strongly fibred.

  Moreover, it can be described type theoretically as the type below.
  \begin{equation}
    \label{eq:strfibstep1}
    \Sigma_{j : J} \Sigma_{y : Y(j)} \Sigma_{i : I} (U(i) \rightarrow X(j, y))
  \end{equation}
\end{theorem}

\begin{proof}
  We first give a type theoretic description of the proof since, this
  is easier to follow. We use the description of the universal lifting
  problem in \eqref{eq:piulptt}.

  Since $V \rightarrow I$ is an isomorphism, we may assume that for
  any $i : I$, $V(i)$ is a singleton, that is, isomorphic to the unit
  type $1$. Hence we can replace the term $\beta \colon V(i)
  \rightarrow Y(j)$ with $y \colon Y(j)$, and we can replace the term
  $\Pi_{v : V(i)} (U(i, v) \rightarrow X(j, \beta(v)))$ with $U(i)
  \rightarrow X(j, y)$. This gives the following simplification.
  \begin{equation*}
    \Sigma_{i : I} \Sigma_{j : J} \Sigma_{y : Y(j)} (U(i) \rightarrow X(j, y))
  \end{equation*}
  But this is isomorphic to \eqref{eq:strfibstep1}.

  It's then clear that pulling back along any map
  $g \colon Y' \rightarrow Y$ over $J$ is the same as substituting
  $g(y)$ for $y$, which is the same as pulling back $X$ along $g$, and
  then forming the type, so it is indeed stable under pullback.

  For completeness we now also include a proof that doesn't use any
  type theory.
  
  Recall that over codomain fibrations $\fhom(V, Y)$ is
  $\fhom_{I \times J}(\pi_0^\ast(V), \pi_1^\ast(Y))$ where
  $\fhom_{I \times J}$ is the exponential in $\catc / I \times J$ and
  $\pi_0$ and $\pi_1$ are the projections from $I \times J$ to $I$ and
  $J$. Since $V \rightarrow I$ is an isomorphism we know that
  $\pi_0^\ast(V)$ is the terminal object in $\catc / I \times J$ and
  so $\fhom(V, Y)$ is isomorphic to $\pi_1^\ast(Y)$ with evaluation
  map corresponding to the identity on $\pi_1^\ast(Y)$, which is
  isomorphic to $I \times Y$. So we can take $e$ to be the projection
  $I \times Y \rightarrow Y$.

  For convenience we will now assume that in fact $V = I$ and the map
  $V \rightarrow I$ is the identity. We then have the further
  simplifications that $U \times_I \fhom(V, Y)$ is $U \times Y$, with
  $q'$ equal to $\langle m, 1_Y \rangle$ and $V \times_I \fhom(V, Y)$
  is $I \times Y$, with $\tilde{m}$ equal again to
  $\langle m, 1_Y \rangle$. We will write both $q$ and $\tilde{m}$ as
  $\tilde{m}_Y$.

  In summary we can rewrite the universal lifting problem as below.
  \begin{equation*}
    \begin{gathered}
      \xymatrix{ U \ar[d]_m & U \times_I \Pi_{\tilde{m}_Y}
        \tilde{m}_Y^\ast(I \times X) \ar[l] \ar[r] \ar[d] & X \ar[d] \\
        I \ar@{=}[d] & \Pi_{\tilde{m}_Y} \tilde{m}_Y^\ast(I \times X) \ar@{=}[d]
        \ar[r] \ar[l] & Y \ar[d] \\
        I & \Pi_{\tilde{m}_Y} \tilde{m}_Y^\ast(I \times X) \ar[l] \ar[r] & J}
    \end{gathered}
  \end{equation*}

  Now suppose we are given $f' \colon X' \rightarrow Y'$ and maps
  making the following pullback square.
  \begin{equation*}
    \begin{gathered}
      \xymatrix{ X' \ar[r] \ar[d]_{f'} \pullbackcorner & X \ar[d]^f \\
        Y' \ar[r] & Y}
    \end{gathered}
  \end{equation*}

  Consider the following cube, where the left and right faces are
  pullbacks by definition, and the upper back edge is given by the
  universal property of the pullback on the right.
  \begin{equation*}
    \xymatrix@=1.5em{m_{Y'}^\ast (I \times X') \ar[rr] \ar[dr] \ar[dd] &
      & \tilde{m}_Y^\ast (I \times X) \ar'[d][dd] \ar[dr] & \\
      & I \times X' \ar[rr] \ar[dd] & & I \times X \ar[dd] \\
      U \times Y' \ar'[r][rr] \ar[dr]_{m_{Y'}} & & U \times Y
      \ar[dr]^{\tilde{m}_Y} & \\
      & I \times Y' \ar[rr] & & I \times Y}
  \end{equation*}
  The front face is a pullback by diagram chasing, and we deduce that
  the back face is also a pullback.
  
  Since dependent products in locally cartesian closed categories
  always satisfy Beck-Chevalley, we can deduce that we also have the
  following pullback.
  \begin{equation*}
    \xymatrix{\Pi_{m_{Y'}}m_{Y'}^\ast(I \times X') \ar[r] \ar[d]
      \pullbackcorner & \Pi_{m_{Y}}m_{Y}^\ast(I \times X) \ar[d] \\
      I \times Y' \ar[r] & I \times Y
    }
  \end{equation*}
  We can now see that pulling back along $Y' \rightarrow Y$ preserves
  the universal lifting problem, and so it also preserves step-one of
  the small object argument, which is just a pushout of the upper
  right square in the universal lifting problem (and pullbacks in
  locally cartesian closed categories always preserve pushouts since
  they are left adjoints). Since step-one is always fibred in a locally
  cartesian closed category, we can now deduce by lemma
  \ref{lem:strongfibrelem} that it is strongly fibred.
\end{proof}

\begin{corollary}
  Suppose that $m$ is as above the map $V \rightarrow I$ is an
  isomorphism. Then the cofibrantly generated awfs is strongly fibred
  if it exists.
\end{corollary}

\begin{proof}
  By theorems \ref{thm:codlawfsstrfib} and
  \ref{thm:algfreeawfsstrfib}.
\end{proof}

\begin{remark}
  One might expect fibred and strongly fibred to be equivalent for
  functorial factorisations over $\cod$, since both involve
  pullbacks. However, this is not the case. A functorial factorisation
  takes as input a family of maps over an object $J$ as below.
  \begin{equation*}
    \xymatrix{ X \ar[dr] \ar[rr]^f & & Y \ar[dl] \\
      & J &}
  \end{equation*}
  It then factorises $f$ as $L f$ followed by $R f$.

  For the factorisation to be fibred says that it is stable under
  pullback along maps into $J$. We have seen that for step-one this is
  always the case. On the other hand, strongly fibred says that the
  factorisation is stable under pullback along all maps into $Y$. We
  have seen an interesting special case when step-one is strongly
  fibred, but it is not the case in general.
\end{remark}

\subsubsection{Individual Morphisms}
\label{sec:individual-morphisms}

We will illustrate how lifting problems work over codomain fibrations
by first applying the definitions when we are just given individual
maps as input (i.e. families of maps over the terminal object). This
case can also be understood via enriched lifting problems where we
view $\catc$ as enriched over itself with cartesian product (as
appears, for example in \cite[Section 13.3]{riehlcht}), but it
illustrates similar ideas that turn up in the general case of families
of maps in codomain fibrations, where it is not so clear how the
definitions relate to existing notions in homotopical algebra.

Let $i \colon U \rightarrow V$ and $f \colon X \rightarrow Y$ be
morphisms of $\catc$. Then we can view them as morphisms in
$\catc^\btwo$ in the fibre of $1$. A family of lifting problems
consists of an object $Z$ of $\catc$ together with a commutative
square of the following form in $\catc / Z$.
\begin{equation*}
  \begin{gathered}
    \xymatrix{ Z \times U \ar[rr] \ar[d]_{\langle 1, i \rangle} & & Z
      \times X \ar[d]^{\langle 1, f \rangle} \\
      Z \times V \ar[rr] \ar[dr] & & Z \times Y \ar[dl] \\
      & Z & }
  \end{gathered}
\end{equation*}

Then, either using the adjunction between composition and pullback, or
by checking directly, finding a solution to the family of lifting
problems above is equivalent to finding a diagonal filler of the
following square in $\catc$ (or equivalently $\catc / 1$).
\begin{equation*}
  \begin{gathered}
    \xymatrix{ Z \times U \ar[r] \ar[d]_{\langle 1, i \rangle} & X
       \ar[d]^{f} \\
      Z \times V \ar[r] & Y \\
    }
  \end{gathered}
\end{equation*}

Recall that for codomain fibrations, $\fhom_I(A, B)$ is just the local
exponential in $\catc/I$. In particular $\fhom_1(A, B)$ is just the
exponential $B^A$ in $\catc$.

For the example above, we know by proposition \ref{prop:univlp} that a
filler for every family of lifting problems corresponds to a solution
of the universal lifting problem and in turn to a coherent choice of
solutions to lifting problems. Here coherence says that for any
$g \colon Z' \rightarrow Z$, the triangle in the middle of the diagram
below commutes, where the diagonals are the choices of solutions.
\begin{equation*}
  \begin{gathered}
    \xymatrix{ Z' \times U \ar[r] \ar[d] & Z \times U \ar[r] \ar[d] &
      X \ar[d] \\
      Z' \times V \ar[r] \ar[urr] & Z \times V \ar[r] \ar[ur] & Y}
  \end{gathered}
\end{equation*}

The indexing object of the universal family is $X^U \times_{Y^U} Y^V$.

Then a solution to the universal lifting problem corresponds to a
solution of the following lifting problem in $\catc$, where the top
and bottom maps are given by evaluating the appropriate exponentials.
\begin{equation*}
  \begin{gathered}
    \xymatrix{ X^U \times_{Y^U} Y^V \times U \ar[r] \ar[d]_{\langle 1,
        i \rangle} & X
      \ar[d]^f \\
      X^U \times_{Y^U} Y^V \times V \ar[r] & Y}
  \end{gathered}
\end{equation*}

\subsubsection{Strongly Fibred Cofibrations and Pushout
  Product}
\label{sec:class-class-monom}

We now consider a general construction that generates a strongly
fibred lawfs $(C_1, F^t_1)$ a fibred (but usually not strongly fibred)
lawfs $(C^t_1, F_1)$ together with a morphism of lawfs
$\xi \colon (C_1, F^t_1) \rightarrow (C^t_1, F_1)$. We will show that
this generalises certain constructions considered by Van den Berg and
Frumin in \cite{vdbergfrumin} and by Pitts and Orton in
\cite{pittsortoncubtopos}.

Let $\catc$ be a locally cartesian closed category with pushouts and
disjoint coproducts.

We suppose that we are given an interval object $\intv$ with endpoints
$\delta_0, \delta_1 \colon 1 \rightarrow \intv$ together with a family
of maps of the following form.
\begin{equation}
  \label{eq:11}
  \begin{gathered}
    \xymatrix{ 1 \ar[rr]^{i_0} \ar[dr]_{i_0} & & I \ar@{=}[dl] \\
      & I & }
  \end{gathered}
\end{equation}

Write $\mathbf{\Sigma}$ for the class of maps that are pullbacks of
$i_0$.

We first remark that we get the following explicit description of
right maps over $1$, as a special case of proposition
\ref{prop:codfibredlpasordinary}.
\begin{proposition}
  Let $f$ be a morphism of $\catc$. We also view $f$ as a map in
  $\catc/1$. Then $f$ has the (fibred) right lifting property against
  $i_0$ if and only if it has the (ordinary) right lifting property
  against $m$ for every $m \in \mathbf{\Sigma}$.
\end{proposition}

As we saw in theorem \ref{thm:codlawfsstrfib} the fact that the right
hand map is an identity means that we get a more explicit type
theoretic description of the universal lifting problem, and that step
1 is strongly fibred.

Since the object on the left is the terminal object, we can in fact
further reduce the type theoretic description even further in this
case. We think of maps $1 \rightarrow U$ as terms of type $\Sigma_{i :
  I} U(i)$ in the empty context. Since this map is also the display
map for $U$ over $I$, we see that we can in fact think of this as a
term $i_0$ of type $I$ such that the family of types $U(i)$ over $i :
I$ is defined by $U(i) :=\, i = i_0$.

However, this implies that the top right horizontal map in the
universal lifting problem \eqref{eq:ulpdepproddiag} is an
isomorphism. Explicitly, this map is the evaluation map
$\Sigma_{i : I} (i = i_0 \;\times\; (i = i_0 \rightarrow X(j, y)))
\rightarrow X(j, y)$. Its inverse is then the map sending
$x : X(j, y)$ to $(i_0, \mathtt{refl}, \lambda u.x)$.

But step-one of the small object argument is defined via pushout, and
pushouts preserve isomorphisms, so we can deduce the following.
\begin{theorem}
  \label{thm:codclassmapstep1strfib}
  Suppose we are given a family of maps as in \eqref{eq:11}. Then step
  1 of the small object argument (which we refer to as $(C_1, F^t_1)$)
  is isomorphic to the following type, and it is strongly fibred.
  \begin{equation*}
    \Sigma_{j : J} \Sigma_{y : Y(j)} \Sigma_{i : I} (i = i_0 \rightarrow X(j, y))
  \end{equation*}
\end{theorem}

\begin{proposition}
  \label{prop:classifiedmonosequivdef}
  The following are equivalent.
  \begin{enumerate}
  \item Every element of $\mathbf{\Sigma}$ is a pullback of $i_0$ in a
    unique way.
  \item $i_0$ is the terminal object of the category with objects the
    elements of $\mathbf{\Sigma}$ and morphisms pullback squares.
  \item The following ``propositional extensionality'' principle holds
    in the internal logic: $\forall i, i' \in I\,(i = i_0
    \Leftrightarrow i' = i_0) \Rightarrow i = i'$.
  \end{enumerate}
  If $\catc$ has a subobject classifier,
  $\top \colon 1 \rightarrow \Omega$, then these are equivalent to $I$
  being a subobject of $\Omega$, with $i_0$ the pullback of $\top$
  along the subobject inclusion.
\end{proposition}

\begin{definition}
  \label{def:classifiedmono}
  If $\mathbf{\Sigma}$ satisfies one of the equivalent conditions in
  proposition \ref{prop:classifiedmonosequivdef}, we say it is
  \emph{extensional}.
\end{definition}

\begin{proposition}
  If $\mathbf{\Sigma}$ is extensional then the $L_1$ coalgebra
  structure on a map $m$ is unique (if it exists).
\end{proposition}

\begin{proof}
  This is straightforward by using the type theoretic definitions and
  then working in the internal logic of the
  category.
\end{proof}

\begin{theorem}
  If we are given a natural way to compose elements of
  $\mathbf{\Sigma}$, then we can assign $(C_1, F^t_1)$ a
  multiplication map making it into an awfs.
\end{theorem}

\begin{proof}
  By (the dual of) lemma \ref{lem:comptoawfs}.
\end{proof}

\begin{example}
  \label{ex:classifiedcof}
  Suppose that $\mathbf{\Sigma}$ is
  closed under composition and extensional.

  In this case we end up with an identical situation to the one
  considered by Bourke and Garner in \cite[Section
  4.4]{bourkegarnerawfs1}, and indeed the theorems above are minor
  variants of those considered by Bourke and Garner.

  Gambino and Sattler proved in \cite[Lemma 9.7]{gambinosattlerpi}
  that if $\catc$ is a presheaf category, then there is a suitable
  such $i_0$ for any class $\mathbf{\Sigma}$ of monomorphisms closed
  under pullback and composition.
\end{example}

We now define the second lawfs $(C^t_1, F_1)$ to be the one generated
by the coproduct of the two families of maps below.
\begin{equation}
  \label{eq:26}
  \begin{gathered}
    \xymatrix{ \intv +_1 I \ar[rr]^{\delta_0 \hat \times i_0} \ar[dr]
      & & \intv
      \times I \ar[dl] \\
      & I & }
  \end{gathered}
  \qquad
  \begin{gathered}
    \xymatrix{ \intv +_1 I \ar[rr]^{\delta_1 \hat \times i_0} \ar[dr]
      & & \intv
      \times I \ar[dl] \\
      & I & }
  \end{gathered}
\end{equation}
In many natural examples the interval comes equipped with a symmetry
operation swapping the two endpoints. In this case, we clearly only
need to use one of the diagrams above to get the same class of maps
with right lifting property.

Again we get simple description of maps with the right lifting
property.
\begin{proposition}
  \label{prop:codfibppchar}
  A map $f$ in $\catc$, viewed as a map in $\catc/1$ has the (fibred)
  right lifting property against the coproduct of the maps in
  \eqref{eq:26} if and only if it has the (ordinary) right lifting
  property against $\delta_k \hat \times m$ for all $m \in
  \mathbf{\Sigma}$ and $k \in \{0, 1\}$.
\end{proposition}

\begin{proof}
  By propositions \ref{prop:codfibredlpasordinary} and
  \ref{prop:fibpushoutproddescr}.
\end{proof}

\begin{example}
  Suppose that $i_0$ is terminal in the category of pullback squares
  and $\mathbf{\Sigma}$ is closed under composition, as in example
  \ref{ex:classifiedcof}.

  Suppose further that $\catc$ is a topos, that the interval object
  comes equipped with connections, and the endpoint inclusions
  $\delta_0$ and $\delta_1$ are disjoint.

  Suppose further that elements of $\mathbf{\Sigma}$ are closed under
  finite union, and contain the map $[\delta_0, \delta_1] \colon 1 + 1
  \rightarrow \intv$.

  This is now the situation considered by Van den Berg and Frumin in
  \cite{vdbergfrumin}.
  
  By proposition \ref{prop:codfibppchar} we see that the class of maps
  with the fibred right lifting property against $i_0$ are precisely the
  class of Kan fibrations that appear as \cite[Definition
  3.3]{vdbergfrumin}.

  We also recover the notion of \emph{filling structure} due to Orton
  and Pitts in \cite[Definition 4.2]{pittsortoncubtopos} as
  follows. Firstly, we only use one of the maps in \eqref{eq:26},
  $\delta_0$. This means we are considering the universal lifting
  problem of $i_0 \hat{\times} \delta_0$ against $f$. Then one can
  show by working internally in type theory that solutions to the
  universal lifting problem correspond precisely to filling
  structures. One might worry that pushout product refers to pushout
  which requires quotients to define in type theory, which don't
  appear in the work of Orton and Pitts. An explanation for this is
  that one can, by proposition \ref{prop:fibpushoutproddescr}, instead
  consider solutions of the universal lifting problem of $i_0$ against
  $\widehat{\{\delta_0, f\}}$, which can be defined just using
  pullback and exponentials and more closely matches the Pitts-Orton
  definition.

  We will see later in section \ref{sec:furth-gener} that it is also
  possible to formulate the Pitts-Orton definition of
  \emph{composition structure} in a similar way.
\end{example}

\begin{example}[van den Berg, Frumin]
  \label{ex:efftopos}
  As a special case of the previous example, we can take $\catc$ to be
  the effective topos, the interval $\intv$ to be $\nabla 2$ and
  $\mathbf{\Sigma}$ to be the class of all monomorphisms. This gives a
  nontrivial model structure on a subcategory of the effective
  topos. See \cite{vdbergfrumin} for details.
\end{example}

\begin{example}
  \label{ex:endpointgencof}
  Suppose we are just given an interval object
  $\delta_0, \delta_1 \colon 1 \rightarrow \intv$. Then we can take
  $I := \intv$ and $i_0 := \delta_1$.
\end{example}

\subsubsection{Trivial Fibrations in $01$-Substitution Sets}
\label{sec:triv-fibr-01}

Recall that Pitts in \cite{pittsnompcs} defined the category of
$01$-substitution sets as an equivalent category to
the category of cubical sets studied by Bezem Coquand and Huber in
\cite{bchcubicalsets}. See also \cite[Section 1.2]{swannomawfs} for a
description of the category and the definition of Kan fibration. In a
later paper \cite{bchunivalence}, Bezem, Coquand and Huber returned to
this category of cubical sets and showed that it has a univalent
universe. One of the ideas that they developed in that paper was a
notion of cofibration and trivial fibration in the category of cubical
sets.

In this section we will define the corresponding notion of trivial
fibration in $01$-substitution sets, and in fact define it as a
cofibrantly generated class over the codomain fibration. We will
assume the reader is familiar with nominal sets. See
\cite{pittsnomsets} for a general introduction.

We first define a $01$-substitution set $\powfin(\names) + 1$ as
follows. The underlying nominal set is the usual definition of
$\powfin(\names) + 1$ in nominal sets (recalling $\powfin(\names)$ is the
nominal set of finite subsets of $\names$ with the pointwise action).
We write the unique element of
$1$ as $\top$. We define the action of substitutions by
$x (a := i) := \top$ for all $x \in V$, $a \in \names$ and $i \in
2$. We also write $\top$ for the coproduct inclusion
$1 \rightarrow \powfin(\names) + 1$ (which we note is a morphism in
$\sub$, not just in nominal sets).

We consider the lawfs cofibrantly generated by the following family of
maps:
\begin{equation*}
  \xymatrix{ 1 \ar[rr]^\top \ar[dr]_\top & & \powfin(\names) + 1 \ar@{=}[dl] \\
    & \powfin(\names) + 1 & }
\end{equation*}

\begin{proposition}
  $\powfin(\names) + 1$ is extensional (in the sense of definition
  \ref{def:classifiedmono}).
\end{proposition}

\begin{proof}
  We use the condition that ``propositional extensionality'' holds.

  Suppose we are given $A, B \in \powfin(\names) + 1$ such that $A =
  \top \Leftrightarrow B = \top$ holds in the internal logic. First
  note that this clearly rules out $A = \top$ and $B \in
  \powfin(\names)$. Similarly for $A \in \powfin(\names)$ and $B =
  \top$. When $A = \top$ and $B = \top$ we vacuously have $A = B$.

  It only remains to consider the case where both $A$ and $B$ are
  elements of $\powfin(\names)$. For all $a \in \names$ we have that
  $A(a := 0) = \top$ if and only if $B(a := 0) = \top$. But this
  implies $a \in A$ if and only if $a \in B$, and so $A = B$ as
  required.
\end{proof}

Let $\square_A$ be the image of a representable under Pitts'
equivalence between $01$-substitution sets and cubical sets, as
defined in \cite[Section 5.1]{swannomawfs}. Note that maps
$\square_A \rightarrow \powfin(\names) + 1$ correspond precisely to elements of $\powfin(\names) + 1$ for
which $A$ is a support. Such an element is either $\top$, or of the
form $A_1$ where $A_1 \subset A$. Write $A_2$ for $A \setminus
A_1$. Then we have $\square_A \cong \square_{A_1} \otimes
\square_{A_2}$. Write $\partial \square_{A_1}$ for the subobject of
$\square_{A_1}$ with elements $\sigma \colon A_1 \rightarrow 2$ such
that $\sigma(a) \in 2$ for some $a \in A_1$. Then the pullback of
$\top$ is of the following form.
\begin{equation}
  \label{eq:8}
  \begin{gathered}
  \xymatrix{ \partial \square_{A_1} \otimes \square_{A_2}
    \pullbackcorner \ar[r] \ar[d] & 1 \ar[d]^\top \\
    \square_{A_1} \otimes \square_{A_2} \ar[r] & \powfin(\names) + 1}    
  \end{gathered}
\end{equation}

Suppose we are given a map $\sigma \colon \square_{A'} \rightarrow
\square_{A}$. We say it is non degenerate if the map $\square_{A'}
\rightarrow \powfin(\names) + 1$ is non trivial. In this case we must have $\square_{A'}
\cong \square_{A_1} \otimes \square_{A_2'}$ for some $A_2'$ where
$\sigma$ is an automorphism of $\square_{A_1}$. We then have a
pullback of the following form.
\begin{equation}
  \label{eq:5}
  \begin{gathered}
  \xymatrix{ \partial \square_{A_1} \otimes \square_{A_2}
    \pullbackcorner \ar[r] \ar[d] & \partial \square_{A_1} \otimes
    \square_{A_2'} \ar[d] \\
    \square_{A_1} \otimes \square_{A_2} \ar[r] & \square_{A_1} \otimes
    \square_{A_2'}}    
  \end{gathered}
\end{equation}

Then we see that the fibred right lifting property against $\top$
gives us a choice of filler for the (ordinary) lifting
problem against each map on the left of \eqref{eq:8}, subject to the
compatibility condition given by \eqref{eq:5}. However, since the
category of $01$-substitution sets is equivalent to a presheaf
category, the converse also holds, by a similar argument to
\cite[Theorem 9.1]{gambinosattlerpi}.

By the same reasoning as in \cite[Remark 3.9]{huberthesis} or
\cite[Section 5]{swannomawfs}
we see that maps $f \colon X \rightarrow Y$ with this
property against such maps correspond
precisely with those with a \emph{boundary filling operator}, defined
as follows.

\begin{definition}
  Let $f \colon X \rightarrow Y$. A \emph{boundary} (or
  \emph{tube})
  over a finite set $A \subseteq \names$ consists of an element $y$ of
  $Y$ together with a function $u \colon A \times 2 \rightarrow X$
  satisfying the following conditions for all $a \in A$ and $i
  \in 2$.
  \begin{enumerate}
  \item $a \,\#\, u(a, i)$
  \item $f(u(a, i)) = y(a := i)$
  \item For $a' \in A$ such that $a' \neq a$ and for $i' \in 2$, $u(a,
    i)(a' := i') = u(a', i')(a := i)$
  \end{enumerate}

  A \emph{filler} of a boundary $(u, y)$ consists of an element $x$ of
  $X$ such that $f(x) = y$ and for all $a \in A$ and $i \in 2$,
  $u(a, i) = x(a := i)$.

  A \emph{boundary filling operator} on $f$ consists of a choice of
  filler ${\uparrow} (u, y)$ for every boundary satisfying the
  following.
  \begin{enumerate}
  \item For all $\pi \in \perma$, $\pi \cdot ({\uparrow} (u, y)) =
    {\uparrow} (\pi \cdot u, \pi \cdot y)$.
  \item For all $a \in \names \setminus A$ and $i \in 2$, we have
    ${\uparrow} (u, y) (a := i) = {\uparrow} (u(a := i), y (a := i))$.
  \end{enumerate}
\end{definition}

% \begin{proposition}
%   The map
%   $\partial \square_{A_1} \otimes \square_{A_2} \rightarrow
%   \square_{A_1} \otimes \square_{A_2}$ is isomorphic to the pushout
%   product of the inclusion
%   $\partial \square_{A_1} \hookrightarrow \square_{A_1}$ with the
%   unique map $0 \rightarrow \square_{A_2}$.
% \end{proposition}

% \begin{proof}
%   By the general properties of pushout product.
% \end{proof}

\subsubsection{Trivial Fibrations and Fibrations in CCHM Cubical Sets}
\label{sec:triv-fibr-fibr}

We will define two classes of maps in the
Cohen-Coquand-Huber-M\"{o}rtburg category of cubical sets from
\cite{coquandcubicaltt}. Following Gambino and Sattler in
\cite{gambinosattlerpi}, Van den Berg and Frumin in
\cite{vdbergfrumin} and Orton and Pitts in \cite{pittsortoncubtopos},
we do this as a special case of the construction in section
\ref{sec:class-class-monom}.

We first recall the definition of CCHM cubical sets from
\cite{coquandcubicaltt}.
\begin{definition}
  A \emph{de Morgan algebra} is a bounded distributive lattice $\langle
  L, \wedge, \vee \rangle$ with top
  element, $1$, and bottom element $0$, and an involution $\neg \colon
  L \rightarrow L$ satisfying the following for all $r, s \in L$:
  \begin{equation*}
    \neg 1 \,=\, 0 \qquad \neg 0 \,=\, 1 \qquad \neg (r \vee s) \,=\, (\neg r)
    \wedge (\neg s) \qquad \neg (r \wedge s) \,=\, (\neg r) \vee (\neg s)
  \end{equation*}

  The forgetful functor from de Morgan algebras to $\set$ is monadic,
  so we can alternatively view de Morgan algebras as algebras over a
  monad on $\set$ which we denote $\dmmod$.
\end{definition}

\begin{definition}
  Fix a countably infinite set, $\names$. (When working constructively
  also assume that $\names$ has decidable equality.)
  
  The \emph{category of cubes}, $\smcat{C}$ is the full (small)
  subcategory of the Kleisli category of $\dmmod$ on finite subsets of
  $\names$.

  The category of CCHM cubical sets is the functor category
  $\set^\smcat{C}$.
\end{definition}

\begin{definition}
  We define the interval object $\intv$ as the cubical set defined by
  $\intv(A) := \dmmod(A)$. Alternatively, this is the canonical map
  from the Kleisli category to $\xalg{\dmmod}$ composed with the
  forgetful functor $\xalg{\dmmod} \rightarrow \set$.

  Alternatively again, this is also isomorphic to the representable
  $\yoneda \{a\}$ for $a \in \names$.

  The endpoints $\delta_0, \delta_1 \colon 1 \rightarrow \intv$ are
  given by $0$ and $1$ respectively in the de Morgan algebras.
\end{definition}

We now just need to define the classifying map for the
cofibrations. This is what Cohen, Coquand, Huber and M\"{o}rtburg
refer to as the \emph{face lattice}. We will denote it $1 \rightarrow
F$, and give three alternative definitions. Two abstract, and one more
concrete (based on two definitions from \cite{coquandcubicaltt}).

First, note that one might be tempted to take $F = \intv$, and
$1 \rightarrow F$ to be be one of the endpoint inclusions, say
$\delta_1 \colon 1 \rightarrow \intv$. This does generate a pair of
lawfs's (as we saw in example \ref{ex:endpointgencof}), but it is not
extensional (in the sense of definition \ref{def:classifiedmono}). To
see this note that for all
$\sigma \colon \{a\} \rightarrow \dmmod \{a\}$, we have
$\sigma(a \wedge \neg a) \neq 1$, so $\sigma(a \wedge \neg a) = 1$ if
and only if $0 = 1$, but $a \wedge \neg a \neq 0$. Hence, one might
motivate the definition of face map, by ``making $\delta_1$
extensional efficiently as possible.''

\begin{definition}
  (This appears at the bottom of \cite[Section
  8.1]{coquandcubicaltt}). We define $1 \rightarrow F$ as follows. Let
  $\chi \colon \intv \rightarrow \Omega$ be the classifying map for
  the monomorphism $\delta_1$. Let
  $\intv \twoheadrightarrow F \rightarrowtail \Omega$ be the image
  factorisation of $\chi$. Define the classifying map $1 \rightarrow
  F$ to be the pullback of $\top \colon 1 \rightarrow \Omega$ along
  the inclusion $F \rightarrowtail \Omega$.
\end{definition}

The construction above requires the existence of a subobject
classifier. This is sometimes rejected in constructive mathematics for
predicativity reasons (see \cite{moerdijkpalmgrenast1}). Hence we
give the following alternative
definition, which works more generally in any $\Pi$-pretopos (but is
equivalent to the definition above in a topos).
\begin{definition}
  We define an equivalence relation, $\sim$ on $\intv$ using the
  following definition in the internal logic. For $i, i' \in \intv$,
  we set $a \sim b$ if $(a = 1) \Leftrightarrow (b = 1)$. We define
  $F$ to be the quotient $\intv / {\sim}$. We define $1 \rightarrow F$
  to be the composition
  $1 \stackrel{\delta_1}{\rightarrow} \intv \twoheadrightarrow \intv
  /{\sim}$.
\end{definition}

Finally, we recall from \cite[Section 4.1]{coquandcubicaltt}, that we
can also give the following more concrete, syntactic definition of the
face lattice.
\begin{definition}
  Given a finite set $A$, we define $F(A)$ to be the distributive
  lattice generated by $A + A$, subject to the following
  relation. Write the elements of $A + A$ as $(a = 0)$ and $(a = 1)$
  for $a \in A$. The relation is $(a = 0) \wedge (a = 1) = 0$.
\end{definition}

Now similarly to \cite[Example 9.3]{gambinosattlerpi}, \cite[Example
3.1(2)]{vdbergfrumin} and \cite{pittsortoncubtopos}, we can
characterise the CCHM notion of Kan filling operator as follows. Given
a morphism $f \colon X \rightarrow Y$ in cubical sets, a \emph{Kan
  filling operator} is a solution to the universal lifting problem to
$f$ from the following family of maps.
\begin{equation*}
  \begin{gathered}
    \xymatrix{ \intv +_1 F \ar[rr]^{\delta_1 \hat \times \top} \ar[dr]
      & & \intv
      \times F \ar[dl] \\
      & F & }
  \end{gathered}  
\end{equation*}

In fact, we can now characterise Kan filling operators as cofibrantly
generated in two senses, since, as Gambino and Sattler show in
\cite{gambinosattlerpi},
they are also algebraically cofibrantly generated in Garner's sense.

\section{A Further Generalisation: Lifting Problems for Squares}
\label{sec:furth-gener}

In this section we consider more general notion of lifting problem due
to Sattler \cite[Section 6]{sattlermodelstructures} and show how the
earlier results in this paper can be adapted to also work with this
definition.

\subsection{Definition}
\label{sec:definition}

We first give a fibred version of Sattler's definition. Throughout, we
assume that we are given a fibration $p \colon \cat{E} \rightarrow
\cat{B}$. We will recover Sattler's definition by applying this to a
category indexed family fibration.

\begin{definition}
  A \emph{family of squares} over $I \in \cat{B}$ is a commutative
  square in $\cat{E}_I$.
\end{definition}

\begin{proposition}
  Equivalently, a family of squares is a commutative square in
  $\cat{E}$ where all maps are vertical, or a vertical map in
  $\vrt{\cat{E}}$, or an object of $\vrt{\vrt{\cat{E}}}$.
\end{proposition}

\begin{definition}
  Suppose we are given a family of squares over $I \in \cat{B}$ and a
  family of squares over $J \in \cat{B}$, as below.
  
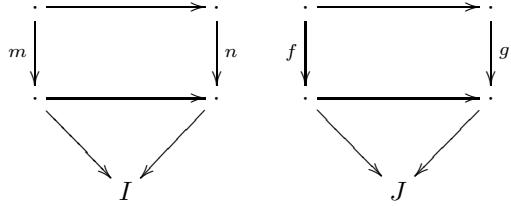
\begin{equation*}
    \begin{gathered}
      \xymatrix{ \cdot \ar[d]_{m} \ar[rr] & & \cdot \ar[d]^{n} \\
        \cdot \ar[rr] \ar[dr] & & \cdot \ar[dl] \\
      & I & }
    \end{gathered}
    \quad
    \begin{gathered}
      \xymatrix{ \cdot \ar[d]_{f} \ar[rr] & & \cdot \ar[d]^{g} \\
        \cdot \ar[rr] \ar[dr] & & \cdot \ar[dl] \\
      & J &}
    \end{gathered}
  \end{equation*}

  A \emph{family of lifting problems} from $(m, n)$ to $(f, g)$ is an
  object $K$ of $\cat{B}$ together with maps
  $\sigma \colon K \rightarrow I$ and $\tau \colon K \rightarrow J$
  and a map from $\sigma^\ast(n)$ to $\tau^\ast(f)$ in
  $\vrt{\cat{E}}$, or equivalently, the middle square in the
  diagram below in $\cat{E}_K$. A \emph{solution}, or \emph{family of
    fillers} of the lifting problem is the dotted diagonal map in the
  diagram below making two commutative triangles.
  
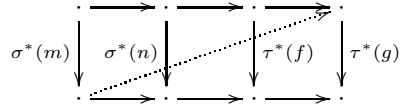
\begin{equation*}
    \xymatrix{ \cdot \ar[d]_{\sigma^\ast(m)} \ar[r] & \cdot
      \ar[d]_{\sigma^\ast(n)} \ar[r] &
      \cdot \ar[d]^{\tau^\ast(f)} \ar[r] & \cdot \ar[d]^{\tau^\ast(g)} \\
      \cdot \ar[r] \ar@{.>}[urrr] & \cdot \ar[r] & \cdot \ar[r] &
      \cdot}
  \end{equation*}
\end{definition}

Note that a lifting problem of the square $(m, n)$ against the
square $(f, g)$ is exactly a lifting problem of $n$ against $f$.
\begin{definition}
  A lifting problem of $(m, n)$ against $(f, g)$ is a \emph{universal
    lifting problem} if it is universal as a lifting problem of $n$
  against $f$.
\end{definition}

Since this is a special case of our earlier definition, we immediately
see that universal lifting problem is unique up to isomorphism and
that the universal lifting problem exists whenever $p$ is locally
small and $\cat{B}$ has all finite limits.

We easily get the counterpart to proposition \ref{prop:univlp} as
below.
\begin{proposition}
  \label{prop:univlpsq}
  Let $\cat{B}$ have finite limits, let $p$ be a locally small
  fibration and let $(m, n)$ and $(f, g)$ be families of squares. Then
  the following are equivalent.
  \begin{enumerate}
  \item Every family of lifting problems from $(m, n)$ to $(f, g)$ has
    a solution.
  \item The universal family of lifting problems from $(m, n)$ to
    $(f, g)$ has a solution.
  \item There is a coherent choice of solutions to all families of
    lifting problems from $(m, n)$ to $(f, g)$.
  \end{enumerate}
\end{proposition}

Just as in loc. cit. we note that every family of maps can be viewed
as a family of squares. We can show this succinctly by working over
$\vrt{\cat{E}}$: every object $f$ of $\vrt{\cat{E}}$ gives
us a vertical map $1_f$ of $\vrt{\cat{E}}$. Again, following
Sattler, we will focus on lifting problems where the right hand map is
of this form. This is most useful in practice since usually what we
are interested in is algebraic structure on maps cofibrantly generated
by squares.

We can carry out step-one of the small object argument as follows. This
time the construction only gives a functorial factorisation rather
than a lawfs in general.

As before, it is easiest to work in $\vrt{\cat{E}}$, over
the composition of bifibrations, $\vrt{\cat{E}}
\stackrel{\dom}{\longrightarrow} \cat{E} \stackrel{p}{\longrightarrow}
\cat{B}$.

We are given a generating square $m \rightarrow n$ in
$\vrt{\cat{E}}$ over $I \in \cat{B}$ and an object $g$ of
$\vrt{\cat{E}}$ over $J \in \cat{B}$. We first form the
universal lifting problem. Say that
$h \colon \fhom(n, g) \rightarrow I$. This gives us a diagram of the
form below.
\begin{equation*}
  \xymatrix{ m \ar[r] & n & & g \\
    & h^\ast(m) \ar[ul] \ar[r] & h^\ast(n) \ar[ul] \ar[ur] }
\end{equation*}

We then factorise the map $h^\ast(m) \rightarrow g$ as an opcartesian
map over $\dom$ followed by a vertical map over $\dom$. By lemma
\ref{lem:domopcartchar} this is the same as taking the following
pushout in $\cat{E}$ (where the right hand square is the universal
lifting problem). In particular we know that the pushout exists and
can also be expressed as a levelwise opcartesian map followed by a
vertical pushout.
\begin{equation*}
  \xymatrix{ \cdot \ar[r] \ar[dd]_m & \cdot
    \ar[rr] \ar[dd]_n & & \cdot \ar[dl] \ar[dd]^g \\
    & & \cdot \pushoutcorner \ar[dr] & \\
    \cdot \ar[r] \ar[urr] & \cdot \ar[rr] & & \cdot
  }
\end{equation*}
However, in this form it is clear that we get a factorisation of $g$,
which in fact gives a functorial factorisation such that algebras over
the corresponding pointed endofunctor correspond precisely to
solutions of the universal lifting problem.

\subsection{Squares over a Codomain Fibration}

We now specialise to codomain fibrations. In this case a family of
squares indexed by $I$ is a diagram of the following form.
\begin{equation}
  \label{eq:codsquare}
  \begin{gathered}
    \xymatrix{ U_0 \ar[rr] \ar[d] & & U_1 \ar[d] \\
      V_0 \ar[rr] \ar[dr] & & V_1 \ar[dl] \\
      & I &
    }
  \end{gathered}
\end{equation}

\begin{lemma}
  Suppose that the map $V_1 \rightarrow I$ in \eqref{eq:codsquare} is
  an isomorphism. Then,
  \begin{enumerate}
  \item The functorial factorisation on \eqref{eq:codsquare} is
    strongly fibred.
  \item If the algebraically free rawfs on the functorial
    factorisation exists, then it is also strongly fibred.
  \end{enumerate}
\end{lemma}

\begin{proof}
  By an easy argument in the internal logic similar to that of theorem
  \ref{thm:codlawfsstrfib}, and then applying theorem
  \ref{thm:algfreerawfsstrfib} for showing the awfs is also fibred
  (when it exists).
\end{proof}

\begin{example}
  As Sattler shows in \cite[Section 6]{sattlermodelstructures}, this
  notion of lifting problem can be used to define Kan composition in
  CCHM cubical sets. Combining this with our earlier remarks, we see
  that the category of maps with Kan composition operator is
  cofibrantly generated by the following family of squares.
  \begin{equation*}
    \xymatrix{ 1 \ar[rr] \ar[d]_{\top} & & \intv +_1 F
      \ar[d]^{\top \hat{\times} \delta_0} \\
      F \ar[rr]^{\langle \delta_1, 1_F \rangle} \ar[dr] & &
      \intv \times F \ar[dl] \\
      & F &
    }
  \end{equation*}
\end{example}

\begin{example}
  Again working in CCHM cubical sets, we define a \emph{weak
    fibration} to be a map with the right lifting property against the
  following family of squares. The intuition is that we define a
  weaker notion of Kan filling operator in which instead of requiring
  a diagonal filler for all lifting problems of
  $m \hat \times \delta_0$ against a map $f \colon X \rightarrow Y$,
  we only require it for those where the map
  $\intv \times \Sigma \rightarrow Y$ factors (necessarily uniquely)
  through the projection $\intv \times \Sigma \rightarrow \Sigma$.
  \begin{equation*}
    \xymatrix{ \intv +_1 F \ar@{=}[rr] \ar[d]_{\top \hat \times
        \delta_0} & & \intv +_1 F \ar[d] \\
      \intv \times F \ar[rr]^{\pi_1} \ar[dr] & & F \ar[dl] \\
      & F &
    }
  \end{equation*}
  Since the lower right object is terminal in $\catc / \Sigma$, we see
  that the cofibrantly generated lawfs is strongly fibred, and the
  cofibrantly generated awfs is too, if it exists. This construction
  may be useful for developing an abstract version of the approach to
  the semantics of higher inductive types in \cite{coquandcubicaltt}.

  We can also combine this with the previous example to get a weak
  version of Kan composition operator:
  \begin{equation*}
    \xymatrix{ 1 \ar[rr] \ar[d]_{\top}
      & & \intv +_1 F \ar[d] \\
      F \ar[rr] \ar[dr]
      & & F \ar[dl] \\
      & F &
    }
  \end{equation*}
\end{example}

\section{Conclusion and Directions for Future Work}
\label{sec:concl-direct-future}

\subsection{Cofibrantly Generated Awfs's in $\Pi W$-Pretoposes with
  WISC}
\label{sec:cofibr-gener-awfss}

In this paper we saw a new fibred variation of the definition of
cofibrantly generated awfs. We also saw that Garner's small object
argument tells us that for certain categories, cofibrantly generated
awfs's over the category indexed families fibration always
exist. However, so far we have not seen any corresponding result for
cofibrantly generated awfs's over codomain fibrations. One possible
way to approach this would be to carry out a transfinite construction
similar to Garner's small object argument. However, this approach has
the drawback that it requires infinite colimits. This results in
natural examples (such as internal presheaves in realizability
toposes) where one can define step-one of the small object
argument, but where the construction of cofibrantly generated awfs's
does not work. The author will instead develop a new, alternative
approach in a separate paper. Roughly speaking the idea is as
follows. A category has $W$-types if certain endofunctors, referred to
as \emph{polynomial endofunctors} admit initial algebras
\cite{moerdijkpalmgrenwtypes}. The author will develop a new
generalisation of $W$-types in which one instead uses initial algebras
of certain pointed endofunctors, and that these initial algebras can
be constructed from $W$-types provided that a weak choice axiom known
as weakly initial set of covers (WISC) holds. For codomain fibrations,
the pointed endofunctor in corollary \ref{cor:initialalg} will be an
example of such a pointed endofunctor, and so we will deduce that
cofibrantly generated awfs's over the codomain fibration exist in this
case.

\subsection{Applications to Realizability}
\label{sec:appl-real}

One of the main aims of this work was to develop a definition of
cofibrantly generated that is suitable for use in realisability
toposes, categories of assemblies and variants, which at the same time
can be used in proofs that are easy generalisations of existing work
in homotopical algebra. The main issue is that these categories are
not cocomplete. This makes it difficult to apply some standard
arguments in homotopical algebra, such as the small object argument.

A natural way to develop a realizability variant of CCHM cubical sets
is to construct the category of cubical sets internally in a category
of assemblies. Within our general framework we have now seen two
different approaches to defining classes of maps within these
categories. The most promising approach is to work over the codomain
fibration as in section \ref{sec:class-class-monom}. However, the
approach of internal category indexed families of presheaves from
section \ref{sec:presheaf-assemblies} is in some ways more flexible
and may also be useful, for instance when working with realizability
variants of BCH cubical sets, where separated product is not fibred
with respect to the codomain fibration.

A somewhat surprising fact, first observed by Van Oosten in
\cite{vanoostenhtyeff} is that the effective topos itself admits
nontrivial homotopical structure. We have shown (in example
\ref{ex:efftopos}) that a more recent variant by Van den Berg and
Frumin fits within our general framework. Another realizability topos
that promises to have rich (but apparently as yet unexplored)
homotopical structure is the function realizability topos, and its
relative the Kleene-Vesley topos (as defined in \cite[Section
4.5]{vanoosten}).

\subsection{The BCH Cubical Set Model}
\label{sec:bch-cubical-set}

The Bezem-Coquand-Huber cubical set model \cite{bchcubicalsets} was
the first example of a constructively valid ``homotopical'' model of
type theory. Since then, many authors have focused on the newer
Cohen-Coquand-Huber-M\"{o}rtburg cubical set model. However, the
original BCH model remains an interesting topic. Bezem, Coquand and
Huber have shown in a more recent paper \cite{bchunivalence} that the
univalence axiom holds in BCH cubical sets, confirming that this
approach does indeed give a model of homotopy type theory.

We have
seen here in section \ref{sec:triv-fibr-01} that the acyclic
fibrations appearing in that paper can be characterised elegantly as
cofibrantly generated with respect to the codomain fibration (for the
equivalent category of $01$-substitution sets).  Kan fibrations in BCH
cubical sets are part of an awfs cofibrantly generated with respect to
the category indexed family fibration, and moreover this can be done
constructively \cite{swannomawfs}, but the question remains whether
there is a more elegant definition similar to the case for acyclic
fibrations, or Kan fibrations in CCHM cubical sets.

The main obstacle
is that the definition used by Bezem, Coquand and Huber requires a
particular monoidal product, called \emph{separated product} to
state. Since it is unclear how to extend separated product to a fibred
monoidal product over $\cod \colon \sub^\btwo \rightarrow \sub$, we
cannot readily use the fibred Leibniz construction to get the BCH
definition of Kan fibration. A partial answer to this has been
provided by Alex Simpson, who has shown (in currently unpublished
work) that if instead of $\sub^\btwo$, one works in a suitable
subcategory (of so called \emph{independent squares}), and the
restriction of $\cod$ to this subcategory, then one does obtain a
cartesian monoidal fibration whose restriction to the fibre over the
terminal object is the monoidal category of $\sub$ with separated
product.

\subsection{Relation to Other Generalised Notions of Lifting Problem}
\label{sec:combining-with-other}

In this paper we have seen a generalised notion of cofibrantly
generated awfs's. However, it is not the only such generalisation.

For
example, as Riehl explains in \cite[Section 13.3]{riehlcht}, given a
monoidal category, one can define a notion of enriched lifting
property. There is an overlap between the examples considered here and
enriched lifting properties. Namely, the fibred lifting property
between maps over the terminal object in a codomain fibration can also
be viewed as the enriched lifting property (over the cartesian
monoidal product) between functors from the trivial enriched
category. However, neither approach seems to be more general than the
other. One can give a rough intuition for the relation between fibred
and enriched lifting problems as follows. In an enriched category, for
any two objects $X$ and $Y$, $\hom(X, Y)$ is an object of a certain
category and so can be manipulated via the internal logic of that
category. However, to talk about a collection of objects, or maps with
different domains/codomains we still need to use some external notion
of set. On the other hand, when working over a fibration, even more
can be done internally, requiring very little from the ambient set
theory. It is natural to ask whether it is possible to combine the
definitions together to get something that subsumes both notions. Such
a combination would likely involve Shulman's work on monoidal
fibrations from \cite{shulmanfbmf}.

In \cite[Section 6]{bourkegarnerawfs1}, Bourke and Garner consider a
notion of lifting problem between double categories. Once again, it
seems that this is neither more general, nor a special case of the
framework we consider here. Therefore it's natural to ask whether
there is another more general notion that includes both.

\bibliographystyle{abbrv}
\bibliography{mybib}{}

\end{document}